\DocumentMetadata{
    lang=en-GB,
    tagging=on,
}
\documentclass{article}

\usepackage{geometry}
\usepackage[british]{babel}

\usepackage{enumitem}

\usepackage{booktabs, array}

\usepackage[svgnames,dvipsnames]{xcolor}

\usepackage{graphicx}
\graphicspath{ {./figures/} }
\usepackage{pdflscape}
\usepackage{caption}
\usepackage{tikz}
\usetikzlibrary{calc}
\usetikzlibrary{patterns,patterns.meta}
\usetikzlibrary{decorations.markings}
\usepackage{pgfplots}
    \pgfplotsset{
        compat=1.18,
        table/search path={./figures/}
    }

\usepackage{etoolbox}  

\usepackage{amsmath}
\usepackage{amssymb}
\usepackage{prodint}
\usepackage{siunitx}

\usepackage{hyperref}
\hypersetup{
	pdftitle={Survival Isotonic Distributional Regression (S-IDR)},
	pdfauthor={Martin Bladt, Alexander Henzi, Bram van den Heuvel, Johanna Ziegel},
	colorlinks=true,
	linkcolor=MidnightBlue,
	citecolor=MidnightBlue,
	urlcolor=MidnightBlue
}
\usepackage{doi}

\usepackage{natbib}
\usepackage{algpseudocode}
\usepackage{algorithm}

\usepackage[nameinlink,capitalize,noabbrev]{cleveref} 
\usepackage{amsthm}
\usepackage{thmtools}
\declaretheorem[name=Theorem]{theorem}
\declaretheorem[name=Proposition,sibling=theorem,refname={Proposition,Propositions}]{proposition}
\declaretheorem[name=Lemma,sibling=theorem,refname={Lemma,Lemmas}]{lemma}
\declaretheorem[name=Corollary,sibling=theorem,refname={Corollary,Corollaries}]{corollary}
\declaretheorem[name=Condition,refname={Condition,Conditions}]{condition}
\declaretheorem[name=Condition,numberwithin=condition,refname={Condition,Conditions}]{subcondition}
\theoremstyle{definition}
\declaretheorem[name=Definition]{definition}

\newtheorem{remark}{Remark}
\crefname{appendix}{Appendix}{Appendices}
\Crefname{appendix}{Appendix}{Appendices}

\usepackage{mathbbol}
\DeclareSymbolFontAlphabet{\amsmathbb}{AMSb}%
\newcommand{\pr}{{\amsmathbb{P}}}
\newcommand*\dd{\mathop{}\!\mathrm{d}}
\newcommand\independent{\protect\mathpalette{\protect\independenT}{\perp}}
\def\independenT#1#2{\mathrel{\rlap{$#1#2$}\mkern2mu{#1#2}}}
\DeclareMathOperator{\clamp}{clamp}
\DeclareSymbolFont{bbold}{U}{bbold}{m}{n}
\DeclareMathSymbol{\bboldone}{\mathord}{bbold}{`1}
\newcommand{\ind}[1]{\bboldone_{\{#1\}}}

\newcommand{\KM}[1]{\hat{\amsmathbb{F}}_{#1}}
\newcommand{\RKM}[1]{\widetilde{\amsmathbb{F}}_{#1}}
\newcommand{\R}{\amsmathbb{R}}
\newcommand{\Rp}{\amsmathbb{R}_+}
\newcommand{\Ucdf}[2]{U_{[#1,#2]}(y)}
\newcommand{\leso}{\preceq_{\mathrm{st}}}
\newcommand{\lehro}{\preceq_{\mathrm{hro}}}

\renewcommand{\subset}{\subseteq}

\usepackage{xspace}
\defcitealias{barmiInferencesStochasticOrdering2005}{EBM}
\newcommand{\EBM}[1][]{%
  \citetalias{barmiInferencesStochasticOrdering2005}%
  \ifx\\#1\\\else\ [#1]\fi
  \xspace
}
\defcitealias{parkPointwiseNonparametricMaximum2012}{PRK}
\newcommand{\PRK}[1][]{%
  \citetalias{parkPointwiseNonparametricMaximum2012}%
  \ifx\\#1\\\else\ [#1]\fi
  \xspace
}

\DeclareRobustCommand{\mx}{\tikz[baseline=-0.6ex]{\draw[thick] (-0.6ex,-0.6ex) -- (0.6ex,0.6ex) (-0.6ex,0.6ex) -- (0.6ex,-0.6ex);}\xspace}
\DeclareRobustCommand{\mo}{\tikz[baseline=-0.6ex]{\draw[thick] (0,0) circle (0.6ex);}\xspace}
\DeclareRobustCommand{\ms}{\tikz[baseline=-0.6ex]{\draw[thick] (-0.6ex,-0.6ex) rectangle (0.6ex,0.6ex);}\xspace}

\usepackage[affil-it]{authblk}

\title{Survival Isotonic Distributional Regression (S-IDR)}
\author[1]{Martin Bladt\thanks{\href{mailto:martinbladt@math.ku.dk}{martinbladt@math.ku.dk}}}
\author[2]{Alexander Henzi\thanks{\href{mailto:henzia@tsinghua.edu.cn}{henzia@tsinghua.edu.cn}}}
\author[3]{Bram van den Heuvel\thanks{\href{mailto:bram.vandenheuvel@stat.math.ethz.ch}{bram.vandenheuvel@stat.math.ethz.ch}}}
\author[3]{Johanna Ziegel\thanks{\href{mailto:ziegel@stat.math.ethz.ch}{ziegel@stat.math.ethz.ch}}}
\affil[1]{Department of Mathematical Sciences, University of Copenhagen}
\affil[2]{Department of Statistics and Data Science, Tsinghua University}
\affil[3]{Seminar for Statistics, ETH Z\"urich}
\date{\today}

\begin{document}
\maketitle

\begin{abstract}
    We introduce Survival-IDR (S-IDR), a nonparametric estimator of conditional survival distributions under order restrictions, extending Isotonic Distributional Regression (IDR; \citep{henziIsotonicDistributionalRegression2021}) to right-censored outcomes. S-IDR has no tuning parameters and accommodates continuous, discrete, and partially ordered covariates. We first study the direct Kaplan--Meier adaptation of IDR: it is uniformly consistent at the minimax rate, but only when the conditional outcomes are hazard-rate ordered. We trace this restriction to the Kaplan--Meier estimator's failure to satisfy the Cauchy mean value property on non-i.i.d.\ samples, and use the diagnosis to construct S-IDR. The S-IDR estimator is uniformly consistent under only stochastic dominance of the conditional outcomes, attains the minimax rate when the smoothness of the conditional CDFs is known, and admits a known cross-threshold PAVA acceleration. We further embed S-IDR in a distributional single-index framework on a benchmark suite, and apply it in a case study that validates the MELD score used for liver-transplant wait list management. Accompanying R, Python and Rust packages are available at \url{https://github.com/AlexanderHenzi/isodistrreg}.
\end{abstract}

\setcounter{tocdepth}{2}
\tableofcontents

\section{Introduction}
When modelling an outcome from covariates --- say, predicting short-term mortality from liver failure given a patient's lab measurements --- a practitioner typically estimates the response’s conditional mean. Complete conditional distributions of the response given the covariate are more informative, but inherently harder to estimate. In survival analysis, the response is commonly time-to-event data and (partially) right-censored. Despite this additional difficulty, distributional targets in the form of survival curves are common, enabled by simplifying and possibly restrictive modelling assumptions. In the case of liver donation urgency, such a distributional target could represent a simultaneous estimate for the probability of death due to liver failure on each day in the next three months. Often, the distributional regression estimate should adhere to a certain directionality: when some clinical lab measurements are higher, like INR (the \emph{International Normalized Ratio} measures blood clotting factors impacted by liver function decline), then predicted survival times should be lower (or at least not higher). This work focuses on distributional survival regression under such order restrictions and proposes the first nonparametric estimator that can work with continuous covariates.

Distributional regression dates back to the nineteenth century, when several conditional quantiles were modelled jointly to characterise dispersion around the centre of the conditional distribution \citep{Galton1889,parzenQuantileProbabilityStatistical2004}. With growing computational power, simultaneous estimation of many quantiles became feasible \citep{koenkerRegressionQuantiles1978,koenkerComputingRegressionQuantiles1987}, effectively approximating the full conditional distribution. However, a quantile-based model represents conditional distributions only if the quantiles do not cross \citep{bassettjr.EmpiricalQuantileFunction1982}. Distributional regression methods estimate conditional distributions directly, from which functionals like the mean or median can be derived, see \citet{kneibRageMeanReview2023, kleinDistributionalRegressionData2024} for reviews. Distributional estimates should be calibrated to accurately reflect uncertainty while being as informative as possible \citep{gneitingProbabilisticForecastsCalibration2007}.

Isotonic regression is a nonparametric method which imposes order restrictions and is widely used in practical modelling. For example, mechanical failures tend to occur earlier if the quality of the materials used is lower, and firms with poorer credit rating are more likely to default in the next year. Order restrictions have been elaborately investigated in testing and mean regression \citep{barlowStatisticalInferenceOrder1972,robertsonOrderRestrictedStatistical1988}, but less so in the setting of distributional regression, where the conditional distributions are constrained to be ordered \citetext{\citealp{lehmannOrderedFamiliesDistributions1955}; \citealp[Chapter~8]{rossStochasticProcesses1983}}. Consistent quantiles in isotonic regression have previously been studied in \citet{moschingMonotoneLeastSquares2020} and in the Isotonic Distributional Regression (IDR) estimator of \citet{henziIsotonicDistributionalRegression2021}, which we generalise in this work. Isotonic regression can be considered a generalisation of other models whose quantiles move in one direction with the covariate(s) (e.g., Cox PH, AFT), but with the direction fixed in advance. When this does not hold, the distributional methods proposed in this work can be combined with general survival regression techniques to provide consistent distributional estimates using an index model approach as taken by \citet{henziDistributionalSingleIndex2023, walzEasyUncertaintyQuantification2024, balabdaouiEstimationConvergenceRates2024}.

Survival modelling involves time-to-event data, which may represent death due to liver failure, contraction or recurrence of a disease, or mechanical failure, among many others \citep{linSurvivalAnalysis2014}. Often, a ``nuisance'' censoring mechanism is involved, obscuring the exact value of the variable of interest and revealing only that it falls within a (possibly open-ended) interval. For example, a patient may die from a different cause than is being studied, a machine part might be replaced before a mechanical failure occurs, or a firm could be acquired, making it impossible to observe a default that may have occurred otherwise. Censoring may also be present in less traditional settings, such as in insurance claims data, which could be censored because only the claim amount and the policy limit are recorded, while the variable of interest is the total loss amount before policy limits \citep{charpentierComputationalActuarialScience2014}. Survival analysis can incorporate such censoring, without which estimation would be biased.
Classic parametric models for time-to-event data include the Cox model \citep{coxRegressionModelsLifeTables1972}, Tobit \citep{breenRegressionModelsCensored1996}, accelerated failure time (AFT) \citep{buckleyLinearRegressionCensored1979}, and censored linear quantiles \citep{portnoyCensoredRegressionQuantiles2003}. These models impose strong assumptions on the response distribution and its dependence on covariates --- ordered or even proportional hazards, for example --- which can be reasonable in specific applications such as AFT in \citet{stroustrupTemporalScalingCaenorhabditis2016}, but are often violated, particularly in medical studies \citep{stensrudWhyTestProportional2020}. Semi-parametric approaches such as generalised linear models (GLM) \citep{nelderGeneralizedLinearModels1972}, generalised additive models~(GAM) \citep{hastieGeneralizedAdditiveModels2017}, and extensions like GAMs for location, scale and shape~(GAMLSS) \citep{stasinopoulosGeneralizedAdditiveModels2008} relax these constraints by allowing a flexible parametric family whose parameters vary with covariates; they can also accommodate general censoring \citep{stasinopoulosGeneralizedAdditiveModels2008}. Nonparametric models go further by making minimal distributional assumptions, with model complexity often growing with the data size; examples include random forests \citep{breimanRandomForests2001}, later adapted to distributional/quantile regression \citep{meinshausenQuantileRegressionForests2006,cevidDistributionalRandomForests2022} and right-censoring via random survival forests \citep{ishwaranRandomSurvivalForests2008}. Deep learning methods have also been proposed for time-to-event data, though many incorporate parametric components \citep{wiegrebeDeepLearningSurvival2024}. For distributional regression with right-censoring \emph{under order restrictions} specifically, prior nonparametric estimators of \citet{barmiInferencesStochasticOrdering2005} and \citet{parkPointwiseNonparametricMaximum2012} require ordinal covariates.
Survival models are regularly used to compare risk, as in the liver transplant application. Most provide a ``risk score'', based either on the survival probability up to a specific time or on some other (scalar) summary of the estimated survival distribution. For example, in the Cox proportional hazard model, this score is simply the inner product of parameter and covariate, while some nonparametric methods like random survival forests \citep{ishwaranRandomSurvivalForests2008} also provide such a score.

This work lies at the intersection of distributional regression, survival modelling, and order constraints. We propose the first nonparametric estimators of conditional survival distributions under order restrictions that handle continuous, discrete, and mixed covariates. Our two main contributions are as follows. The first is the \emph{plain survival IDR}, a minimal adaptation of IDR that substitutes the Kaplan--Meier estimator \citep{kaplanNonparametricEstimationIncomplete1958}. It is parameter-free, requires only that the censoring be non-informative, and --- under hazard rate ordering of the outcome --- is uniformly consistent at the minimax optimal rate $n^{-\alpha/(1+2\alpha)}$, up to a logarithmic factor, adapting to the unknown smoothness $\alpha$ of the conditional distributions in the covariate. The second is the \emph{Survival-IDR} (S-IDR), which lifts the hazard rate assumption by repairing a violation of the Cauchy mean value property in the Kaplan--Meier estimator. S-IDR is uniformly consistent under only stochastic dominance of the outcome, attains the minimax rate when the smoothness is known with additional bucketing, allows an accelerated algorithmic implementation, and has an equivalent formulation in quantile space.
Both estimators handle covariates and outcomes that are discrete, continuous, or mixed, and perform well empirically in terms of convergence rate and calibration. They may be applied to risk scores of an existing, carefully vetted model to produce a survival distribution that respects those scores through ordered quantiles and survival probabilities. In the organ donation setting, this might reveal, for example, that two patients with similar risk scores have identical short-term survival probabilities and diverge only at later times --- suggesting that characteristics such as match quality or geographical distance should be brought in to break the tie.

To illustrate the proposed model, we provide in \cref{fig:illustration} an example analogous to \citet[Fig.~1]{henziIsotonicDistributionalRegression2021}, using our S-IDR method. Consider a covariate $X$ that is uniformly distributed on $(0,10)$ and, conditional on $X$, let the outcome $Y$ have a gamma distribution, specifically,
\begin{equation}\label{eq:illustration}
    Y \mid X \sim \text{Gamma}\left(\text{shape} = \sqrt{X}, \text{ scale} = \clamp(X, 1, 6)\right),
\end{equation}
where the $\clamp$ function (sometimes referred to as $\operatorname{clip}$ function) is defined as $\clamp(x, l, u) := \min \{ \max \{ x, l\}, u\}$ for $l \le u$.
The conditional distributions satisfy the isotonicity assumption of ordered quantiles and ordered hazard ratios (and even have ordered likelihood ratios). For the conditional distribution $C \mid X$ of the censoring variable $C$, we choose the same distribution as $Y \mid X$ conditional on $C < \infty$, and $\pr(C = \infty) = 1/2$. We let $C$ be independent of $Y$ given $X$. \Cref{fig:illustration} shows that the S-IDR fit closely tracks the true conditional CDFs despite about 25\% censored observations, while the censoring-ignorant IDR estimate is visibly biased upward.

\begin{figure}[h]
    \centering
    \includegraphics[
        width=\textwidth,
        alt={Top panel shows for 5 covariate levels X a smooth ground-truth CDF, the not-corrected-for-censoring IDR estimate, and the corrected-for-censoring S-IDR estimate. Bottom panel shows a scatter plot of censored and observed X / T points, with the center of the estimated distribution highlighted.}
    ]{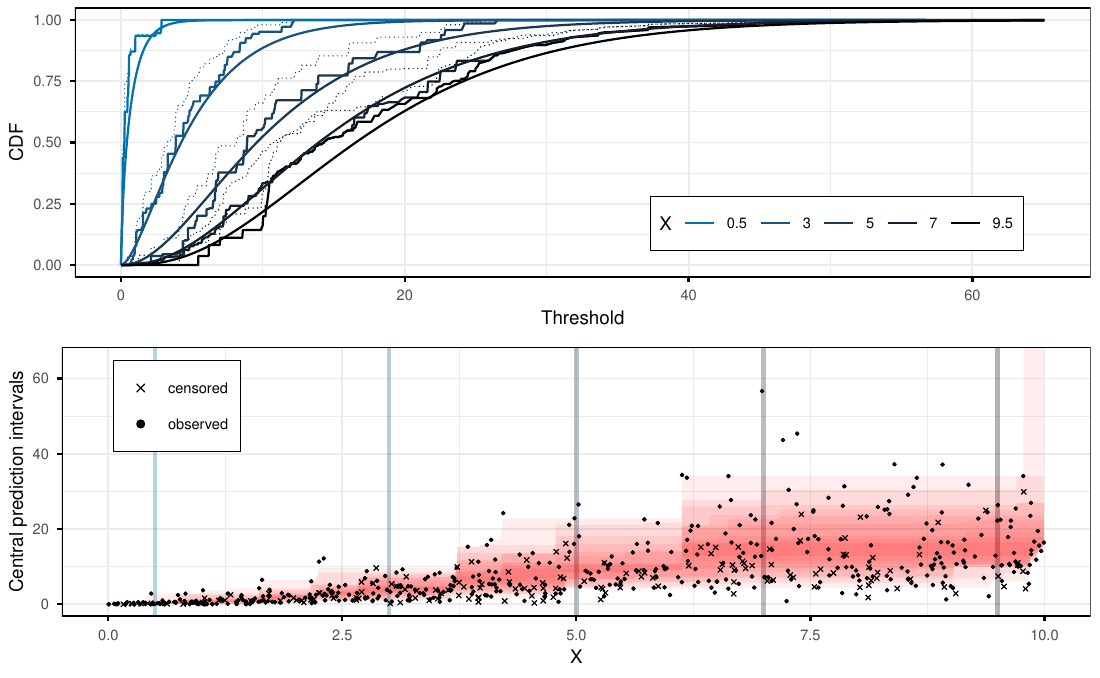}
    \caption{
        Simulation with $n = 600$ partially censored observations from \cref{eq:illustration}. Top: S-IDR sub-CDF estimates (step functions) and the smooth true CDFs they converge to; dotted lines show upward-biased IDR estimates that ignore censoring. Bottom: the (partly censored) data used for S-IDR, with red bands marking the centre of the estimated conditional distribution of $Y$ (more saturated means closer to the centre). Compared to \citet[Fig.~1]{henziIsotonicDistributionalRegression2021}, fit quality is only slightly reduced despite about $25\%$ censored observations.
    }
    \label{fig:illustration}
\end{figure}

The article is organised as follows. \Cref{s:prerequisites} introduces notation and reviews IDR, which \cref{s:sidr_plain} generalises to the survival setting. \Cref{s:sidr} examines the assumptions required for consistency and develops S-IDR, and \cref{s:extensions} extends the construction to multi-dimensional covariates. \Cref{s:empirical_results} then presents simulations in which we compare against existing methods, performance within a distributional index model framework, and a case study on organ donation wait list management. \Cref{s:discussion} discusses open problems. All proofs are deferred to the appendix. Software packages for Rust, R and Python are provided at \url{https://github.com/AlexanderHenzi/isodistrreg}.

\section{Notation and prerequisites}\label{s:prerequisites}
All random variables are defined on a probability space with probability measure $\pr$. Let $X$ be a random covariate taking values in an interval $\mathbf{I} \subset \R$ and $Y \geq 0$ a response variable. Extensions to general partially ordered covariate spaces, similar to \citet{henziIsotonicDistributionalRegression2021}, are discussed in \cref{s:extensions} and \cref{aa:partial_orders}. We want to model the conditional distribution of outcome $Y$ given $X$, so we define the conditional cumulative distribution functions (CDFs)
\[
	F_x(y) := \pr(Y \leq y \mid X = x),
\]
the target of our modelling. Because $Y$ is right-censored by a random variable $C$, called the censoring variable, taking values in $\bar{\R}_+ := \Rp \cup \{ \infty \}$ with $\Rp = [0,\infty)$, instead of $Y$ we observe
\[
    T := \min(Y, C), \qquad \Delta := \ind{Y \leq C}.
\]

Let $(X_i, Y_i, C_i)$, $i = 1, \dots, n$, be independent copies of $(X, Y, C)$, with corresponding $T_i$ and $\Delta_i$. Our estimators use data modelled as realisations of the observables $((X_i, T_i, \Delta_i))_{i=1}^n$.

We require that the censoring is noninformative, which is needed for $F_x$ to be identifiable \citep{ebrahimiIdentifiabilityCensoredData2003,overgaardAssumptionIndependentRight2021}.
\begin{condition}[Noninformative censoring]\label{cond:noninformative_censoring}
    The variables $Y$ and $C$ are conditionally independent given $X$, that is,
    \[
        Y \independent C \mid X.
    \]
\end{condition}
This assumption can fail in practice: for example, if patients leave a longevity study more often as their health worsens, a time-to-death estimate may be biased \citep{jacksonRelaxingIndependentCensoring2014}. Existing work that relaxes it does so only in specific parametric settings \citep{jacksonRelaxingIndependentCensoring2014,sunTestingIndependentCensoring2011}; we adopt the assumption, as is standard in survival analysis.

We now proceed to formalise the monotonicity assumption we impose on the conditional CDFs $F_x$. Two strengths will appear: the weaker stochastic ordering condition stated next, and a stronger hazard-rate variant introduced in \cref{ss:sidr_plain}. We number them \cref{cond:stochastic_order,cond:hazard_rate_order} to make the comparison explicit.
\setcounter{condition}{2} 
\begin{subcondition}[Conditional stochastic ordering of $Y$]\label{cond:stochastic_order}
    Assume that the conditional law of the outcome variable $Y$ given $X = x$ is increasing in stochastic order (the usual stochastic order, also referred to as first-order stochastic dominance) with respect to $x \in \mathbf{I}$, that is
    \[
        \forall x, x' \in \mathbf{I}: \quad x \leq x' \quad \Longrightarrow \quad F_x(y) \geq F_{x'}(y) \quad \forall y \in \R.
    \]
    When $F_x$ is smaller or equal to $F_{x'}$ in the stochastic order, we denote this by $F_x \leso F_{x'}$.\end{subcondition}
Informally, and in the parlance of survival analysis, higher values of the covariate shift the distribution to the right, so by any given time the probability of the event having occurred is no larger and the probability of still surviving is no smaller. \Cref{cond:stochastic_order} is reasonable in many practical settings; for example, at any fixed time since diagnosis, a patient is less likely to survive the further the disease had progressed at the time of diagnosis (the covariate).

The methods proposed in this paper provide an estimate for the family $(F_x)_{x \in \mathbf{I}}$ by computing a collection of conditional CDFs on the grid of observed covariate values $X_1, \ldots, X_n$. We extend to the full interval $\mathbf{I}$ by (e.g., linear) interpolation between the closest observed covariate values. The result is a collection of step-function CDFs, whose quantiles do not cross.

For convenience and without loss of generality, we relabel our data triples $(X_i, T_i, \Delta_i)$ such that $X_1 \leq \dots \leq X_n$. Moreover, we denote the $m \leq n$ unique covariate values of the observed $X_i$ by $\xi_1 < \ldots < \xi_m$. If $1 \leq r \leq s \leq m$, we write $[r:s]$ for the integer interval $r, \ldots, s$ and denote the set of indices of observations with a covariate value in the interval $[\xi_r, \xi_s]$ by $O_{r:s} := \{ i \mid \xi_r \leq X_i \leq \xi_s, 1 \leq i \leq n \}$. We estimate conditional CDFs $\hat{F}_x$ for each $x \in \{\xi_1, \ldots \xi_m\}$, and we denote joint estimates by $\mathbf{\hat{F}} = (\hat{F}_{\xi_1}, \ldots, \hat{F}_{\xi_m})$.

For uncensored data, the IDR estimator $(\hat{F}_x)_{x \in \mathbf{I}}$ is defined as the linear interpolation (in the covariate direction) of the minimiser $\mathbf{\hat{F}}$ of an empirical risk $\ell_S$, defined through a proper scoring rule $S$ \citep{waghmareProperScoringRules2026}:
\[
    \ell_S(\mathbf{F}) := \frac{1}{n} \sum_{i=1}^n S(F_{X_i}, Y_i),
\]
subject to \cref{cond:stochastic_order}, here equivalent to $F_{\xi_r} \leso F_{\xi_s}$ for $1 \leq r \leq s \leq m$. \citet{henziIsotonicDistributionalRegression2021} choose for $S$ the continuous ranked probability score ($\operatorname{CRPS}$), defined as
\[
    \operatorname{CRPS}(F, Y) := \int_\R (F(y) - \ind{Y \leq y})^2 dy.
\]
Minimisation can be performed for each threshold $y \in \R$ separately, which amounts to antitonic regression (decreasing monotonic regression) for the probability $F_x(y)$ with a squared penalty. This lets \citet{henziIsotonicDistributionalRegression2021} arrive at a classic $\min$-$\max$ formula for IDR,
\begin{equation}\label{eq:idr_min_max}
    \hat{F}_{\xi_i}(y) = \min_{1 \leq r \leq i} \, \max_{i \leq s \leq m} \frac{1}{|O_{r:s}|} \sum_{j \in O_{r:s}} \ind{Y_j \leq y}, \qquad y \in \R, \, 1 \le i \le m,
\end{equation}
compare \citep{barlowIsotonicRegressionProblem1972}.
\citet{henziIsotonicDistributionalRegression2021} proceed to demonstrate how a wide class of proper scoring rules $S$ all arrive at the same result. The inner quantities $|O_{r:s}|^{-1} \sum_{j \in O_{r:s}} \ind{Y_j \leq y}$ are the empirical CDFs of the samples $\{ Y_i \mid i \in O_{r:s} \}$. This formulation at \eqref{eq:idr_min_max} is our starting point for deriving the Survival-IDR (S-IDR) estimator.

A \emph{sub-CDF} is a right-continuous increasing function with vanishing left limit and right limit at most $1$. Under right-censoring, the survival-analytic counterpart of the empirical CDF, the Kaplan--Meier estimator (defined in \cref{ss:sidr_plain}), is in general only a sub-CDF rather than a CDF. Accordingly, S-IDR and existing alternatives target tuples of sub-CDFs rather than full CDFs.

We will compare against two related isotonic distributional regression methods. They can both accommodate right-censoring and provide estimates that satisfy the conditional stochastic ordering \cref{cond:stochastic_order}, but require the covariate to be discrete. The first is the estimator from \citet{barmiInferencesStochasticOrdering2005} (hereafter \EBM), defined as
\begin{equation}\label{eq:el_barmi}
    \hat{F}_{\xi_i}(y) = \min_{1 \leq r \leq i} \, \max_{i \leq s \leq m} \frac{1}{|O_{r:s}|} \sum_{j=r}^s |O_{j:j}| \, \KM{j:j}(y), \qquad \forall y \in \Rp, \, 1 \le i \le m,
\end{equation}
the threshold-wise (weighted) isotonic regression of the $m$ Kaplan--Meier estimators $\KM{j:j}$ computed on respectively $\{(T_i, \Delta_i) \mid i \in O_{j:j}\}$, defined in \cref{ss:sidr_plain}. The second is the pointwise nonparametric maximum likelihood estimator by \citet{parkPointwiseNonparametricMaximum2012} (hereafter \PRK). Their profile likelihood approach leads to an estimator that applies a generalised pool-adjacent violators (PAV) algorithm for separable functions: for each $y \in \Rp$, it is defined as the global maximum likelihood solution evaluated at $y$. Both estimators, as well as IDR, coincide when evaluated on \emph{uncensored} observations with a discrete covariate. With right-censored data, the estimators from \EBM and \PRK coincide for those groups that are not merged with any others, in which case the estimate for that group is simply the Kaplan--Meier estimator.

\section{Warm-up: A direct estimator}\label{s:sidr_plain}
\subsection{An IDR-inspired estimator with uniform consistency}\label{ss:sidr_plain}
For IDR, the inner quantities that the min-max at \eqref{eq:idr_min_max} is taken over are the empirical CDFs of the response observations that belong to covariates in a certain range. In this section, we simply replace the empirical CDF (evaluated on the relevant $Y_i$) with the Kaplan--Meier estimator (evaluated on the relevant $(T_i, \Delta_i)$) and investigate the properties and shortcomings of the resulting estimator. The findings motivate the introduction of S-IDR.

For $1 \le r \le s \le m$, $y \in \Rp$, we denote by $\KM{r:s}(y)$ the Kaplan--Meier estimator evaluated on $(T_i, \Delta_i)$, $i \in O_{r:s}$, at threshold $y$, that is,
\begin{equation}\label{eq:kaplan_meier}
    1 - \KM{r:s}(y) = \prod_{t \in \{T_i \mid i \in O_{r:s}, T_i \leq y\}} \left(1 - \frac{\#\{i \in O_{r:s} \mid T_i = t, \Delta_i = 1\}}{\#\{i \in O_{r:s} \mid T_i \geq t\}}\right).
\end{equation}
This definition can be extended to observations with nonnegative weights $w_i$, where instead of counting indices, the sum of the weights enters into \eqref{eq:kaplan_meier}.

\begin{definition}[Plain survival IDR]\label{def:sidr_plain}
    The \emph{plain survival IDR} is defined for $y \in \Rp$ as
    \begin{equation}\label{eq:sidr_plain}
        \hat{F}_{\xi_i}(y) := \min_{1 \leq r \leq i} \, \max_{i \leq s \leq m} \KM{r:s}(y),
    \end{equation}
    at the observed covariate values $\xi_i$, $1 \leq i \leq m$. For $x \in (\xi_1, \xi_m)$, we set $\hat{F}_x(y)$ as the linear interpolation between the nearest observed covariate values, otherwise $\hat{F}_x(y) := \hat{F}_{\clamp(x, \xi_1, \xi_m)}(y)$.
\end{definition}

The estimators $\hat{F}_{\xi_i}$ are indeed conditional sub-CDFs which satisfy the stochastic order assumption. This follows because the Kaplan--Meier estimators are sub-CDFs and because the minimum in the definition of $\hat{F}_{\xi_i}(y)$ is taken over a larger set and the maximum over a smaller set as the index $i$ increases.

Before studying this definition in more detail, we compare it with the \EBM estimator at \eqref{eq:el_barmi}, which weights individual Kaplan--Meier estimators but does not use the Kaplan--Meier estimator over a combined sample of multiple covariate values. As we will see later, the Kaplan--Meier estimator is generally not equal to a weighted average of the same estimator evaluated on sub-samples, so these estimators differ. \EBM is consistent only when $X$ is ordinal (i.e., discrete with finite, ordered support). When all observations are uncensored, both \eqref{eq:el_barmi} and \eqref{eq:sidr_plain} coincide with \eqref{eq:idr_min_max}.

To state the uniform consistency result for plain survival IDR, we need an additional stochastic ordering condition.
\begin{subcondition}[Conditional hazard rate ordering of $Y$]\label{cond:hazard_rate_order}
    The outcome variable $Y$ is \emph{increasing in hazard rate order (HRO)} \citep[Section~1.B]{shakedStochasticOrders2007} with respect to $x \in \mathbf{I}$ if the CDFs $F_x$ admit a Lebesgue density $f_x$ and
    \[
        \forall x, x' \in \mathbf{I}: \quad x \leq x' \Longrightarrow \frac{f_x(y)}{1 - F_x(y)} \geq \frac{f_{x'}(y)}{1 - F_{x'}(y)} \qquad \forall y \in \R.
    \]
    When $F_x$ is smaller or equal to $F_{x'}$ in the hazard rate order, we denote this by $F_x \lehro F_{x'}$.
\end{subcondition}
Hazard rate ordering implies the more common stochastic order in \cref{cond:stochastic_order}, is what we need for the plain survival IDR to be uniformly consistent, and is implied by the likelihood ratio order \citep[Section~1.C]{shakedStochasticOrders2007}.

We impose the following assumptions. We restrict thresholds $y$ to an interval in which there is positive probability of an uncensored observation, avoiding regions where the data is no longer informative about $F_x$.
\begin{condition}[Positive probability of uncensored observation]\label{cond:positive_probability_uncensored}
    There exist $\tau \in \Rp$ and $\eta > 0$ such that for all $x \in \mathbf{I}$,
	\[
		\pr(T \leq \tau \mid X = x) \leq 1 - \eta.
	\]
\end{condition}

The following two conditions correspond to assumptions (A.1) and (A.2) of \citet{moschingMonotoneLeastSquares2020}. The first ensures that $F_x(y)$ is H\"older in the covariate direction $x$ uniformly in $y$, while the second ensures that the covariates are sufficiently dense in $\mathbf{I}$.
\begin{condition}[Regularity of $F_x$]\label{cond:regularity_F}
    There exist constants $\alpha \in (0,1]$ and $C_1 > 0$ such that
    \[
        \sup_{y \leq \tau} |F_u(y) - F_v(y)| \leq C_1 |u-v|^\alpha, \quad u, v \in \mathbf{I}.
    \]
\end{condition}

For the assumption below, let $\rho_n = \log(n)/n$ and let $\lambda$ denote Lebesgue measure.

\begin{condition}[Covariates are dense in $\mathbf{I}$]\label{cond:dense_covariates}
    There exist constants $C_2, C_3 > 0$ such that for arbitrary intervals $\mathbf{I}_n \subset \mathbf{I}$,
    \[
        \frac{|\{i \in \{1, \dots, n\}\colon X_{i} \in \mathbf{I}_n\}|}{n\lambda(\mathbf{I}_n)} \geq C_2 \text{ whenever } \lambda(\mathbf{I}_n) \geq \delta_n = C_3 \rho_n^{1/(2\alpha + 1)}
    \]
    with asymptotic probability one. That is, if $A_n$ denotes the above event, then $\lim_{n \rightarrow \infty} \pr(A_n) = 1$.
\end{condition}

\begin{theorem}[Uniform consistency]\label{th:sidr_plain_consistency}
    Assume non-informative censoring (\cref{cond:noninformative_censoring}), hazard rate ordering of $Y \mid X$ (\cref{cond:hazard_rate_order}), and the regularity \cref{cond:positive_probability_uncensored,cond:regularity_F,cond:dense_covariates}. Let $\hat{F}_x(y)$ be the estimator of \cref{def:sidr_plain}, and set $\mathbf{I}_n = \{x \in \R \colon [x\pm \delta_n] \subseteq \mathbf{I}\}$. Then there exists a constant $C = C(C_1, C_2, C_3, \eta) > 0$ such that
    \[
    	\lim_{n\rightarrow\infty} \pr\Big(\sup_{x \in \mathbf{I}_n, y \leq \tau} |\hat{F}_x(y) - F_x(y)| \geq C\rho_n^{\alpha/(1+2\alpha)}\Big) = 0.
    \]
\end{theorem}
\noindent\emph{Proof.} See \cref{a:plain_survival_idr_consistency}.

The rate $n^{-\alpha/(1+2\alpha)}$ in \cref{th:sidr_plain_consistency} is minimax optimal for distributional regression in the energy- or Wasserstein distance \citep{dombryDistributionalRegressionCRPSerror2024,dombryStonesTheoremDistributional2025}, and the plain survival IDR attains it locally and adaptively: if the conditional CDFs are H\"older continuous with index $\alpha$ on an interval $I \subset \mathbf{I}$ --- with $\alpha$ possibly depending on $I$ --- then the estimator converges at rate $n^{-\alpha/(1+2\alpha)}$ on $I$ up to a logarithmic factor. The result is the censored counterpart of Theorem 3.3 of \citet{moschingMonotoneLeastSquares2020}.

Plain survival IDR at \eqref{eq:sidr_plain}, while requiring the true $(F_x)_x$ to be hazard rate increasing (\cref{cond:hazard_rate_order}) for consistency, does not yield a solution that satisfies this condition in general. That is, the estimator at \eqref{eq:sidr_plain} does not choose an approximation of $(F_x)_x$ within a class that satisfies \cref{cond:hazard_rate_order}.

\begin{remark}\label{remark:sidr_plain_consistency_requirements}
    \Cref{th:sidr_plain_consistency} can be shown under variations of \cref{cond:hazard_rate_order}. Essential is that for $\emptyset \ne [x_l, x_r] = J \subseteq \mathbf{I}_n$, the Kaplan--Meier estimator evaluated on $\{(T_i, \Delta_i) \colon i \in O_{1:n}, X_i \in J \}$ converges to a certain mixture $\bar{F}_J$, described at \eqref{eq:fbar_k} in \cref{a:plain_survival_idr_consistency}, such that $\forall y \in \Rp \colon F_{x_l}(y) \ge \bar{F}_J(y) \ge F_{x_r}(y)$. This is satisfied under \cref{cond:hazard_rate_order}, as shown in \cref{prop:order}, but also when the censoring distribution is independent of the covariate ($\mathcal{L}(C \mid X) = \mathcal{L}(C)$), for example.
\end{remark}

\subsection{Investigating the conditions for consistency}\label{ss:from_plain_to_sidr}
We modified the non-censoring IDR at \eqref{eq:idr_min_max} using the Kaplan--Meier estimator to arrive at a new distributional estimator for the isotonic regression problem under censoring, but needed the condition of \cref{remark:sidr_plain_consistency_requirements} to show that the plain survival IDR is consistent.

The notion we require generalises Cauchy's \emph{internality} \citep{cauchyCoursDanalyseLEcole1821,bazPropertyReductionVariability2025}. It first appeared without a name in the isotonic regression literature \citep[property (e)]{robertsonConsistencyGeneralizedIsotonic1975} and was later named \emph{Cauchy mean value property} by \citet{robertsonAlgorithmsOrderRestricted1980}, whose terminology we adopt.
\begin{definition}\label{def:cmv}
    Let $U$ be any set and $M$ a function on $\bigcup_{d \in \mathbb{N}} U^d$ to some totally ordered set. Then $M$ has the \emph{Cauchy mean value} (CMV) property if for all $d_1, d_2 \in \mathbb{N}$,
    \[
        \forall A \in U^{d_1}, B \in U^{d_2}: \quad \min \{ M(A), M(B) \} \leq M(A \cdot B) \leq \max \{ M(A), M(B) \},
    \]
    where $A \cdot B := (A_1, \ldots, A_{d_1}, B_1, \ldots, B_{d_2}) \in U^{d_1 + d_2}$.
\end{definition}

In short, this definition requires that pooling data always results in a value between those on the individual subsets, no matter the split. This ``pooling betweenness'' property is satisfied by many aggregating functions, like the mean and median, see \citep{robertsonMultipleIsotonicMedian1973}. However, it is not satisfied by the Kaplan--Meier estimator: for some fixed threshold $y$, the estimator evaluated on data $((T_i, \Delta_i))_{i=1}^{n_1+n_2}$ may not lie between the Kaplan--Meier estimates on $((T_i, \Delta_i))_{i=1}^{n_1}$ and $((T_i, \Delta_i))_{i=n_1 + 1}^{n_1+n_2}$. \Cref{fig:example_cmv_violation_sample} shows that the Kaplan--Meier estimator may violate the CMV property by at least $|1/2 - 5/8| = 1/8$.

\begin{figure}[h]
    \centering
    \begin{tikzpicture}[
        alt={Sample of four data points, two for covariate 1 and two for covariate 2, with the threshold 3.5 marked.}
    ]
        \draw[->] (0.5,0) -- (2.5,0);
        \node[right] at (2.5,-0.075) {\(x\)};
        
        \draw (1,0) -- (1,2.25);
        \draw (2,0) -- (2,2.25);
        
        \fill (1,0.5) circle (2pt);
        \fill (2,1.5) circle (2pt);
        \fill (2,2) circle (2pt);
        
        \draw (1-0.1,1-0.1) -- (1+0.1,1+0.1);
        \draw (1-0.1,1+0.1) -- (1+0.1,1-0.1);
        
        \draw[dashed, thin] (0.5,1.75) -- (2.5,1.75);
        
        \node[below] at (1,0) {$1$};
        \node[below] at (2,0) {$2$};
        
        \draw[->, thin] (0.5,0) -- (0.5,2.25);
        \node[above] at (0.5,2.25) {\(y,c\)};
        \node[left] at (0.5,0.5) {$1$};
        \node[left] at (0.5,1) {$2$};
        \node[left] at (0.5,1.5) {$3$};
        \node[left] at (0.5,2) {$4$};
        
        \fill (3,1.375) circle (2pt);
        \node[right] at (3,1.4) {\ observed};
        \draw (3-0.1,0.95-0.1) -- (3+0.1,0.95+0.1);
        \draw (3-0.1,0.95+0.1) -- (3+0.1,0.95-0.1);
        \node[right] at (3,1) {\ censored};
    \end{tikzpicture}
    \hspace{4em}
    \begin{tikzpicture}[
        alt={Triangular matrix showing the along the diagonal that the Kaplan-Meier estimator takes values 0.5 at both covariates and 5/8 on the combined data set.}
    ]
        \draw[thick] (0,0) -- (0,-2);
        \node[left] at (0,-0.5) {$1$};
        \node[left] at (0,-1.5) {$2$};
        \draw[thick] (0,-2) -- (2,-2);
        \node[below] at (0.5,-2) {$1$};
        \node[below] at (1.5,-2) {$2$};
        \draw[thick] (0,-1) -- (2,-1);
        \draw[thick] (0,0) -- (1,0);
        \draw[thick] (1,0) -- (1,-2);
        \draw[thick] (2,-1) -- (2,-2);
        
        \draw[->, thin] (0,0.25) -- (2,0.25);
        \node[above] at (1,0.25) {\(r\)};
        
        \draw[->, thin] (-0.5,0) -- (-0.5,-2);
        \node[left] at (-0.5,-1) {\(s\)};
        
        \large
        \node at (0.5,-0.5) {\(\frac{1}{2}\)};
        \node at (0.5,-1.5) {\(\frac{5}{8}\)};
        \node at (1.5,-1.5) {\(\frac{1}{2}\)};
    \end{tikzpicture}
    
    \caption{The Kaplan--Meier estimator does not have the CMV property. Left panel: Sample of four data points $\{ (X_i, T_i, \Delta_i) \}_{i=1}^4 = \{ (1, 1, 1), (1, 2, 0), (2, 3, 1), (2, 4, 1) \}$. Right panel: Kaplan--Meier estimator $\KM{r:s}$ at thresholds $y \in [3, 4)$.}
    \label{fig:example_cmv_violation_sample}
\end{figure}
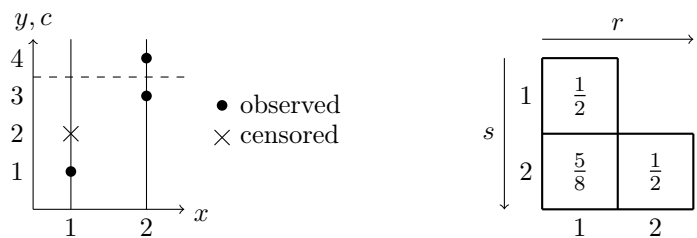

The CMV violation of the Kaplan--Meier estimator is not a finite-sample artefact, and any estimator that is to be consistent for $(F_x)_x$ under the usual stochastic order (\cref{cond:stochastic_order}) must therefore address it directly. \Cref{fig:example_cmv_violation_population} shows that the violation persists at the population level: when evaluating the estimator on an ever-growing sample of $(T_i, \Delta_i)$ observations, half coming from $(F_1, G_1)$ and half from $(F_2, G_2)$ as specified there, for $y \in [3, 4)$ the estimator converges to $7/16$ and not to $(F_1(y) + F_2(y))/2 = 1/2$. The violation can be made arbitrarily close to $1$ by modifying $F_1$, $F_2$, $G_1$ and $G_2$, see \cref{a:cmv_violation}; the limit object $\bar{F}$ is described in \cref{a:plain_survival_idr_consistency}.

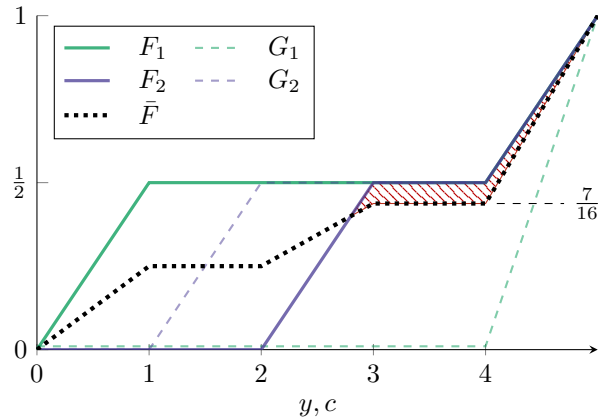
\begin{figure}[h]
    \centering
    \begin{tikzpicture}[
        alt={Line plot with time-to-event on the horizontal axis cdfs F1, F2, G1, G2 and Kaplan-Meier limit object F bar on the vertical axis, showing that the latter is below F1 and F2 near the interval from 3 to 4.}
    ]
        \begin{axis}[
            width=9cm,
            height=6cm,
            xmin=0, xmax=5,
            ymin=0, ymax=1,
            axis x line=bottom,
            axis y line*=left,
            xlabel={$y, c$},
            ylabel={},
            xtick={0,1,2,3,4},
            ytick={0,1/2,1},
            yticklabels={$0$, $\tfrac12$, $1$},
            legend cell align=left,
            legend pos=north west,
            legend columns=2,
            column sep=8pt
        ]
            \filldraw[
              draw=BrickRed,
              pattern={Lines[angle=-45,distance=2.5pt]},
              pattern color=BrickRed
            ] (axis cs:14/5,2/5) -- (axis cs:3,1/2) -- (axis cs:4,1/2) -- (axis cs:5,1) -- (axis cs:4,7/16) -- (axis cs:3,7/16) -- cycle;

            \addplot+[no marks, very thick, ForestGreen, opacity=0.75]
                coordinates {(0,0) (1,0.5) (4,0.5) (5,1)};
            \addlegendentry{$F_1$}
            \addplot+[no marks, thick, dashed, ForestGreen, opacity=0.5]
                coordinates {(0,0.01) (4,0.01) (5,1)};
            \addlegendentry{$G_1$}
            \addplot+[no marks, very thick, BlueViolet, opacity=0.75]
                coordinates {(0,0) (2,0) (3,0.5) (4,0.5) (5,1)};
            \addlegendentry{$F_2$}
            \addplot+[no marks, thick, dashed, BlueViolet, opacity=0.5]
                coordinates {(0,0) (1,0) (2,0.5) (4,0.5) (5,1)};
            \addlegendentry{$G_2$}
            \addplot+[no marks, ultra thick, dotted, black]
                coordinates {(0,0) (1,0.25) (2,0.25) (3,7/16) (4,7/16) (5,1)};
            \addlegendentry{$\bar{F}$}

            \draw[dashed] (4.1,7/16) -- (4.7,7/16);
            \node[right] at (4.7,7/16) {$\frac{7}{16}$};
        \end{axis}
    \end{tikzpicture}
    \caption{The Kaplan--Meier estimator does not have the CMV property at the population level. With $U[a,b]$ denoting the uniform distribution on $[a, b]$, we display the CDFs of $F_1 = (1/2) U[0,1] + (1/2) U[4,5]$, $F_2 = (1/2) U[2,3] + (1/2) U[4,5]$, $G_1 = U[4,5]$ and $G_2 = (1/2) U[1,2] + (1/2) U[4,5]$. The dotted line is the CDF $\bar{F}$ of the limit of the Kaplan--Meier estimator, when evaluated on samples of $(T_i, \Delta_i)$ observations, half coming from $(F_1, G_1)$ and half from $(F_2, G_2)$. The violation of the CMV property is highlighted in red.}
    \label{fig:example_cmv_violation_population}
\end{figure}

When the existing \EBM and \PRK, discussed at the end of \cref{s:prerequisites}, are applied to an ordered categorical covariate, the CMV property is not relevant, because the Kaplan--Meier estimator is not evaluated on samples from a mixture of distributions. This changes when these methods are applied to a continuous covariate through bucketing (see \cref{ss:simulations,ss:case_study}). The buckets each cover some interval of covariate values $x$, and because $F_x$ will in general change over that interval, the failure of the CMV property can cause deviations in the estimates, potentially breaking consistency.

\section{Survival Isotonic Distributional Regression (S-IDR)}\label{s:sidr}
\subsection{Definition and characterisation}\label{ss:definition_characterisation}
As established in the previous section, modifying IDR by replacing empirical CDFs with Kaplan--Meier estimators has undesirable effects. We would like to resolve the CMV obstruction of \cref{s:sidr_plain}, and will define a new inner quantity $\RKM{r:s}(y)$ of the $\max$-$\min$ equation for S-IDR. It should fulfil the CMV property of \cref{def:cmv}, that is, 
\begin{equation}\label{eq:inclusion_criterion}
    \RKM{r:s}(y) \in \bigcap_{P \in \mathcal{P}_{r:s}} \left[\min_{I \in P} \RKM{I}(y), \, \max_{I \in P} \RKM{I}(y)\right],
\end{equation}
for all $y \ge 0$.
The set $\mathcal{P}_{r:s}$ consists of all interval partitions of $[r:s]$ into at least two non-empty subsets.

Heuristically, S-IDR is a procedure built around the Kaplan--Meier estimator that enforces the CMV property by construction: at every partition scale, the recursion replaces $\KM{r:s}(y)$ by its projection onto the interval guaranteed by \eqref{eq:inclusion_criterion} on its strict sub-partitions. Formally:

\begin{definition}[Survival-IDR]\label{def:sidr}
    For $1 \leq r \leq m$, define $\RKM{r:r}(y) := \KM{r:r}(y)$ and, recursively for $1 \leq r < s \leq m$,
    \begin{equation} \label{eq:sidr_clamp}
        \RKM{r:s}(y) := \clamp\left(\KM{r:s}(y), \max_{P \in \mathcal{P}_{r:s}} \min_{I \in P} \RKM{I}(y), \min_{P \in \mathcal{P}_{r:s}} \max_{I \in P} \RKM{I}(y)\right).
    \end{equation}
    The \emph{Survival-IDR} (S-IDR) is defined as
    \begin{equation}\label{eq:sidr}
        \hat{F}_{\xi_i}(y) := \min_{r \leq i} \, \max_{s \geq i} \,\, \RKM{r:s}(y), \qquad y \in \Rp, \, 1 \le i \le m.
    \end{equation}
\end{definition}

For $x \not \in \{\xi_1, \dots, \xi_m\}$, one can again choose any interpolation that respects the stochastic order assumption. The above definition reduces to plain survival IDR at \eqref{eq:sidr_plain} if the Kaplan--Meier estimators respect the CMV property for the specific sample. For example, when there are no censored observations, S-IDR at \eqref{eq:sidr} is equal to IDR at \eqref{eq:idr_min_max}. We refer to $\RKM{}{}$ as the \emph{self-consistent Kaplan--Meier estimator} because, by construction, its value on $[r:s]$ is consistent with its values on every sub-partition of $[r:s]$, in the sense of \eqref{eq:inclusion_criterion}. The S-IDR definition can easily accommodate observation weights by extending \eqref{eq:kaplan_meier}, and indeed this functionality is included in the provided software packages.

\Cref{fig:plain_vs_self_consistent} illustrates the difference between plain survival IDR and S-IDR. With the distributions of \cref{fig:example_cmv_violation_population}, we simulate $n/2 = \num{25000}$ observations $(Y, C) \sim (F_1, G_1)$ and $X \sim U(0.5, 1.5)$ and the same number with $(F_2, G_2)$ and $X \sim U(1.5, 2.5)$. We compare the plain survival IDR and the S-IDR fit at threshold $y=3.5$. The latter takes few distinct values, as is common in isotonic regression. This is not the case for plain survival IDR, where just after $x=1.5$, the value smoothly changes and then approaches $\bar{F}(3.5) = 7/16$. The plain survival IDR is inconsistent for covariate $x = 2$ and threshold $y = 3.5$ since it pools the two groups with the Kaplan--Meier estimator. S-IDR stays close to the true value 1/2 on the interior of the entire covariate domain.

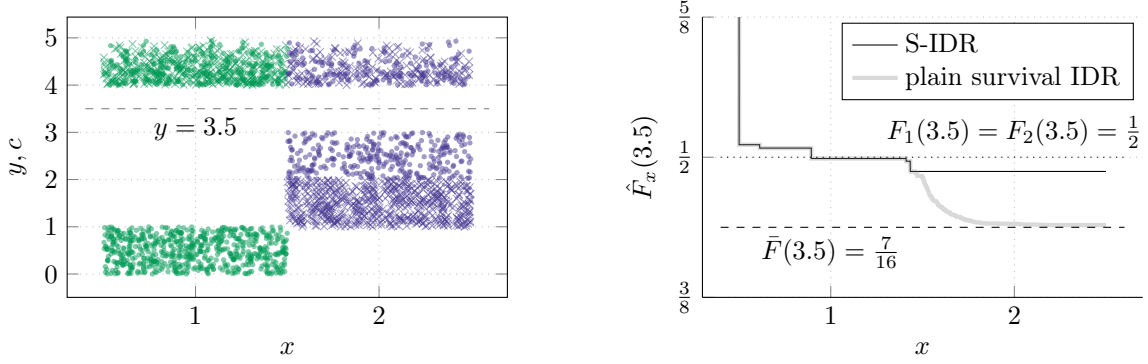
\begin{figure}[h]
    \begin{tikzpicture}[
        alt={Scatter plot with observed and censored points plotted for the covariate horizontally and response vertically, with a horizontal line marking the threshold 3.5.}
    ]
        \begin{axis}[
            width=0.49\linewidth,
            height=0.49*0.7142857142857143\linewidth,
            grid=major,
            grid style={dotted},
            xlabel={$x$},
            xtick={1,2},
            ylabel={$y, c$},
            ytick={0,1,2,3,4,5},
            table/col sep=comma,
            scatter/classes={
                TRUE={mark=*, mark size=1pt, draw opacity=0, fill opacity=0.5}, 
                FALSE={mark=x, mark size=2pt, draw opacity=0.5} 
            },
            scatter, only marks,
            scatter src=explicit symbolic
        ]
            \addplot+[
                mark options={draw=ForestGreen, fill=ForestGreen},
                x filter/.code={
                    \ifdim\pgfmathresult pt>1.5pt
                        \def\pgfmathresult{}
                    \fi
                },
            ] table[header=true, x=x, y=t, meta=d] {e2-consistency-data.csv};

            \addplot+[
                mark options={draw=BlueViolet, fill=BlueViolet},
                x filter/.code={
                    \ifdim\pgfmathresult pt<1.5pt
                        \def\pgfmathresult{}
                    \fi
                },
            ] table[header=true, x=x, y=t, meta=d] {e2-consistency-data.csv};

            \draw[thin, dashed, gray] (0.4,3.5) -- (2.6,3.5);
            \node[below] at (1,3.5) {$y = 3.5$};
        \end{axis}
    \end{tikzpicture}
    \hfill
    \begin{tikzpicture}[
        alt={Line plot with covariate on the horizontal axis and the S-IDR and plain survival IDR estimates on the vertical axis. They coincide on the left half of the plot, on the right half of the plot S-IDR stays close to 1/2 while the plain survival IDR smoothly converges to 7/16.}
    ]
        \begin{axis}[
                width=0.49\linewidth,
                height=0.49*0.7142857142857143\linewidth,
                ymin=3/8, ymax=5/8,
                axis x line*=bottom,
                axis y line*=left,
                xlabel={$x$},
                xtick={1,2},
                ylabel={$\hat{F}_x(3.5)$},
                ytick={3/8,1/2,5/8},
                yticklabels={$\tfrac38$, $\tfrac12$, $\tfrac58$},
                grid=major,
                grid style={dotted},
                table/col sep=comma,
                legend cell align=left
            ]

            \addplot[
              black, mark=none
            ] table[
              x=x, y=recursive,
              col sep=comma
            ] {e2-consistency-fits-subsampled.csv};
            \addlegendentry{S-IDR}
            \addplot[
              gray, ultra thick, draw opacity=0.3, mark=none
            ] table[
              x=x, y=nonrecursive,
              col sep=comma
            ] {e2-consistency-fits-subsampled.csv};
            \addlegendentry{plain survival IDR}
            \draw[dotted] (0.5-0.1,1/2) -- (2.5+0.1,1/2);
            \node[above] at (2,1/2) {$F_1(3.5)=F_2(3.5) = \frac{1}{2}$};
            \draw[dashed] (0.5-0.1,7/16) -- (2.5+0.1,7/16);
            \node[below] at (1,7/16) {$\bar{F}(3.5) = \frac{7}{16}$};
        \end{axis}
    \end{tikzpicture}

    \caption{
        Differences between plain survival IDR and S-IDR. Left panel: Subsample of $\num{2500}$ of the $\num{50000}$ simulated data points according to the distributions from \cref{fig:example_cmv_violation_population} and uniform random covariates. Right panel: Plain survival IDR (grey) and S-IDR (black) fit at threshold $y=3.5$. The dotted horizontal line marks the true value $F_1(3.5)=F_2(3.5)=1/2$; the dashed horizontal line indicates the (incorrect) value of the pooled Kaplan--Meier estimator $\bar{F}(3.5)=7/16$.
    }
    \label{fig:plain_vs_self_consistent}
\end{figure}

\begin{proposition}[S-IDR is well-defined]\label{prop:sidr_well_defined}
    The S-IDR clamping step at \eqref{eq:sidr_clamp} is well-defined because the lower bound is not above the upper bound, that is,
    \begin{equation}\label{eq:max_min_inequality}
        \max_{P \in \mathcal{P}_{r:s}} \min_{I \in P} \RKM{I}(y) \leq \min_{P \in \mathcal{P}_{r:s}} \max_{I \in P} \RKM{I}(y), \quad y \in \Rp, \, 1 \leq r \leq s \leq m.
    \end{equation}
    Consequently, S-IDR is well-defined.
\end{proposition}
\noindent\emph{Proof.} See \cref{a:sidr_well_defined}.

This inequality corresponds to the desired inclusion at \eqref{eq:inclusion_criterion}, because the left endpoint of the interval is the maximum of all left endpoints and analogously for the right endpoint. While equation \eqref{eq:max_min_inequality} is reminiscent of the max--min inequality, the latter does not directly apply in this case.

As in the uncensored case, the estimator can be equivalently defined via a max-min formula, and, at every threshold $y \in \Rp$, it creates a partition of the data into groups on which the estimator is constant; these are consequences of the CMV property \citep[Theorem 1a, 3a]{qianMinimumLowerSets1992}. The plain survival IDR generally does not satisfy these properties, since the Kaplan--Meier estimator can violate the CMV property. Furthermore, analogously to the IDR estimator, the S-IDR estimator can be equivalently formulated in terms of min-max or max-min formulas for the quantile function.

\begin{proposition}[Characterising S-IDR]\label{prop:sidr_characteristics}
\begin{itemize}
    \item[]
    \item[(i)] For $y \in \Rp$, S-IDR satisfies
    \[
    \hat{F}_{\xi_i}(y) = \min_{r \leq i} \max_{s \geq i} \RKM{r:s}(y) = \max_{s \geq i} \min_{r \leq i} \RKM{r:s}(y), \quad 1 \le i \le m.
    \]
    \item[(ii)] For $y \in \Rp$, there exists a partition $0 =: r_0 < 1 \leq r_1 < r_2 \dots < r_K = m$ such that 
    \[
        \hat{F}_{\xi_i}(y) = \RKM{(r_{j-1}+1):r_j}(y), \, \ r_{j-1} < i \leq r_j, \quad j = 1, \dots, K.
    \]
    \item[(iii)] For $\alpha \in [0, 1]$, let $\hat{q}_{r:s}(\alpha) = \KM{r:s}^{-1}(\alpha) = \inf\{z\colon \KM{r:s}(z) \geq \alpha\}$, with $\hat{q}_{r:s}(\alpha) = \infty$ for $\alpha > \lim_{z \rightarrow \infty} \KM{r:s}(z)$, and define $\tilde{q}_{r:r}(\alpha) = \KM{r:r}^{-1}(\alpha)$,
    \[
    \tilde{q}_{r:s}(\alpha) := \clamp\left(\hat{q}_{r:s}(\alpha), \max_{P \in \mathcal{P}_{r:s}} \min_{I \in P} \tilde{q}_I(\alpha), \min_{P \in \mathcal{P}_{r:s}} \max_{I \in P} \tilde{q}_I(\alpha)\right), \ 1 \leq r < s \leq m.
\]
Then, for $i = 1, \dots, m$ and $\alpha \in (0,1)$,
\[
    \hat{F}_{\xi_i}^{-1}(\alpha) = \min_{r \leq i} \max_{s \geq i} \tilde{q}_{r:s}(\alpha) = \max_{s \geq i} \min_{r \leq i} \tilde{q}_{r:s}(\alpha).
\]
\end{itemize}
\end{proposition}
\begin{proof}
    Because for $y \in \Rp$, the function $I \mapsto \RKM{I}(y)$ defined on index intervals fulfils the CMV property, Properties \emph{(i)} and \emph{(ii)} follow from Theorems 3 and 1a of \citet{qianMinimumLowerSets1992} respectively. For the proof of property \emph{(iii)}, see \cref{aa:quantile_perspective}.
\end{proof}

If $I$ is some index interval, there always exists some index sub-interval $J \subseteq I$ such that $\RKM{I}(y) = \KM{J}(y)$.  We may have $J$ as small as $|J| = 1$ even as $I$ is arbitrarily large. This raises the question of whether the estimator's performance is determined by the sample size in the interval that realises the estimator, but simulations suggest that this is not the case. When evaluating the self-consistent Kaplan--Meier estimator $\RKM{}{}$ and the usual Kaplan--Meier estimator $\KM{}$ on i.i.d. pairs of observations $(T, \Delta)$, we observe that $\RKM{}(y)$ and $\KM{}(y)$ converge at the same rate and are close already for moderate $n$. However, the probability of the two estimators being equal seems to approach 0 as $n \to \infty$.

\subsection{Uniform consistency}\label{ss:uniform_consistency}
To show consistency of S-IDR with the recursive estimator, we restrict the index intervals of the procedure to contain at least a certain number $\lceil c_n \rceil \in \amsmathbb{N}$ of observations. This makes analysis more tractable, but we believe that the consistency result holds for the unmodified S-IDR as well; the simulations of \cref{ss:simulations} provide supporting evidence. We denote this modified S-IDR by $\hat{F}_{x, \mathcal{I}(c_n)}$, defined precisely in \cref{aa:sidr_modified_consistency}.

\begin{condition}[Regularity of $G_x$]\label{cond:regularity_G}
    The censoring distributions $(G_u)_{u \in \mathbf{I}}$ satisfy \cref{cond:regularity_F} with the same constants $\alpha$ and $C_1$.
\end{condition}

The equality of the constants for $G_u$ and $F_u$ in the above condition is merely for simplification. If the $G_u$ are H\"older continuous with a different constant $\beta \neq \alpha$, then the rate of the estimator becomes $\rho_n^{\min(\alpha,\beta)/(1+2\min(\alpha,\beta))}$.

\begin{condition}[Support of $X$]\label{cond:support_X}
    The covariate $X$ satisfies $\amsmathbb{P}(X \in \mathbf{I}) = 1$.
\end{condition}

\Cref{cond:support_X} means that the assumptions on smoothness of the outcome and censoring CDFs with respect to the covariate (\cref{cond:regularity_F,cond:regularity_G}) and the assumption that covariates are dense (\cref{cond:dense_covariates}) hold globally on the support of $X$, and not only on a local interval as for \cref{th:sidr_plain_consistency}.

\begin{theorem}[Uniform consistency]\label{th:sidr_modified_consistency}
    Assume that \cref{cond:noninformative_censoring}, \cref{cond:stochastic_order}, and \cref{cond:positive_probability_uncensored,cond:regularity_F,cond:dense_covariates,cond:regularity_G,cond:support_X} hold. If $c_n = kC_2C_3n \rho_n^{1/(1+2\alpha)}$ for some $k \in (0,1]$, there exists a constant $\tilde{C} = \tilde{C}(C_1, C_2, C_3, \eta, k) > 0$ such that
    \[
        \lim_{n\rightarrow\infty} \pr\Big(\sup_{x \in \tilde{\mathbf{I}}_n, y \leq \tau} |\hat{F}_{x, \mathcal{I}(c_n)}(y) - F_x(y)| \geq \tilde{C}\rho_n^{\alpha/(1+2\alpha)}\Big) = 0,
    \]
    where $\tilde{\mathbf{I}}_n = \{x \in \R \colon [x\pm 4\delta_n] \subseteq \mathbf{I}\}$.
\end{theorem}
\noindent\emph{Proof.} See \cref{aa:sidr_modified_consistency}.

The main difference to \cref{th:sidr_plain_consistency} is that the recursive estimator is consistent under the usual stochastic order (\cref{cond:stochastic_order}), rather than the stronger hazard rate order (\cref{cond:hazard_rate_order}). This requires additional technical assumptions (\cref{cond:regularity_G,cond:support_X}) to ensure that the Kaplan--Meier estimator over non-identically distributed observations with covariates $X_i$ in a small interval around a given $x$ is close to the true outcome CDF at that $x$; see \citet{zhouConsistencyKaplanMeierEstimator2017} for related theoretical results. The hazard rate order for \cref{th:sidr_plain_consistency} avoids these difficulties, but it is an unnecessarily strong assumption for (the modified version of) S-IDR.

While the modified S-IDR achieves the same rate as the direct estimator in \cref{th:sidr_modified_consistency}, there is a caveat. The value $c_n = kC_2C_3n \rho_n^{1/(1+2\alpha)}$ for choosing the restricted index intervals depends on $\alpha$, $C_2$ and $C_3$, which are unknown in practice. The parameter $c_n$ behaves like the bandwidth of a histogram, which needs to be chosen carefully to achieve the optimal convergence rate. In our experience the estimator performs well without binning (i.e., with $c_n = 1$), and we believe it is also consistent without this modification, but we were not able to prove this.

\subsection{Computational aspects}\label{ss:computational_aspects}
S-IDR has a clear motivation in terms of the CMV property, but it may seem computationally intractable for larger $m$ due to the exponential number of partitions $|\mathcal{P}_{r:s}|$. In this section, we show that this is not an obstacle and, moreover, it is much faster to compute S-IDR than the plain survival IDR on the same data. First, we apply the \emph{Pool Adjacent Violators} (PAV) algorithm \citep{ayerEmpiricalDistributionFunction1955,barlowIsotonicRegressionProblem1972} to accelerate the computation for any fixed threshold $y \in \Rp$, then a multi-threshold acceleration algorithm [\PRK, \citealp{henziAcceleratingPoolAdjacentViolatorsAlgorithm2022}] to compute the complete solution for all $y$ simultaneously.

The following lemma helps to deal with these partitions, which determine the clamping bounds.
\begin{lemma}[One-step recursion form]\label{lem:sidr_one_step_recursion}
    The S-IDR clamping step at \eqref{eq:sidr_clamp} is equivalent to
    \begin{equation}\label{eq:sidr_clamp_binary_splits}
        \RKM{r:s}(y) = \clamp\left(\KM{r:s}(y), \max_{r \leq k < s} \min \{ \RKM{r:k}(y), \RKM{k+1:s}(y) \}, \min_{r \leq k < s} \max \{ \RKM{r:k}(y), \RKM{k+1:s}(y) \}\right),
    \end{equation}
    where the bounds are determined by interval partitions $P \in \mathcal{P}_{r:s}$ of size $|P| = 2$.
\end{lemma}

Informally, to find the bounds within which $\KM{r:s}(y)$ is clamped for the index interval $[r:s]$, we need to account for the $s - r$ split points $r \leq k < s$ and the associated values $\RKM{r:k}(y), \RKM{k+1:s}(y)$. This allows a dynamic programming approach to compute $\RKM{r:s}(y)$, going down the diagonals of the triangle $r \leq p \leq q \leq s$, described in detail in \cref{alg:rkm}.

\begin{algorithm}
    \caption{Dynamic program for $\RKM{r:s}(y)$}\label{alg:rkm}
    \begin{algorithmic}[1]
        \Procedure{Triangle Bounding Sweep}{$\{(X_i, T_i, \Delta_i)\}_{i=1}^n$, $1 \leq r \leq s \leq m$, $y \in \Rp$}
            \State $\text{RKM}[r\!:\!s][r\!:\!s]\gets\text{undef}$ \Comment{The $\{\RKM{I}(y): I \subseteq [r:s]\}$ are stored in this array}

            \For{$i\gets r$ \textbf{to} $s$}
                \State $\text{RKM}[i][i]\gets \KM{i:i}(y)$ \Comment{Clamping is not needed for $I$ with $|I| = 1$}
            \EndFor

            \For{$\ell\gets 2$ \textbf{to} $s-r+1$} \Comment{Interval length $\ell$ counts the sub-diagonal index}
                \For{$p\gets r$ \textbf{to} $s-\ell+1$} \Comment{Interval start index $p$}
                    \State $q\gets p+\ell-1$;\; $L \gets 0$;\; $U \gets 1$ \Comment{Interval end index (inclusive) $q$}
                    \For{$k\gets p$ \textbf{to} $q-1$} \Comment{Cut points $k$ parameterise $\{P \in P_{p:q} : |P| = 2\}$}
                        \State $L\gets \max\!\big(L,\,\min\{\text{RKM}[p][k],\text{RKM}[k+1][q]\}\big)$
                        \State $U\gets \min\!\big(U,\,\max\{\text{RKM}[p][k],\text{RKM}[k+1][q]\}\big)$
                    \EndFor
                    \State $\text{RKM}[p][q]\gets \clamp(\KM{p:q}(y), L, U)$
                \EndFor
            \EndFor
            \State \Return $\text{RKM}[r][s]$
        \EndProcedure
    \end{algorithmic}
\end{algorithm}

\begin{remark}[Implementation tweaks]\label{remark:alg_rkm_tweaks}
    Small modifications to \cref{alg:rkm} may improve performance. For example, at line \num{13}, the computation of the Kaplan--Meier estimator $\KM{p:q}(y)$ may be skipped if $L = U$. Moreover, inside the loop over $k$ starting at line \num{9}, we could check whether $L = U$ and break from the loop earlier. Such modifications are implemented in the software package.
\end{remark}

The computation of the Kaplan--Meier estimators $\{\KM{I}(y): I \subseteq [r:s]\}$ cannot be accelerated beyond $\Theta(|O_{r:s}|^3 \log |O_{r:s}|)$. We can store our data to make it inexpensive to select by ranges $[r:s]$ of the covariate $X$, but we need to sort the observations $O_{r:s}$ by event time $T$ (and indicator $\Delta$), incurring a $\Theta(|O_{r:s}| \log |O_{r:s}|)$ cost. The $(s - r + 1) (s - r + 2) / 2 = \Theta(|O_{r:s}|^2)$ estimators share $\Theta(|O_{r:s}|^3)$ unique factors that the Kaplan--Meier estimators are built from. Faster computation is only possible if the number of unique covariate values $X$ or the number of unique event times $T$ is known to be small. These costs are not incurred for each threshold when we compute the complete S-IDR solution, as we reuse computations between thresholds.

For a threshold $y \in \Rp$, the PAV algorithm can be applied to the values $\RKM{\cdot}(y)$ to evaluate the S-IDR $\max-\min$ equation, because $\RKM{\cdot}(y)$ satisfies the CMV property \citep{robertsonAlgorithmsOrderRestricted1980} by construction. As is needed for quantile regression or the method from \PRK, see \citep{bestMinimizingSeparableConvex2000}, our PAV implementation needs a merge step that is more costly than the weighted average computation used in classical isotonic mean regression. In fact, because the values $\RKM{\cdot}(y)$ are defined recursively, it is necessary to compute all $\{ \RKM{I}(y) : I \subsetneq [r:s] \}$ to obtain $\RKM{r:s}(y)$.

In the worst case, $O(m^2)$ self-consistent Kaplan--Meier estimators $\RKM{\cdot}(\cdot)$ need to be evaluated to compute S-IDR for a single threshold. This is encountered only in unfavourable instances where $Y$ is independent of $X$ or $Y$ decreases with $X$ (violating modelling assumptions); when $Y$ increases everywhere with $X$, the count drops to $O(m)$. \Cref{fig:expected_runtime} illustrates both regimes.

Quantitatively, adding a positive slope to the independent simulation decreases the share of $\RKM{\cdot}(0)$ that must be evaluated from $26.2\%$ to $6.6\%$ of the $m(m+1)/2$ intervals $[r:s] \subseteq [1:m]$. The effect magnifies for the complete S-IDR solution, because more partial computations can be reused across thresholds and the smaller partitions yield faster individual Kaplan--Meier evaluations. Computing the complete S-IDR solutions took $\num{46}$ and $\num{1.1}$ seconds in the unfavourable and favourable cases respectively at $n = m = \num{2500}$, on a 2021 consumer laptop (Intel i7-1165G7 CPU).

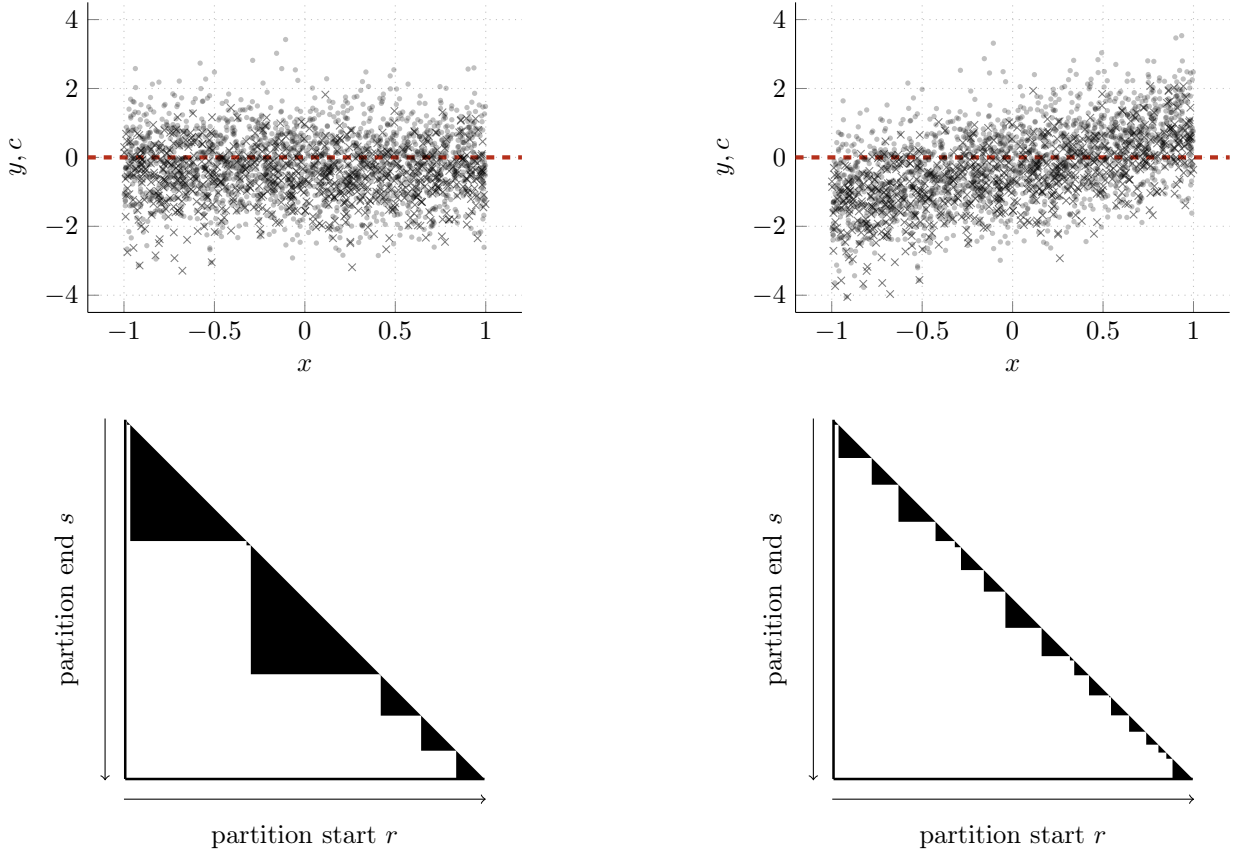
\begin{figure}
    \newlength{\axw}\setlength{\axw}{0.38\textwidth}
    
    \begin{tikzpicture}[
        trim axis left,
        alt={Scatter plot with covariate on the horizontal axis and the response on the vertical axis. Point cloud of censored and observed points hovers around, but on average just below, the 0 threshold with no visible covariate / response relationship.}
    ]
        \begin{axis}[
                scale only axis,
                width=\axw, height=0.7142857\axw,
                xmin=-1.2, xmax=1.2,
                grid=major,
                grid style=dotted,
                axis x line*=bottom,
                axis y line*=left,
                xlabel={$x$},
                ylabel={$y,c$},
                ymin=-4.5, ymax=4.5,
                table/col sep=comma,
                scatter, only marks,
                scatter/classes={
                    TRUE={mark=*, mark size=1pt, draw opacity=0, fill opacity=0.25}, 
                    FALSE={mark=x, mark size=2pt, draw opacity=0.5} 
                },
                scatter src=explicit symbolic,
                legend cell align=left
            ]
            \addplot+[
                mark options={draw=black, fill=black},
            ] table[header=true, x=x, y=t, meta=d] {e4-worst-case-normal.csv};
            \draw[BrickRed, dashed, ultra thick, opacity=1] (-1.2,0) -- (1.2,0);
        \end{axis}
    \end{tikzpicture}
    \hfill
    \begin{tikzpicture}[
        trim axis left,
        alt={Scatter plot with covariate on the horizontal axis and the response on the vertical axis. Point cloud of censored and observed points hovers around, but on average just below, the 0 threshold with a slight linearly increasing covariate / response relationship visible.}
    ]     
        \begin{axis}[
                scale only axis,
                width=\axw, height=0.7142857\axw,
                xmin=-1.2, xmax=1.2,
                grid=major,
                grid style=dotted,
                axis x line*=bottom,
                axis y line*=left,
                xlabel={$x$},
                ylabel={$y,c$},
                ymin=-4.5, ymax=4.5,
                table/col sep=comma,
                scatter, only marks,
                scatter/classes={
                    TRUE={mark=*, mark size=1pt, draw opacity=0, fill opacity=0.25}, 
                    FALSE={mark=x, mark size=2pt, draw opacity=0.5} 
                },
                scatter src=explicit symbolic,
                legend cell align=left
            ]
            \addplot+[
                mark options={draw=black, fill=black},
            ] table[header=true, x=x, y=t, meta=d] {e4-worst-case-shifted.csv};
            \draw[BrickRed, dashed, ultra thick, opacity=1] (-1.2,0) -- (1.2,0);
        \end{axis}
    \end{tikzpicture}
    
    \begin{tikzpicture}[
        trim axis left,
        alt={A right triangle with the right angle at the bottom-left, containing smaller right triangles in the same orientation along its diagonal, with two large triangles covering a third of the diagonal and covering a significant of the outer triangle's surface}
    ]
        \begin{axis}[
                scale only axis,
                width=\axw, height=\axw,
                xmin=-1.2, xmax=1.2,
                xlabel={partition start $r$},
                ylabel={partition end $s$},
                axis line style={draw=none},
                xtick pos=bottom,
                ytick pos=left,
                xtick=\empty,
                ytick=\empty,
                tick style={draw=none}
            ]
            
            \addplot graphics[
                xmin=-1,
                xmax=1,
                ymin=-1,
                ymax=1
            ] {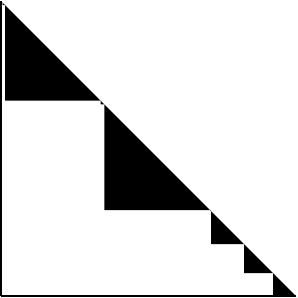};

            \draw[->, thin] ([xshift=-0.25cm] axis cs:-1,1)
                         -- ([xshift=-0.25cm] axis cs:-1,-1);
            \node[left] at ([xshift=-0.25cm] axis cs:-1,0) {};
            
            \draw[->, thin] ([yshift=-0.25cm] axis cs:-1,-1)
                         -- ([yshift=-0.25cm] axis cs:1,-1);
            \node[below] at ([yshift=-0.25cm] axis cs:0,-1) {};
        \end{axis}
    \end{tikzpicture}
    \hfill
    \begin{tikzpicture}[
        trim axis left,
        alt={A right triangle with the right angle at the bottom-left, containing smaller right triangles in the same orientation along its diagonal. All inner triangles are small and cover only a small part of the outer triangle's surface}
    ]
        \begin{axis}[
                scale only axis,
                width=\axw, height=\axw,
                xmin=-1.2, xmax=1.2,
                xlabel={partition start $r$},
                ylabel={partition end $s$},
                axis line style={draw=none},
                xtick pos=bottom,
                ytick pos=left,
                xtick=\empty,
                ytick=\empty,
                tick style={draw=none}
            ]
            
            \addplot graphics[
                xmin=-1,
                xmax=1,
                ymin=-1,
                ymax=1
            ] {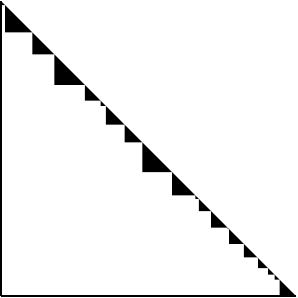};

            \draw[->, thin] ([xshift=-0.25cm] axis cs:-1,1)
                         -- ([xshift=-0.25cm] axis cs:-1,-1);
            \node[left] at ([xshift=-0.25cm] axis cs:-1,0) {};
            
            \draw[->, thin] ([yshift=-0.25cm] axis cs:-1,-1)
                         -- ([yshift=-0.25cm] axis cs:1,-1);
            \node[below] at ([yshift=-0.25cm] axis cs:0,-1) {};
        \end{axis}
    \end{tikzpicture}

    \caption{
        The computational burden depends strongly on the isotonicity of the problem instance. Top left: Simulated $X \sim \mathcal{U}(-1, 1)$, $Y \sim \mathcal{N}(0, 1)$, with $C$ drawn from the same distribution as $Y$ conditional on $C < \infty$, and $\pr(C = \infty) = 1/2$; all variables independent. Top right: Same observations, but $Y$ and $C$ are shifted by $X$, creating a positive dependency. Bottom: Corresponding partitions at $y = 0$ with the black triangles marking the (self-consistent) Kaplan--Meier estimators $\RKM{}$ that need to be computed over ranges $[r:s]$ in the PAV algorithm.
    }

    \label{fig:expected_runtime}
\end{figure}

When computing the complete S-IDR solution $(\hat{F}_x(y))_x$ --- that is, for all thresholds $y \geq 0$ --- we benefit further from the CMV property of $\RKM{\cdot}(y)$. If the $m'$ unique time-to-event observations are $z_1, \ldots, z_{m'}$ in increasing order (analogous to the unique observed covariates $\xi_i$), then the S-IDR $(\hat{F}_x)_x$ is constant at least on the intervals $[z_i, z_{i+1})$ for all $x \in \mathbf{I}$, because each Kaplan--Meier estimator is. Hence, we can restrict our attention to the thresholds with at least one (uncensored) observation. Typically, most covariates $x$ that were grouped together in a level set of $\hat{F}_x(z_i)$ are also grouped together in $\hat{F}_x(z_{i+1})$, and when computing the thresholds from lowest ($z_1$) to highest ($z_{m'}$), we can warm-start the PAV algorithm for threshold $z_{i+1}$ using (some of) the partition elements of threshold $z_i$. This approach has been proposed by \PRK[Algorithm 2 in the Supplementary Material], and by \citet{henziAcceleratingPoolAdjacentViolatorsAlgorithm2022}. Through careful bookkeeping, storing partial computations of the Kaplan--Meier estimators $\KM{\cdot}(y)$ and their self-consistent versions $\RKM{\cdot}(y)$ helps avoid duplicate computations.

Taken together, the full S-IDR solution for all $m'$ thresholds runs in worst-case time
\begin{equation}\label{eq:sidr_complexity}
    O\bigl(n\log(n) \, + \, m' m^3\bigr),
\end{equation}
The $n \log(n)$ term corresponds to the sorting and deduplicating of input data, while the $m' m^3$ term reflects the S-IDR algorithm itself. The cost is determined not by the Kaplan--Meier estimators themselves --- whose partial solutions can be maintained as we ascend through the thresholds $z_i$ --- but by the bounds that are checked. For each threshold, at most $m$ observations (after deduplication) get incorporated into at most $m(m+1)/2$ estimators, and then, at most $m(m+1)/2$ bounds $L$ and $U$ should be collected, which costs at most $m$ for each bound. Both these per-threshold operations cost $O(m^3)$, resulting in the $O(m' m^3)$ complexity over all thresholds.

Several insights further accelerate the computation of the complete S-IDR solution. For example, censored observations whose event time is below that of any uncensored observation can be discarded, because they do not influence any of the individual Kaplan--Meier estimators that S-IDR is built on, and hence do not influence S-IDR itself. Moreover, recalling that S-IDR coincides with IDR when all observations are uncensored, we can compute the lowest thresholds of S-IDR using the significantly faster algorithm for IDR. When we encounter the first threshold $z_i$ with a censored observation, we switch to the S-IDR algorithm. We can warm-start S-IDR by using the last partition $P$ computed by IDR for threshold $z_{i-1}$ and computing $\{ \RKM{I}(z_{i-1}) : I \in P \}$. Then, we update this S-IDR partition for threshold $z_i$ with the procedure from \citet{henziAcceleratingPoolAdjacentViolatorsAlgorithm2022}, but with the previously described modifications.

\begin{figure}
    \includegraphics[
        width=\textwidth,
        alt={Two panels with a line plot, horizontal axis is the response and vertical axis is the cdf. 5 line groups show the true and estimated cdfs.}
    ]{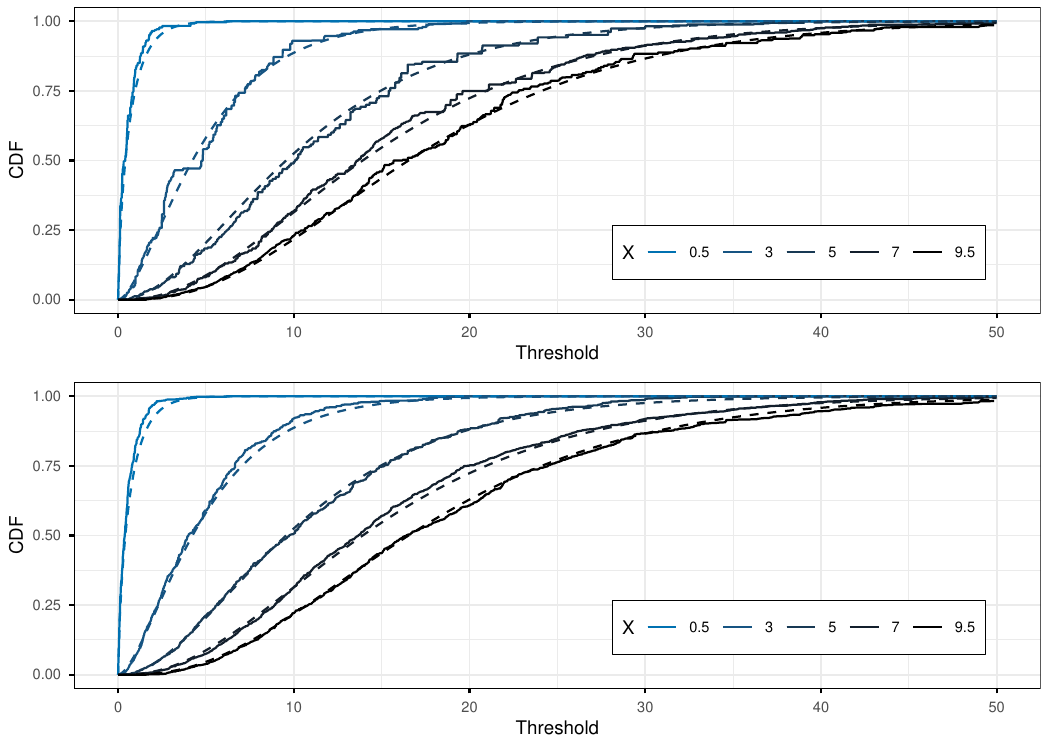}
    \caption{Subagging applied to a partially censored sample of size $n = \num{10000}$ following the distribution of \cref{eq:illustration}. Top panel:  True conditional CDFs (dashed lines) and S-IDR solution (solid lines) on the full sample. Bottom panel:  Subagged S-IDR solution (solid lines) using \num{100} subsamples of size \num{1000} each, linearly aggregated.}
    \label{fig:subagging}
\end{figure}

Like IDR, S-IDR benefits from subsample aggregation (subagging) \citep{buhlmannAnalyzingBagging2002}. As can be seen in \cref{fig:subagging}, the smoothness of the solution can benefit substantially; for computational efficiency too, subagging is generally advisable if the number of uncensored samples $n$ is at least of medium magnitude. The scaling of runtime with $n$ is much more favourable under subagging than without; moreover, computations can be parallelised. The speedup is meaningful: in the case of \cref{fig:subagging}, from $\num{40}$ seconds for $n = \num{10000}$ without subagging, to $\num{6.9}$ seconds for \num{100} subsamples of size $n = \num{1000}$ computed on a single CPU core. Despite worst-case $O(n^4)$ scaling, S-IDR's runtime is acceptable in practice, especially when using subagging on multiple CPU cores.

As will be reflected in the more elaborate performance assessment in the next section, subagging is most beneficial when the conditional distributions are continuous and smoothness is important. The subsample should be large enough to already capture the structure of the problem, which would not arise from aggregating (averaging) individual fits that do not represent these features, as when boosting weak learners.

\section{Extension to multi-dimensional covariates}\label{s:extensions}
So far, we have considered a one-dimensional covariate space $\mathbf{I} \subset \R$, which is a totally ordered set: for $x_1, x_2 \in \mathbf{I}$, either $x_1 \le x_2$ or $x_1 \ge x_2$ (or both). A more general covariate space $\mathcal{X}$, such as (subsets of) $\R^p$, $p \in \amsmathbb{N}$, may not have a total order accurately describing the relationship of all element pairs.
For example, the case study of \cref{ss:case_study} discusses models for mortality based on lab measurements, each affecting mortality in a single known direction through a well-understood biological mechanism. A clean comparison of mortality risk between two patients A and B is feasible only when all measurements are uniformly worse for one patient than the other; once different measurements point in opposing directions, it becomes ambiguous which should carry the greater weight.

There are two ways to apply S-IDR with a partially ordered covariate space $\mathcal{X}$. First, one can apply S-IDR directly, mirroring the approach for IDR \citep[Section~3]{henziIsotonicDistributionalRegression2021}. For statistical and computational efficiency, this works well in practice only if there are few unique covariates or if many of the observed covariates are comparable, such as when covariate dimensions are positively correlated. This approach is implemented in the software package and discussed in more detail in \cref{aa:partial_orders}. Second, the covariate set $\mathcal{X}$ can be reduced to a more structured set --- lower dimensional, typically to an interval $\mathbf{I} \subset \R$ where covariates are totally ordered again --- as proposed in \citet{henziDistributionalSingleIndex2023}. We discuss this distributional index model approach in the remainder of this section, demonstrate its abilities in \cref{ss:index_model_benchmark} and apply it in the case study of \cref{ss:case_study}.

Regression models, especially those working with multidimensional covariates $x \in \mathcal{X}$, commonly assume the target function to lie in a restricted class. The Cox proportional hazards (PH) model, for example, assumes that the target function can be described by rescaling a common hazard by a linear combination of covariates. In the AFT model, too, only a linear combination of covariates determines the represented distribution, and both are examples of distributional linear (single) index models.
We will work with a more general notion: the Distributional Index Model (DIM) of \citet{henziDistributionalSingleIndex2023}, which merely assumes the existence of a family of CDFs $(F_u)_{u \in \R^d}$ (typically $d = 1$) such that $F_u \preceq F_v$ whenever $u \preceq v$, and an index (link) function $\theta: \mathcal{X} \to \R^d$ such that
\begin{equation}\label{eq:dim}
    \pr(Y \le y \mid X) = F_{\theta(X)}(y), \quad y \in \Rp.
\end{equation}
Both the Cox PH and AFT models can be expressed within this framework with $\mathcal{X} = \R^p$ and index $\theta(X) = \beta^T X$ (or equivalently $\theta(X) = \exp(\beta^T X)$) for some coefficient vector $\beta \in \R^p$. The index function $\theta$ should be estimated (up to monotone transformations, under which S-IDR is invariant) by some other estimator $\hat{\theta}$ using the data $((X_i, T_i, \Delta_i))_{1 \le i \le n}$. Then, S-IDR can be fit on $((\hat{\theta}(X_i), T_i, \Delta_i))_{1 \le i \le n}$, with the pseudo-index $\hat{\theta}(X_i)$ as a new covariate. The combined model can then be evaluated at $x \in \mathcal{X}, y \in \Rp$ as $\hat{F}_{\hat{\theta}(x)}(y)$. To avoid overfitting, we recommend splitting the data into two parts: the first for estimating the index model, the second for fitting S-IDR; \cref{ss:index_model_benchmark} reports performance under this protocol. If $\hat{\theta}$ is consistent for $\theta$ and conditions similar to those for \cref{th:sidr_modified_consistency} are satisfied, $\hat{F}_{\hat{\theta}(x)}(y)$ can then be shown to be a consistent estimator for $\pr(Y \le y \mid X = x)$, see \cref{aa:index_model_consistency}. A similar approach was recently taken by \citet{jainIsotonicSurvivalRegression2026} to recalibrate (deep) Cox models.

In non-survival regression the standard target is a conditional point prediction --- mean, median, or some other functional --- so a model for $\hat{\theta}$ (linear regression, random forests, multilayer perceptrons, etc.) can be combined with a distributional model such as IDR to upgrade the point predictor to a distributional one. In survival analysis the standard target is itself a conditional distribution, so an index is not immediately at hand. However, many popular survival models are parametric and relatively simple, so they often specify a risk score, sometimes implicitly \citep{lillelundStopChasingCindex2025}; examples include the linear indices $\beta^T X$ of the Cox PH and AFT models. Certain nonparametric approaches, such as random survival forests \citep{ishwaranRandomSurvivalForests2008}, likewise furnish a score. Other options are to evaluate the predicted survival curve at a fixed time horizon, or to derive a conditional mean (when finite) or median from it. Crucially, in any of these cases the index model must still properly account for censoring whenever it is present.

In practice, it is not always clear whether the model at \eqref{eq:dim} is satisfied and a conditional mean or other risk score can order the conditional distributions stochastically. For models without a canonical risk score, even predicted distributions are typically not strictly stochastically ordered with respect to the risk score, as with random survival forests. Despite these caveats, the index model technique is effective in practice, for example to recalibrate distributions from a miscalibrated model, as we show in \cref{ss:index_model_benchmark}.

\section{Empirical results}\label{s:empirical_results}
Implementations of S-IDR for Rust \citep{matsakisRustLanguage2014}, R \citep{rcoreteamLanguageEnvironmentStatistical2025, wilkeRextendrCallRust2025} and Python \citep{rossumPythonReferenceManual1995, pyo3projectandcontributorsPyO3} extend the existing \texttt{isodistrreg} package of \citet{henziIsotonicDistributionalRegression2021} and are available at \url{https://github.com/AlexanderHenzi/isodistrreg}.

\subsection{Simulations}\label{ss:simulations}
To assess the performance of S-IDR relative to the methods \EBM and \PRK, described at the end of \cref{s:prerequisites}, we compare estimates against six known collections of conditional distributions. All existing order-restricted methods that we are aware of work only on discrete covariates, so we apply these methods to bucketed data instead, and interpolate between the centres of the buckets. Bucket widths follow bias-variance trade-offs influenced by the (possibly unknown) smoothness of the generating distribution. We choose widths $\propto n^{-1/5}$ and $\propto n^{-1/3}$ and compare under both, but show in \cref{fig:bucketing_sensitivity} that for some problems, there is no bucket width that allows these methods to match the performance of S-IDR. This combination of bucketing followed by monotonisation, with \EBM or \PRK, can be interpreted as monotonisation of the survival curve with respect to the covariate, analogous to projection or rearrangement of point estimators \citep{chernozhukovImprovingPointInterval2009}. The test problems are designed to cover a variety of cases: smooth (P1, P2, P3, P6), discrete (P4, P5), adhering to \cref{cond:hazard_rate_order} (P1) or not (all but P1). The censoring distribution is also varied, and problem P2 imitates the case study problem of \cref{ss:case_study}, \cref{fig:post_op_survival}. A detailed description of the test problems and evaluation metrics can be found in \cref{a:sidr_simulations}.

\begin{figure}[h]
    \centering
    \includegraphics[
        width=\textwidth,
        alt={Grid of line plots, one per simulated problem instance, showing sample sizes n on the horizontal axes and metric error on the vertical axis. The lines, each plot has one line per method, decrease and are contained in a shaded band getting narrower to the right. S-IDR has lowest error on the adversarial problems and comparable error on the others.}
    ]{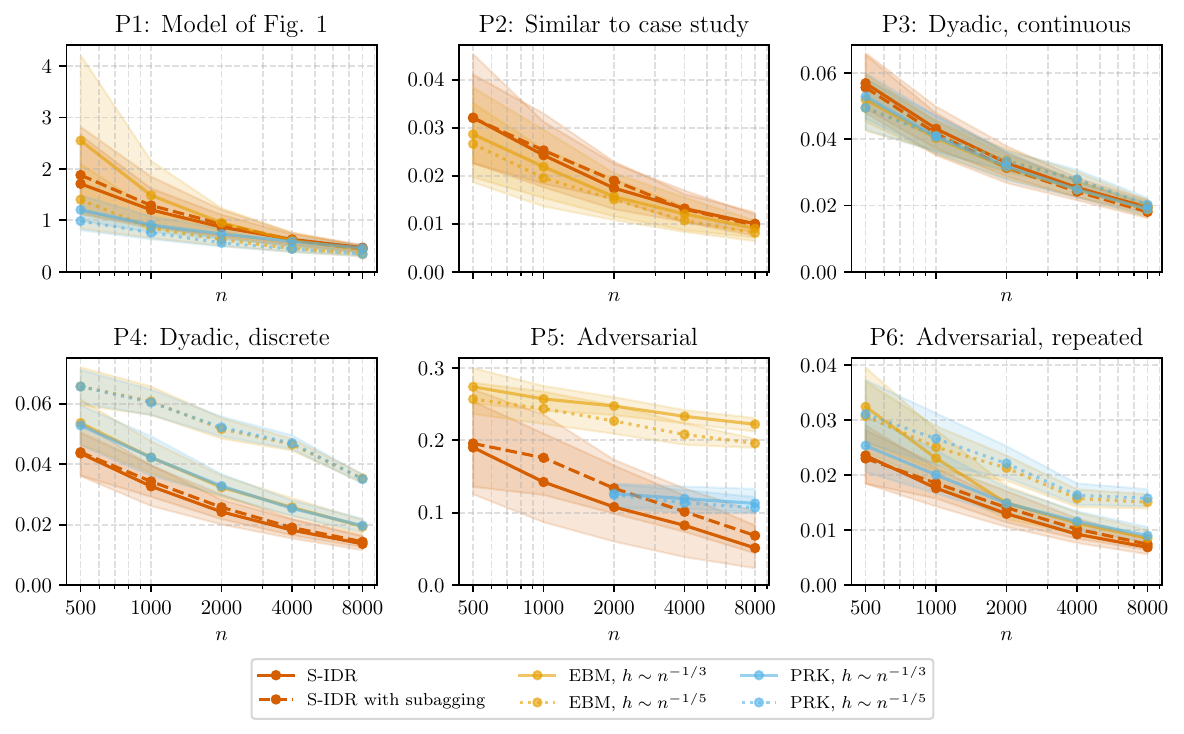}
    \captionsetup{singlelinecheck=off}
    \caption[.]{
        Mean and central 80\% interval of the $L_1$ distance between estimated and ground-truth CDFs, averaged over the covariate support. Lower is better. In some panels, the \EBM (yellow) and \PRK (blue) lines coincide and render as green. For subagging, 50 subsamples of size $n' = \max\{n/2,\num{400}\}$ were used. See \cref{a:sidr_simulations} for details.
    }
    \label{fig:simulation}
\end{figure}

Relative to the other methods, S-IDR performs best on problems that benefit from dynamic bucketing. That includes problems where buckets contain violations of \cref{cond:hazard_rate_order}, which in some cases causes the local estimate to deviate outside the range of the CDFs within the bucket, see \cref{fig:example_cmv_violation_population,a:cmv_violation}. Subagging S-IDR helps when the distributions are smooth, but when the problem structure is more complex and hazard rate order violations are present (P5, P6), it is best to use all samples for a single S-IDR fit (or at least keep subsamples large).

A caveat on the comparison: the \PRK estimator is not defined at all evaluation times for all instances, and is therefore omitted from P2 and from portions of P5.

Additional diagnostic plots, as well as an example fit for each of the methods compared, are contained in \cref{a:sidr_simulations}. These also show that S-IDR can find covariate points at which sharp transitions to a different distribution occur. Moreover, diagonal structures are a natural fit for S-IDR's dynamic bucketing, which is best verified visually.

\subsection{Single index model benchmark}\label{ss:index_model_benchmark}
For each of 12 well-known time-to-event data sets from different domains --- survival in breast cancer research, insurance claim height, and others --- we evaluate three index models, alone and combined with S-IDR, by averaging over many train-test splits. We measure performance using four established metrics in survival analysis: the c-index \citep{harrellEvaluatingYieldMedical1982}, integrated Brier score (IBS) \citep{grafAssessmentComparisonPrognostic1999}, D-calibration \citep{haiderEffectiveWaysBuild2020} and 1-calibration \citep{hosmerGoodnessofFitTestsLogistic1985, andresNovelLearningAlgorithm2018}. Results are reported in \cref{fig:index_model_benchmark,fig:index_model_benchmark_calibration}. S-IDR mostly fixes severely miscalibrated fits, while predictions that are already well calibrated rarely improve, and the procedure is especially effective when the base model is structurally too restrictive (e.g., a Cox PH model fit to a discrete outcome). Because sample splitting halves the data available to each model, a few thousand samples are typically required before the calibration gain outweighs the split penalty. A detailed description of the data sets, metrics and benchmarking methodology is provided in \cref{a:benchmark}.

\begin{figure}[ht]
    \includegraphics[
        width=0.4\textwidth,
        alt={Three columns, one for each base model AFT, Cox PH, and survival forest. Height in the column indicates \% change in concordance index, higher is better. AFT mostly around 0, Cox PH mostly slightly negative, survival forest a wider scatter in both directions. Smaller data sets tend to lie below larger data sets.}
    ]{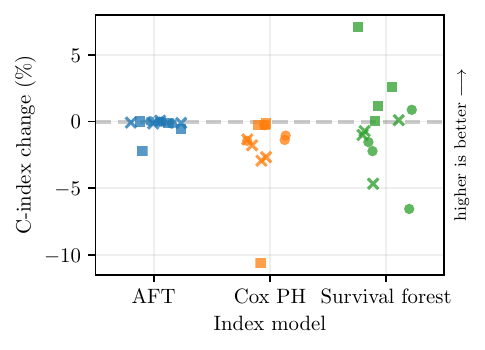}
    \hspace{4em}
    \includegraphics[
        width=0.4\textwidth,
        alt={Three columns, one for each base model AFT, Cox PH, and survival forest. Height in the column indicates \% change in integrated Brier score, lower is better. AFT mostly slightly positive, Cox PH mostly slightly positive, survival forest a wider scatter in both directions. Smaller data sets tend to lie above larger data sets. Each column has one outlier showing a 20\% decrease.}
    ]{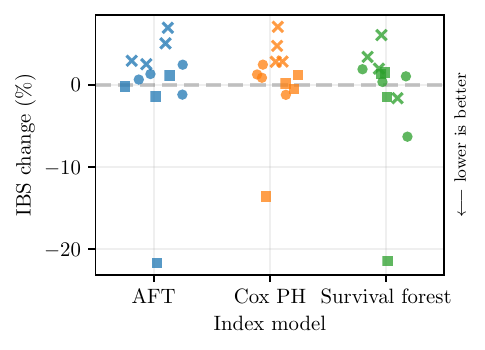}
    \caption{Change in model quality metrics, averaged across \num{50} train-test splits, per data set --- \mx (small, fewer than 500 observed events), \mo (medium, 500--3000 observed events), \ms (large, more than 3000 observed events) --- and index model. Each point represents the average of the metric over the train-test splits of a single data set. Left: c-index. Right: Integrated time-dependent Brier score (IBS).}
    \label{fig:index_model_benchmark}
\end{figure}

\Cref{fig:index_model_benchmark} illustrates how the c-index (a measure of discrimination ability) and IBS (capturing both discrimination and calibration) are affected when S-IDR is stacked as a distributional link via sample splitting. Data sets that are small (fewer than 500 observed events) or medium-sized (up to 3000 observed events) generally exhibit a modest drop in performance, while larger data sets tend to gain from this approach. This pattern arises because sample splitting reduces the number of observations available to each of the two individual models, leaving too few samples to fit either one well. The strongest improvements are seen on (partially) discrete response variables.

\begin{figure}[h]
    \includegraphics[
        width=0.45\textwidth,
        alt={Scatter plot for metric D-calibration with horizontal axis the metric for index model only, vertical axis the metric for index model with S-IDR as distributional link. Diagonal marked to signify no change in metric, with the bottom-left rectangle of values below alpha = 0.25 marked, closer to this rectangle is better. Some points in the rectangle, most points in a slightly uptrending band between 0.4 and 0.9 on the vertical axis. Large datasets with initial miscalibration improve most.}
    ]{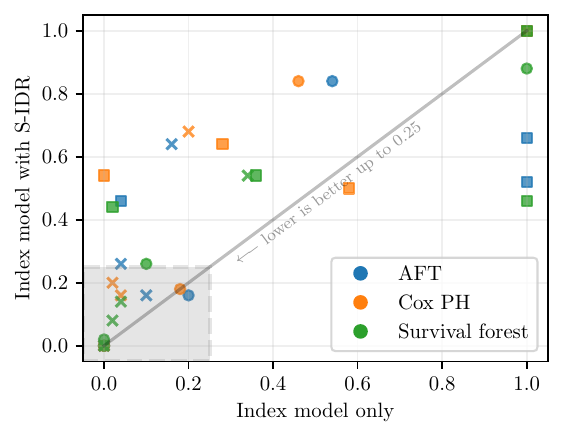}
    \hfill
    \includegraphics[
        width=0.45\textwidth,
        alt={Scatter plot for metric D-calibration with horizontal axis the metric for index model only, vertical axis the metric for index model with S-IDR as distributional link. Diagonal marked to signify no change in metric, with the bottom-left rectangle of values below alpha = 0.25 marked, closer to this rectangle is better. No points near the rectangle, all points in a slightly uptrending band between 0.5 and 0.9 on the vertical axis, starting at 0.2 on the horizontal axis. Large datasets with initial miscalibration improve most.}
    ]{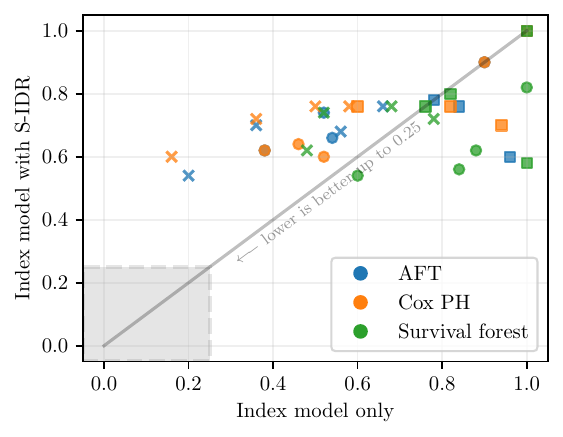}
    \caption{The calibration metrics yield a $p$-value for each of \num{50} train-test splits. For each data set (size markers as in \cref{fig:index_model_benchmark}) and index model, we show the share of train-test splits with a $p$-value below the threshold $\alpha = 0.25$. Left panel: D-calibration. Right panel: 1-calibration.}
    \label{fig:index_model_benchmark_calibration}
\end{figure}

\Cref{fig:index_model_benchmark_calibration} similarly analyses calibration properties through two metrics. Note that most model fits, with or without S-IDR as a distributional link, are miscalibrated according to these metrics: the share of train-test splits for which the test rejects is above the chosen threshold $\alpha$. All cases in which D-calibration does not reject come from small or medium-sized data sets; this is consistent with a calibrated fit, but more likely stems from a lack of power. While applying sample splitting with S-IDR worsens miscalibration in some cases, it tends to improve those cases that are poorly calibrated initially (when the calibration test rejects for $\ge 75\%$ of train-test splits using only the base model), especially when the data set is large (more than 3000 observed events). We conclude that sample splitting with S-IDR is a potential remedy when a test like D-calibration or 1-calibration rejects, suggesting a miscalibrated model fit.

\subsection{Case study: the Model for End-stage Liver Disease (MELD) for organ allocation}\label{ss:case_study}
The Organ Procurement \& Transplantation Network (OPTN) is the U.S.\ government-authorized national system that manages organ transplant policy and allocation, including maintenance of the national waiting list. Adopted in 2022 \citep{OPTN_LiverIntestinalCommittee_2022_ImprovingLiverAllocation}, the Model for End-Stage Liver Disease (MELD) 3.0 score \citep{kimMELD30Model2021} assigns liver transplant candidates a value of at least 6 (in practice typically capped at 40), where higher scores indicate greater urgency for transplantation. Its inputs --- gender and lab measurements such as serum bilirubin and sodium --- relate to mortality from liver failure through known mechanisms:
\begin{equation}\label{eq:meld}
    \text{MELD 3.0} = 1.33 \, (\text{if female}) + 4.56 \ln(\text{bilirubin}) + 0.82 (137 - \text{Na}) - \ldots
\end{equation}
These measurements, and the resulting scores, are objective and verifiable, which is essential for policy use. The score derives from a Cox PH model, which predicts mortality up to 90 days after wait list registration of liver transplant candidates at least twelve years old. This probability can be computed from a MELD score $x$ through the relationship
\[
1 - \hat{F}_x(t) = \hat{S}_x(t) = \hat{S}_0(t)^{\exp(0.17698\,x - 3.56)}, \qquad 6 \le x,
\]
where $\hat{S}_0(t)$ is the base survival estimate of the Cox PH model,
\[
\hat{S}_0(0) = 1, \qquad 
\hat{S}_0(15) = 0.991, \qquad 
\ldots \qquad 
\hat{S}_0(90) = 0.946,
\]
specified in the supplementary material of \citet{kimMELD30Model2021}. Observe that the lifetime distributions are stochastically decreasing in the covariate, in contrast to the increasing convention adopted so far in the paper.
Together with geographical proximity, the MELD score is the primary determinant of the organ assignment among compatible matches, unless the candidate is a minor and has been assigned a special urgency status. Only when the (integer-rounded) MELD scores are equal is waiting time used as a tie breaker. Mortality while on the wait list varies strongly with the initially assigned MELD score: After 30 days, almost all candidates with a score below 15 are still on the list, whereas a significant majority of candidates with a score 35 or above are no longer. The OPTN reports that in 2024, \num{11458} liver transplants occurred under its allocation rules \citep{hrsa_optn_2024_liver_transplants_11458}.
Our aim is to validate the fit of the above Cox model using S-IDR, and to estimate survival curves post transplant, while comparing with alternative nonparametric methods. Finally, we will investigate when two scores should be considered ``tied'' for allocation tie breaking, and show how S-IDR makes new allocation rules possible that rely strictly on an existing, trusted urgency score.
To perform this analysis, we rely on the OPTN Standard Transplant Analysis and Research files, containing data up to December 31, 2025. We will use only MELD scores based on lab measurements, excluding exception scores, to make the results more interpretable, but the analyses below can be performed for any urgency score.

\subsubsection{Model validation}
We validate MELD's fit to newer, out-of-sample data at the 90-day horizon for which it was optimised. MELD has been in use since mid-July 2023, and we restrict to liver transplant candidates joining the wait list in 2024 or 2025. Using MELD scores assigned when joining the wait list as a covariate, we model candidate mortality over time. If MELD is a good (index) model for mortality, the resulting S-IDR fit should be accurate and close to the original model. If it is not, predictions will differ and violations of monotonicity will be reflected in large plateaus in the S-IDR fit.

Time is defined as the number of days since joining the wait list; events are either death or removal from the wait list due to the patient being too ill to receive a transplant. Censoring occurs at the last follow-up date or on December 31, 2025, whichever is earlier, and consequently the survival probability estimate has an inherent downward bias when interpreted as a model for mortality only. Moreover, the noninformative censoring of \cref{cond:noninformative_censoring} may not be satisfied, as patients whose mortality risk increases while on the wait list are more likely to receive a transplant (and be censored). Nonetheless, these considerations are unavoidable and apply to MELD and S-IDR (and to the models we will compare with) equally, so we can evaluate the MELD model fit using S-IDR. We include only first-time liver transplant candidates 18 years or older when they joined the list, who did not receive an exception score and were not waiting for other organs. A total of \num{22049} events are available, with a 93\% censoring rate, of which a random 80\% subsample is selected for this analysis.

\begin{figure}
    \includegraphics[
        width=0.48\textwidth,
        alt={Line plot with as horizontal axis the MELD score and as vertical axis survival probability. Shown are survival curves decreasing in an S-like shape: first slowly, then quickly, then slowly. The MELD 3.0 Cox PH curve is smooth and lower than the other estimators for high MELD scores. The S-IDR, EBM and PKR scores are non-decreasing, the Beran estimate alternatves heavily for high MELD scores.}
    ]{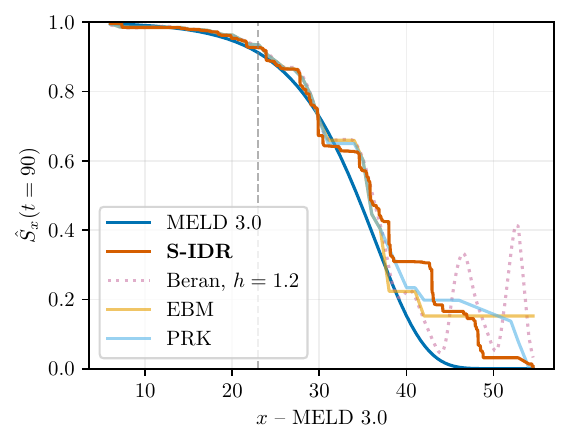}
    \hfill
    \includegraphics[
        width=0.48\textwidth,
        alt={Line plot with as horizontal axis time (one year) and as vertical axis the survival probability. Non-increasing curves start in the top left and decrease ot the bottom right in similar fashion. The MELD 3.0 Cox PH curve stops at 90 days.}
    ]{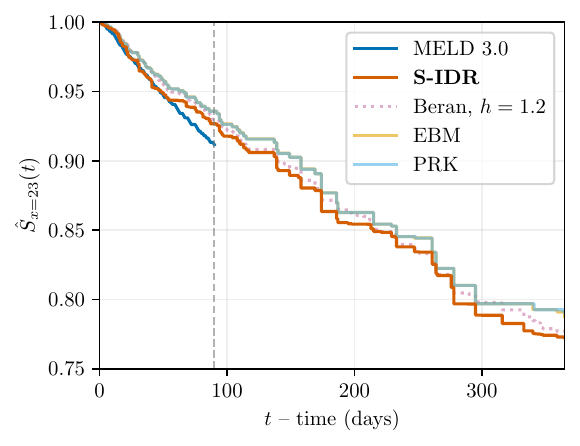}
    \caption{Comparison of estimated survival curves, with the two panels intersecting at the dashed gray lines. Left panel: Estimated survival at $t = 90$ days as a function of the MELD when joining the wait list. Right panel: Estimated survival on the MELD score $x = 23$ as a function of time (MELD 3.0 was computed up to 90 days by its authors).}
    \label{fig:wait_list_survival}
\end{figure}
    
The model fits are visualised for two cross-sections ($t = 90$ days, MELD $x = 23$) in \cref{fig:wait_list_survival}. The survival probability predicted by MELD 3.0 aligns well with the nonparametric S-IDR fit (\num{100} subagging fits using 75\% of the data each); the S-IDR estimate has only a few short plateaus, giving no reason to suspect a major violation of stochastic ordering \cref{cond:stochastic_order}.

The methods \EBM and \PRK, which require a discrete covariate and are therefore applied to the rounded MELD score and visualised with linear interpolation, have longer plateaus than S-IDR or missing values --- especially at high MELD scores, where \PRK reaches 0 and is not defined for all $x$ (masked here by the interpolation). A nonparametric kernel estimate by \citet{beranNonparametricRegressionRandomly1981} (optimised bandwidth $h = 1.25$) exhibits implausible non-monotone behaviour for large MELD scores, illustrating the value of shape constraints in this setting.

Despite MELD being calibrated only for the 90-day horizon, the progression to one year appears natural. The relative gap widens at high MELD scores, possibly reflecting a change in case mix among the most acute patients or improving treatment options for this group.

\Cref{fig:ties} depicts the fit for additional shorter horizons, showing how the Cox PH model compares out-of-sample to S-IDR for $t = 30$ days and even just $t = 10$ days. Since no major disagreement between the models is visible and the Cox PH satisfies proportional hazards, the hazard rate ordering \cref{cond:hazard_rate_order} appears plausible on this data set. Consequently, the Cauchy mean value violations discussed in \cref{ss:from_plain_to_sidr} do not appear to be a concern in this setting, and the bucketing required to implement the approaches of \EBM and \PRK is unlikely to introduce systematic bias.

\subsubsection{Estimating post-transplant mortality}
Violations of the hazard rate order are a more salient concern when examining recipient survival \emph{after} surgery as a function of MELD score, shown in \cref{fig:post_op_survival}. This is of interest if, for example, we want to relate the urgency-based MELD allocation policy to approaches prioritising (also) overall utility, as has been debated for liver wait list management in recent years \citep{luoMELDMetricSurvival2018,lineckerPotentiallyInappropriateLiver2018}.

Under analysis are first-time liver recipients aged at least 18 at the time of the transplant, which took place between January 1, 2006 and December 31, 2025. Censoring occurs at the last attended follow-up, the cohort date or December 31, 2025, whichever is earlier. Recipients who received an exception score are excluded; depicted are lab MELD scores only. Available are \num{141052} events with an 85\% censoring rate, of which a random 80\% subsample is selected for this analysis. S-IDR was applied with \num{100} subsamples of size \num{56421}.

\begin{figure}
    \includegraphics[
        width=0.48\textwidth,
        alt={Line plot with as horizontal axis time and as vertical axis survival probability. Many non-crossing survival curves, corresponding to integer MELD scores in the range 20-50, start in the top left and decrease first quickly, then more slowely. They first diverge, then the scores in the range 20-30 converge after 1.5 years. Scores 20-40 converge almost after 4 years.}
    ]{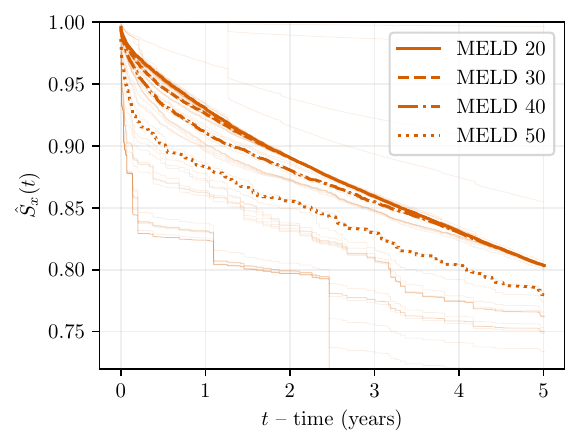}
    \hfill
    \includegraphics[
        width=0.48\textwidth,
        alt={Line plot with as horizontal axis the MELD score and as vertical axis survival probability. EBM and PRK lines coincide and are entirely flat on the left half, while the S-IDR curve initially decreases from a higher level. The Beran curve is not monotone and increases initially. For high MELD scores on the right, S-IDR decreases much more while the other scores remain almost flat.}
    ]{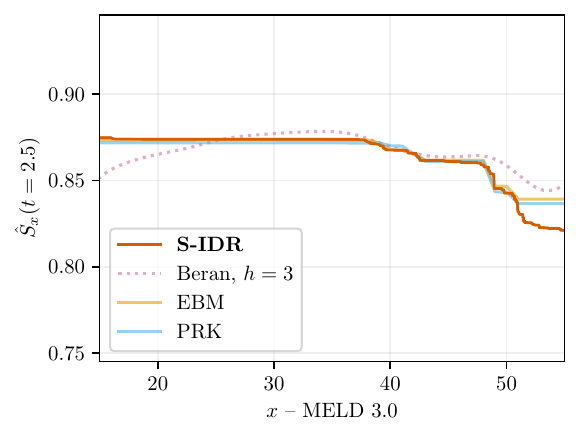}
    \caption{Mortality after liver transplant by last known MELD values just before transplantation.
    Left panel: The S-IDR fit shows how mortality differs in the short term, but after 1.5 years the MELD 20 and MELD 30 mortality is essentially equal. Similarly, the MELD 20 and MELD 40 mortality are no longer meaningfully different after 4 years.
    Right panel: At $t = 2.5$ years we compare the predicted mortality across MELD scores for different methods, linearly interpolating between (missing) integer values at which the discrete covariate groups are centred for the \EBM and \PRK estimators, which are essentially equal for this time horizon.}
    \label{fig:post_op_survival}
\end{figure}

The estimated post-transplant hazards cross: short-term mortality is markedly higher for recipients transplanted with higher MELD scores, in line with expectations, but after roughly three years several of the survival curves draw closer and become nearly indistinguishable, so that the hazards are initially higher for the higher-MELD group and later higher for a lower-MELD group. A plausible mechanism is selective depletion: the most frail individuals among the higher MELD scores die earlier, while comparable mortality in the lower-MELD group is spread over a longer period. Because hazards cross, \cref{cond:hazard_rate_order} is violated, and methods that aggregate observations by rounding MELD into buckets --- such as the \EBM and \PRK procedures --- should be used with caution in this setting. The outer estimates, shown as the top and bottom lines in the left panel and the left- and right-most values in the right panel of \cref{fig:post_op_survival}, are based on few samples and are less reliable, with 99\% of the included patients having a MELD between 18 and 50.

Note that the Cox PH model could not have recovered this effect, because it can only accurately model distributions satisfying the hazard rate order \cref{cond:hazard_rate_order}. Unlike, e.g., the model of \citet{luoMELDMetricSurvival2018}, we did not take into account any covariates besides the MELD score. If that is desired, factors such as age, race/ethnicity, blood type, and primary diagnosis could be incorporated by applying S-IDR using the techniques of \cref{s:extensions}.

\subsubsection{When should scores be considered tied?}
When two liver transplant candidates have the same integer-rounded MELD score, their scores are considered tied and the candidate on the wait list longest will be prioritised. We propose a more refined approach using our S-IDR estimate of \cref{fig:wait_list_survival} for the \num{90}-day mortality, depicted again in \cref{fig:ties}. The estimate is rather flat for certain MELD score ranges, suggesting that the survival probability at \num{90} days is not meaningfully different between such scores.

\begin{figure}[H]
    \centering
    \includegraphics[
        width=0.8\textwidth,
        alt={Line plot with the horizontal axis the MELD score and the vertical axis survival probability. For 10, 30 and 90 days two non-increasing survival curves each: a smooth Cox PH S curve and a closely related S-IDR curve with small plateaus. Highlighted are the MELD scores where the 90-day curve reaches a equally spaced vertical probability thresholds.}
    ]{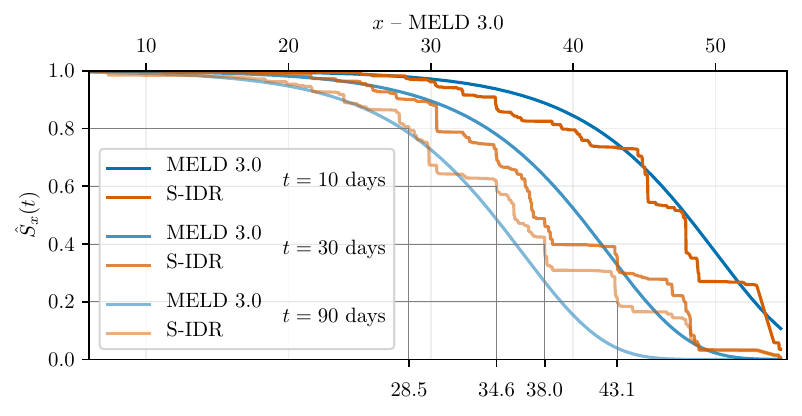}
    \caption{Starting from equally spaced survival probabilities, we identify the MELD score at which S-IDR attains each level. Also shown are the S-IDR fits for $t = 10$ and $t = 30$ days, which could alternatively be used to devise buckets.}
    \label{fig:ties}
\end{figure}
\begin{table}[H]
    \centering
    \captionsetup{aboveskip=0pt, belowskip=10pt}
    \caption{Illustrative partition of MELD scores with boundaries equally spaced in probability space.}
    \label{tab:probability_buckets}
    \begin{tabular}{l|ccccccccc}
        \hline
        $\hat{S}_x(t = 90 \text{ days}) $  & 90\% & 80\% & 70\% & 60\% & 50\% & 40\% & 30\% & 20\% & 10\% \\
        \hline
        $x$ -- MELD 3.0  & 23.9 & 28.5 & 29.9 & 34.6 & 35.8 & 38.0 & 42.7 & 43.1 & 48.2 \\
        \hline
    \end{tabular}
\end{table}

Instead of rounding to the nearest integer, we can use the S-IDR fit to assess more carefully whether the difference in estimated survival probability up to 90 days is above a chosen threshold, e.g., 5\%. Alternatively, we can devise buckets of MELD scores of a chosen granularity that represent equal probability levels, but are not regularly spaced in ``MELD distance''. An illustrative example with 10\% level buckets is depicted in \cref{fig:ties} and \cref{tab:probability_buckets}. The smallest meaningful difference in mortality risk should determine the bucket count. Regardless of MELD's initial 90-day focus, the S-IDR fit gives us an adaptive interval partition of the covariate domain at each time horizon.

Classic isotonic regression for the mean partitions the covariate space using level sets, and S-IDR yields such a partition for each threshold $t$. For the chosen horizon $t$, the buckets equally spaced in terms of $\hat{S}_x(t)$ we propose here are a coarsening of that partition.
The S-IDR estimate itself is subject to uncertainty, which the chosen threshold for a meaningful difference in survival probability should account for. The isotonicity constraint that S-IDR respects ensures that this distributional estimate can lean on the existing consensus the MELD score enjoys.

\section{Discussion}\label{s:discussion}
S-IDR resolves a specific obstruction in the direct survival analogue of isotonic distributional regression, the plain survival IDR. This censored adaptation --- replacing the empirical CDF with the Kaplan--Meier estimator --- is uniformly consistent only under hazard-rate ordering, because the Kaplan--Meier estimator does not satisfy the Cauchy mean value property that any estimator consistent under stochastic dominance must possess (\cref{s:sidr_plain}). S-IDR clamps the Kaplan--Meier value at every partition scale into the interval forced by its sub-partitions, restoring CMV by construction (\cref{s:sidr}). The resulting estimator is uniformly consistent at the minimax rate under only stochastic dominance, admits standard PAV acceleration, accommodates continuous, discrete, and partially ordered covariates without subjective bucketing, and represents atoms in the conditional outcome distribution that smoothness-based likelihood methods such as Cox PH and AFT cannot capture. The estimator carries no tuning parameters, which makes it suitable as a neutral baseline. The MELD case study (\cref{ss:case_study}) puts these properties to work on a real allocation problem and surfaces an empirical finding --- crossing hazards post-transplant --- that the Cox model underlying MELD is structurally unable to detect.

\Cref{th:sidr_modified_consistency} requires restricting the recursive partitions to a minimum block size $\lceil c_n \rceil$, and this restriction is what makes the proof work; classical proof techniques for monotone least squares do not transfer cleanly across the S-IDR clamping recursion. Simulations strongly suggest that the unmodified estimator (with $c_n = 1$) is consistent as well and reaches the same minimax rate as the plain survival IDR, and we conjecture so. A proof is the principal open problem left by this paper.

We do not provide confidence intervals for S-IDR. Pointwise limit distributions for shape-constrained estimators are non-standard at the boundaries of level sets, and the recursion compounds the difficulty. The most promising route is an $m$-out-of-$n$ bootstrap: S-IDR is close in structure to IDR --- itself a nonparametric $M$-estimator amenable to such schemes --- and our subagging procedure (\cref{ss:computational_aspects}) already produces the relevant sub-sample fits, but formal validity remains to be shown.

S-IDR can be applied to forecast evaluation, analogous to existing isotonic-regression-based calibration diagnostics for probability forecasts such as CORP reliability diagrams \citep{dimitriadisStableReliabilityDiagrams2021}. Comparing the S-IDR recalibrated CDFs (or quantiles) to a model's original probability (or quantile) predictions yields insights into their (mis)calibration. Applied to this problem, S-IDR provides a tuning-free alternative to bucketing approaches employed in calibration metrics such as D-calibration \citep{haiderEffectiveWaysBuild2020} or 1-calibration \citep{hosmerGoodnessofFitTestsLogistic1985,andresNovelLearningAlgorithm2018}.

The CMV obstruction is not specific to censored CDFs. Any functional that fails pooling betweenness, such as the conditional variance or expected shortfall, could in principle be paired with the same clamping recursion to produce an isotonic version. Whether the resulting estimators inherit useful statistical guarantees is an open line of research.

\section*{Acknowledgements and disclaimer}
This study used data from the Scientific Registry of Transplant Recipients (SRTR). The SRTR data system includes data on all donors, wait-listed candidates, and transplant recipients in the US, submitted by the members of the Organ Procurement and Transplantation Network (OPTN). The Health Resources and Services Administration (HRSA), U.S. Department of Health and Human Services provides oversight to the activities of the OPTN and SRTR contractors.
The data reported here have been supplied by the Hennepin Healthcare Research Institute (HHRI) as the contractor for the Scientific Registry of Transplant Recipients (SRTR). The interpretation and reporting of these data are the responsibility of the authors and in no way should be seen as an official policy of or interpretation by the SRTR or the U.S. Government. 

Martin Bladt was supported by the Carlsberg Foundation, grant CF23-1096.

We thank Fadoua Balabdaoui and Yuansi Chen for helpful discussions about the self-consistent Kaplan--Meier estimator and related quantities.

\bibliography{references}

\clearpage

\appendix
\crefalias{section}{appendix}
\crefalias{subsection}{appendix}
\pretocmd{\section}{\clearpage}{}{}%

\section{Plain survival IDR: Uniform consistency, proof of \cref{th:sidr_plain_consistency}}\label{a:plain_survival_idr_consistency}
We now prove \cref{th:sidr_plain_consistency}, the uniform consistency result for the plain survival IDR. The proof is strongly inspired by \citet[Theorem 3.3]{moschingMonotoneLeastSquares2020} and proceeds in three steps. First, we adapt the kernel uniform-consistency inequality of \citet{dabrowskaUniformConsistencyKernel1989} to the non-i.i.d.\ censored setting, obtaining a DKW-type bound for the Kaplan--Meier estimator (\cref{th:dabrowska_adaptation}); a covering argument then turns this into a high-probability bound on the supremum process $M_n$ across all index intervals. Second, we establish that, under hazard-rate ordering (\cref{cond:hazard_rate_order}), the mixture target $\bar{F}_{r:s}$ is itself ordered between the conditional CDFs at the interval endpoints (\cref{prop:order}). Third, these two ingredients are plugged into the bias-plus-variance inequality chain of \citet[Theorem 3.3]{moschingMonotoneLeastSquares2020}, yielding the stated rate.

Below, $\lambda(A)$ denotes the Lebesgue measure of a Borel set $A \subset \R$, and we recall that $\rho_n = \log(n)/n$. For $i=1,\dots,k$, let $Y_i \sim F_i$ and $C_i$ be the independent nonnegative event and censoring times, and set
\[
    T_i := \min\{Y_i,C_i\}, \qquad \Delta_i := \ind{Y_i \leq C_i}.
\]
Define
\[
    H_i^1(t) := \pr(T_i > t, \Delta_i = 1), \qquad H_i(t) := \pr(T_i > t), \qquad \Lambda_i(t) := -\int_0^t H_i(s-)^{-1}\,\dd H_i^1(s),
\]
such that $1-F_i(t) := \prodi_{s \leq t}(1-\dd \Lambda_i(s))$ where $\prodi$ denotes the product integral. If $F_i$ has density $f_i$, we write
\[
    h_i(t) := \frac{f_i(t)}{1-F_i(t)}
\]
for the hazard rate of $Y_i$. Moreover, for $j=1,\dots,k$, define the reduction quantities
\begin{align}
    \bar{H}_j^1(t) &= \frac{1}{j}\sum_{i=1}^j H_i^1(t), &
    \bar{H}_j(t) &= \frac{1}{j}\sum_{i=1}^j H_i(t), \nonumber\\
    \bar{\Lambda}_j(t) &= -\int_0^t \bar{H}_j(s-)^{-1}\,\dd \bar{H}_j^1(s), &
    1-\bar{F}_j(t) &= \prodi_{s \leq t}(1-\dd \bar{\Lambda}_j(s)). \label{eq:fbar_k}
\end{align}
Here $\bar{F}_k(\cdot)$ is the counterpart of the empirical average of CDFs $\bar{F}(\cdot)$ in Theorem 4.6 of \citet{moschingMonotoneLeastSquares2020}. Let furthermore $1-\hat{\amsmathbb{F}}_k$ be the Kaplan--Meier estimator evaluated on $(T_i,\Delta_i)$, $i=1,\dots,k$.

\begin{theorem}[DKW inequality for censored data]\label{th:dabrowska_adaptation}
    For a fixed $\tau \in \Rp$, assume that there exists $\eta > 0$ such that $\min_{i=1,\dots,k} H_i(\tau) \geq \eta$. Then there exist universal constants $d_1,d_2$ such that if $k \geq 864/(\varepsilon\eta^5)$, then for all $\varepsilon > 0$,
    \[
        \pr\Big(\sup_{t \leq \tau} |\hat{\amsmathbb{F}}_k(t) - \bar{F}_k(t)| \geq \varepsilon\Big)
        \leq d_1\exp(-d_2k\varepsilon^2).
    \]
\end{theorem}
\begin{proof}
    First, assume that $\varepsilon \in (0, 1)$. The result follows from Theorem 2.1 of \citet{dabrowskaUniformConsistencyKernel1989}, by choosing a suitable distribution for $Z$, conditional sub-survival and survival functions, $K$, and $a_n$. Let
    \[
        I = \Bigl[-\frac{1}{4}, \frac{1}{4}\Bigr], \qquad
        \delta = \frac{3}{4}, \qquad
        a_n = \frac{1}{2}, \qquad
        K(u) = \frac{1}{2}\ind{|u|<1},
    \]
    and let $Z$ be uniformly distributed on $[-1,1]$, so that $g(z)\equiv 1/2$ on $I_\delta=[-1,1]$. The total variation of $K$ is $1$. Partition $[-1/4,1/4)$ into
    \[
        J_i
        :=
        \Bigl[-\frac{1}{4}+\frac{i-1}{2k},\, -\frac{1}{4}+\frac{i}{2k}\Bigr),
        \qquad i=1,\dots,k,
    \]
    and extend this partition $(1/2)$-periodically to $[-3/4,3/4)$. In the notation of \citet{dabrowskaUniformConsistencyKernel1989}, define
    \[
        \widetilde{H}_1(t \mid z) = H_i^1(t),
        \qquad
        \widetilde{H}_2(t \mid z) = H_i(t)
        \quad \text{for } z \in J_i + \frac{m}{2},\ \ m \in \{-1,0,1\},
    \]
    and set $\widetilde{H}_1(t \mid z) = H_1^1(t)$, $\widetilde{H}_2(t \mid z) = H_1(t)$ for $z \in [-1,-3/4) \cup [3/4,1]$. Then, for every $z \in I$,
    \begin{align*}
        \widetilde{H}_{1n}(t,z)
        &=
        a_n^{-1}\int_{-1}^1 \widetilde{H}_1(t \mid u) K(a_n^{-1}(z-u))\,\dd G(u) \\
        &=
        \frac{1}{2}\int_{z-1/2}^{z+1/2} \widetilde{H}_1(t \mid u)\,\dd u
        =
        \frac{1}{2k}\sum_{i=1}^k H_i^1(t),
    \end{align*}
    because $\widetilde{H}_1(t \mid \cdot)$ is $(1/2)$-periodic on $\bigl[-3/4,3/4\bigr)$, and each value $H_i^1(t)$ occupies total Lebesgue measure $1/k$ on any interval of length $1$. Similarly,
    \[
        \widetilde{H}_{2n}(t,z)
        =
        \frac{1}{2}\int_{z-1/2}^{z+1/2} \widetilde{H}_2(t \mid u)\,\dd u
        =
        \frac{1}{2k}\sum_{i=1}^k H_i(t).
    \]
    Moreover,
    \[
        g_n(z)
        =
        a_n^{-1}\int_{-1}^1 K(a_n^{-1}(z-u))\,\dd G(u)
        =
        \frac{1}{2}.
    \]
    Hence the normalised quantities of \citet{dabrowskaUniformConsistencyKernel1989} satisfy
    \[
        \widetilde{H}_1(t,z) = \bar{H}_k^1(t),
        \qquad
        \widetilde{H}_2(t,z) = \bar{H}_k(t),
    \]
    so the corresponding deterministic target distribution is exactly $\bar{F}_k$. The assumptions of Theorem 2.1 of \citet{dabrowskaUniformConsistencyKernel1989} are satisfied with this construction, and the lower bound on $k$ reduces to $k \geq 864/(\varepsilon\eta^5)$, since $a_n=\gamma=1/2$ and $\lambda=1$. The constants $d_1$ and $d_2$ of \citet{dabrowskaUniformConsistencyKernel1989} are absorbed into universal constants after the fixed rescaling by $a_n^{-1}=2$ and $a_n=1/2$.

    We now treat the remaining case $\varepsilon \ge 1$. Both $\hat{\amsmathbb{F}}_k(t)$ and $\bar{F}_k(t)$ are bounded between $0$ and $1$, leaving just the case $\varepsilon = 1$, achievable only when $\bar{F}_k(t) = 0$ and $\hat{\amsmathbb{F}}_k(t) = 1$ (due to $t \le \tau$). But $\bar{F}_k(t) = 0$ implies $F_i(t) = 0$ for $i = 1, \ldots, k$, so we must also have $\hat{\amsmathbb{F}}_k(t) = 0$.
\end{proof}

Let
\[
	M_n := \max_{1 \leq r \leq s \leq m} \sqrt{|O_{r:s}|} \sup_{y \leq \tau}| \KM{r:s}(y) - \bar{F}_{r:s}(y)|,
\]
where $\bar{F}_{r:s}(y)$ is defined as the mixture \eqref{eq:fbar_k}, with the $|O_{r:s}|$ (sub-)survival functions
\[
    H_i^1(t) = \pr(T_i \geq t, \Delta_i = 1 \mid X_i), \quad H_i(t) = \pr(T_i \geq t \mid X_i), \quad i \in O_{r:s}.
\]
Then for any constant $K > 2/d_2$:
\[
	\lim_{n \rightarrow \infty}\pr\left(M_n \leq \log(n)^{1/2}\max\{K^{1/2}, 864 \eta^{-5}\}\right) = 1.
\]
To see this, let $1 \leq r \leq s \leq m$. Define $\kappa_n = \log(n)^{1/2}\max\{K^{1/2}, 864 \eta^{-5}\}$. We apply \cref{th:dabrowska_adaptation} with
\[
    \varepsilon_n = |O_{r:s}|^{-1/2}\kappa_n.
\]
To apply the theorem, we need that
\[
    |O_{r:s}| \geq \frac{864}{\eta^5\varepsilon_n} \iff \sqrt{|O_{r:s}|\log(n)} \geq \frac{864}{\eta^{5}\max\left(K^{1/2}, 864\eta^{-5}\right)},
\]
which holds for all $n \geq 3$, since the right-hand side is bounded from above by $1$. Then,
\begin{align*}
    \pr(M_n > \kappa_n) &\leq \sum_{1 \leq r \leq s \leq m} \pr(\sqrt{|O_{r:s}|} \sup_{y \leq \tau} |\KM{r:s}(y) - \bar{F}_{r:s}(y)| > \kappa_n) \\
    &\leq \frac{n(n+1)}{2} d_1 \exp\left(-d_2\kappa_n^2\right) \\
    &\leq \exp\left(-d_2K\log(n)+2\log(n+1)\right)/2,
\end{align*}
and the upper bound converges to zero as $n \rightarrow \infty$ due to the definition of $K$. For consistency under the hazard rate order assumption, we apply the following result.
\begin{proposition}[Ordering property of \(\bar{F}\)]\label{prop:order}
Assume that each $F_i$ is absolutely continuous and hazard-rate ordered,
\[
    F_1 \lehro F_2 \lehro \cdots \lehro F_k,
\]
that is, the hazard rates satisfy
\[
    h_1(t) \geq h_2(t) \geq \cdots \geq h_k(t)
    \qquad \text{for all } t \geq 0.
\]
Then, for every $j=1, \ldots, k$,
\[
    F_1 \leso \bar{F}_j \leso F_j \leso F_k.
\]
\end{proposition}
\begin{proof}
Let $1 - G_i(t) := \pr(C_i > t)$. Under noninformative censoring,
\[
    H_i(t) = (1-G_i(t))(1-F_i(t)),
    \qquad
    -\dd H_i^1(t) = (1-G_i(t))f_i(t)\,\dd t = H_i(t)h_i(t)\,\dd t.
\]
Hence
\[
    \dd \bar{\Lambda}_j(t)
    =
    -\frac{\dd \bar{H}_j^1(t)}{\bar{H}_j(t)}
    =
    \frac{\sum_{i=1}^j H_i(t)h_i(t)}{\sum_{i=1}^j H_i(t)}\,\dd t.
\]
Thus the integrand of $\bar{\Lambda}_j$ is a weighted average of the individual hazards, with weights proportional to $H_i(t)$:
\[
    \frac{\dd \bar{\Lambda}_j(t)}{\dd t}
    =
    \sum_{i=1}^j w_{i,j}(t) h_i(t),
    \qquad
    w_{i,j}(t) := \frac{H_i(t)}{\sum_{\ell=1}^j H_\ell(t)},
    \qquad i=1,\dots,j.
\]
Since $w_{i,j}(t)\geq 0$ and $\sum_{i=1}^j w_{i,j}(t)=1$, this convex combination lies between the smallest and largest of $h_1(t),\dots,h_j(t)$. Therefore,
\[
    h_1(t) \geq \frac{\dd \bar{\Lambda}_j(t)}{\dd t} \geq h_j(t)
    \qquad \text{for all } t \geq 0.
\]
Integrating over $[0,t]$ yields
\[
    \Lambda_1(t) \geq \bar{\Lambda}_j(t) \geq \Lambda_j(t).
\]
Because the $F_i$ are absolutely continuous,
\[
    1-F_i(t) = \exp\{-\Lambda_i(t)\},
    \qquad
    1-\bar{F}_j(t) = \exp\{-\bar{\Lambda}_j(t)\},
\]
and hence
\[
    F_1(t) \geq \bar{F}_j(t) \geq F_j(t).
\]
The final claim follows from $h_j(t)\geq h_k(t)$, which implies $F_j(t)\geq F_k(t)$ for all $t\geq 0$.
\end{proof}

To prove consistency and the convergence rates, we assume that the estimator is well-defined and that the event in \cref{cond:dense_covariates} occurs. We follow the same steps as in the proof of Theorem 3.3 of \citet{moschingMonotoneLeastSquares2020}. Let $x \in \mathbf{I}_n$, and recall that $\delta_n = C_3 \rho_n^{1/(1+2\alpha)}$. Let $\hat{F}_x$ be any interpolation of $\hat{F}_{x_i}$ and $\hat{F}_{x_{i+1}}$, where $i$ is the unique index in $\{1, \dots, m\}$ such that $\xi_i < x < \xi_{i+1}$, which exists with asymptotic probability one. Let
\begin{align*}
    r(x) & = \min\{k \in \{1, \dots, m\}\colon \xi_k \geq x - \delta_n\}, \\
    j(x) & = \max\{k \in \{1, \dots, m\}\colon \xi_k \leq x\}.
\end{align*} 
Due to \cref{cond:dense_covariates}, these indices are defined with asymptotic probability one, and there are at least $C_2n\delta_n$ observations contained in $[x-\delta_n, x]$. We then have, for any $y \leq \tau$,
\begin{align*}
    \hat{F}_x(y) - F_x(y) & \leq \hat{F}_{x_{j(x)}}(y) - F_x(y) \tag{monotonicity interpolation}\\
    & = \min_{r \leq j(x)} \max_{s \colon s \geq j(x)} \KM{r:s}(y) - F_x(y) \tag{definition \eqref{eq:sidr_plain}}\\
    & \leq \max_{s \geq j(x)} \KM{r(x):s}(y) - F_x(y) \tag{choose $r = r(x)$}\\
    & \leq \frac{M_n}{\sqrt{|O_{r(x):j(x)}|}} + \max_{s \geq j(x)} \bar{F}_{r(x):s}(y) - F_x(y)\\
    & \leq \frac{M_n}{\sqrt{C_2 n \delta_n}} + F_{\xi_{r(x)}}(y) - F_x(y) \\
    & \leq \frac{M_n}{\sqrt{C_2 n \delta_n}} + C_1 \delta_n^\alpha. \tag{H\"older continuity of $F_x$}
\end{align*}
In the fourth step, we add and subtract $\bar{F}_{r(x):s}$ within the maximum, use the definition of $M_n$, and the fact that $|O_{r(x):s}|$ is increasing in $s$. In the fifth step, we use the fact that $|O_{r(x):j(x)}| \geq C_2n\delta_n$ and \cref{prop:order}, where the latter is applicable due to \cref{cond:hazard_rate_order}.

Now we can use that $M_n \leq \log(n)^{1/2}\max(K^{1/2}, 864 \eta^{-5})$ with asymptotic probability one for $K > 2/d_2$, and the upper bound then becomes $C \rho_n^{\alpha / (1+2\alpha)}$ with $C = \max(K^{1/2}, 864 \eta^{-5})/(C_2C_3)^{1/2} + C_1 C_3^\alpha$. An analogous computation shows that the same bound holds for $F_x - \hat{F}_x$. \qed

\section{Cauchy mean value violation of Kaplan--Meier estimator}\label{a:cmv_violation}
We can choose stochastically ordered CDFs $F_i$ and paired censoring distributions $G_i$ such that the Kaplan--Meier estimator evaluated on a pooled sample converges to a $\bar{F}$ outside of the range of the $F_i$, and even violates the Cauchy mean value (CMV) property (\cref{def:cmv}) arbitrarily severely (up to a distance of $1$ between CDFs). Choose $F_1$ and $F_2$ supported on disjoint, interleaved unit intervals across $[0, 4n]$. The geometric censoring schedule $\varepsilon^k$ is such that within each cycle a fraction $1 - O(\varepsilon)$ of the at-risk pool is censored before its events occur --- in particular $G_2$ removes most of group 2 in the very first cycle. The Kaplan--Meier limit $\bar{F}$ stalls at the small mass observed in the early cycles while $F_2$ continues climbing to $1$, so $\max_y\,[F_2(y) - \bar{F}(y)]$ approaches $1$.

For the uniform distribution on an interval $[a,b] \subset \R$, let $\Ucdf{a}{b} := \pr(V \le y)$, $V \sim \mathrm{Unif}[a,b]$, denote the CDF at $y \in \R$. Then we set for small positive $\varepsilon$ and $\delta$, and an integer $n \geq 2$:
\begin{align*}
    F_1(y) &= \frac{1-\delta}{n}\sum_{i=1}^{n} \Ucdf{4(i-1)}{4(i-1)+1} + \delta \Ucdf{4n-1}{4n},\\
    F_2(y) &= \frac{1-\delta}{n}\sum_{i=1}^{n} \Ucdf{4(i-1)+2}{4(i-1)+3} + \delta \Ucdf{4n-1}{4n},\\
    G_1(y) &= \sum_{i=1}^{n-1}\!\bigl(\varepsilon^{2(i-1)} - \varepsilon^{2i}\bigr)  \Ucdf{4(i-1)+3}{4(i-1)+4} + \varepsilon^{2(n-1)} \Ucdf{4n-1}{4n},\\
    G_2(y) &= (1-\varepsilon) \Ucdf{1}{2} + \sum_{i=2}^{n}\!\bigl(\varepsilon^{2(i-1)-1} - \varepsilon^{2(i-1)+1}\bigr)  \Ucdf{4(i-1)+1}{4(i-1)+2} \\
    &\qquad + \varepsilon^{2(n-1)+1} \Ucdf{4n-1}{4n}.
\end{align*}
Then, for all $y \in \Rp$, $F_1(y) \geq F_2(y)$, so that $F_1$ and $F_2$ are stochastically ordered. With the Kaplan--Meier limit $\bar{F}$ defined as in \eqref{eq:fbar_k}, the iterated limit (first $\varepsilon \to 0$, then $n \to \infty$, with $\delta \to 0$ taken at any stage) yields $\max_{y \in \Rp} \left[F_2(y) - \bar{F}(y)\right] \to 1$.
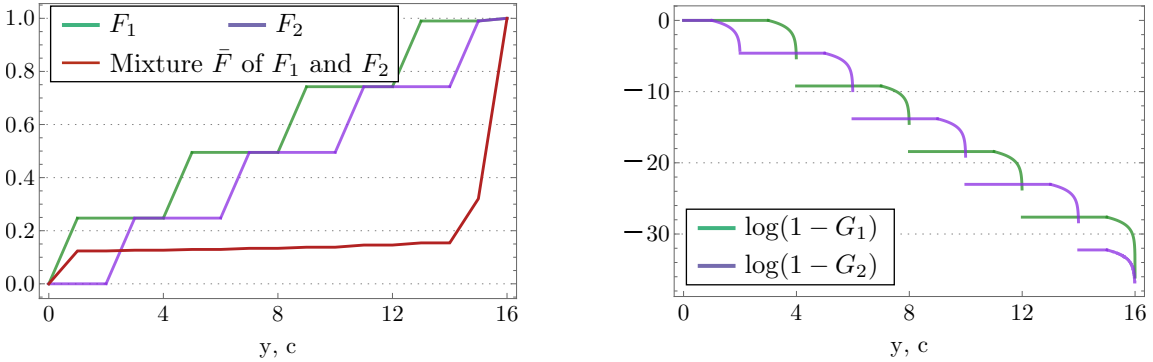
\begin{figure}[h]
    \begin{tikzpicture}[
        alt={Line plot with horizontal axis corresponding to the response / time, vertical axis to the CDF value. Shown are three lines, moving from the bottom-left corner to the top-right, with F1 and F2 alternatingly increasing linearly and being constant. Another line, corresponding to the limit object F bar of the Kaplan-Meier estimator on the mixed sample, briefly is the mean of F1 and F2 but then increases only very slowly and falls far below the other two.}
    ]
        \node[inner sep=0] (img) {\includegraphics[width=0.45\textwidth]{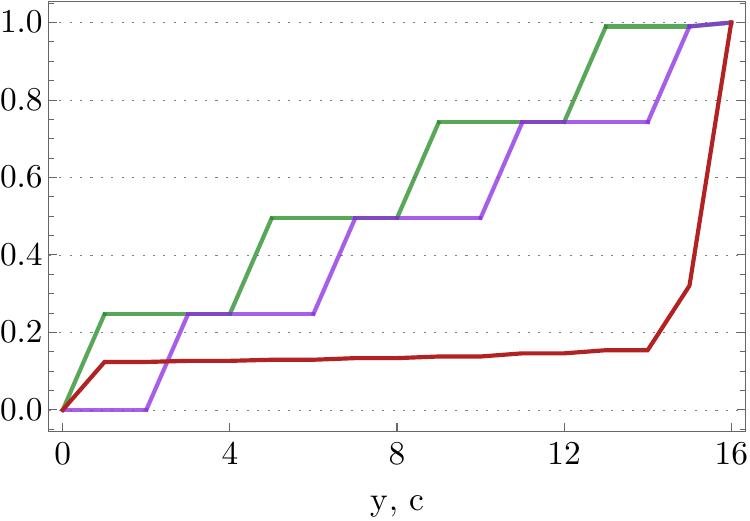}};

        \draw[fill=white, draw=black, line width=0.5pt] ($(img.north west)+(0.6cm,-0.05cm)$) rectangle ++(4.6, -1);
        \draw[ForestGreen, opacity=0.75, line width=1.5pt] ($(img.north west)+(0.75cm,-0.3cm)$) -- ++(0.5, 0);
        \node[right] at ($(img.north west)+(0.75cm+0.5cm,-0.3cm)$) {$F_1$};
        \draw[BlueViolet, opacity=0.75, line width=1.5pt] ($(img.north west)+(0.75cm+2.2cm,-0.3cm)$) -- ++(0.5, 0);
        \node[right] at ($(img.north west)+(0.75cm+2.2cm+0.5cm,-0.3cm)$) {$F_2$};
        \draw[BrickRed, line width=1pt] ($(img.north west)+(0.75cm,-0.3cm-0.5cm)$) -- ++(0.5, 0) node[right, black] {Mixture $\bar{F}$ of $F_1$ and $F_2$};
    \end{tikzpicture}
    \hfill
    \begin{tikzpicture}[
        alt={Line plot with horizontal axis corresponding to the response / time, vertical axis to the log of 1 minus the CDF value. Shown are two lines, moving from the top-left corner to the bottom-right, with G1 and G2 alternatingly decreasing.}
    ]
        \node[inner sep=0] (img) {\includegraphics[width=0.46\textwidth]{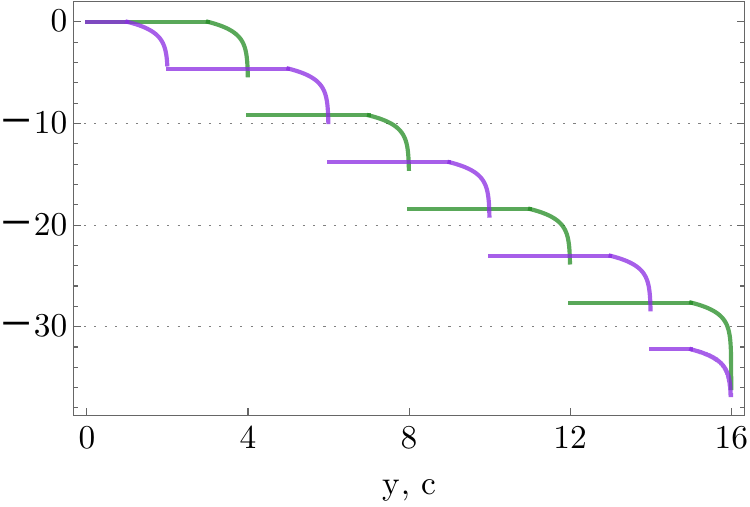}};

        \draw[fill=white, draw=black, line width=0.5pt] ($(img.south west)+(0.85cm,1cm)$) rectangle ++(2.75, 1);
        \draw[ForestGreen, opacity=0.75, line width=1.5pt] ($(img.south west)+(1cm,1.75cm)$) -- ++(0.5, 0);
        \node[right] at ($(img.south west)+(1.5cm,1.75cm)$) {$\log(1 - G_1)$};
        \draw[BlueViolet, opacity=0.75, line width=1.5pt] ($(img.south west)+(1cm,1.75cm-0.5cm)$) -- ++(0.5, 0);
        \node[right] at ($(img.south west)+(1.5cm,1.75cm-0.5cm)$) {$\log(1 - G_2)$};
    \end{tikzpicture}
    \caption{The CDFs $F_1$, $F_2$, $G_1$, $G_2$ chosen as above for $n = 4$, $\varepsilon = 10^{-2}$ and $\delta = 10^{-2}$, and the resulting $\bar{F}$ (see \cref{a:plain_survival_idr_consistency}). Left panel: Distributions of interest with mixture $\bar{F}$ deviating outside the range of $F_1$ and $F_2$. Right panel: Censoring/nuisance distributions $G_1$, $G_2$ alternately change the weighting of CDFs $F_1$, $F_2$.}
\end{figure}

\section{Theoretical details and proofs for S-IDR}\label{a:sidr_proofs}
\subsection{Well-definedness: Proof of \cref{prop:sidr_well_defined}}\label{a:sidr_well_defined}
We prove \Cref{prop:sidr_well_defined} by induction on the index interval length $s - r + 1$ with $1 \leq r \leq s \leq m$.

\begin{proof}
    We fix any $y \in \Rp$ and show \eqref{eq:max_min_inequality} by induction over the interval length $s-r+1$. Equation \eqref{eq:max_min_inequality} holds for all index intervals of length $s - r + 1 = 2$, because $\mathcal{P}_{r:s} = \{ P \}$ with $P = \{ [r : r], [s : s] \}$. Then
    \[
        \min_{I \in P} \RKM{I}(y) = \min_{I \in P} \KM{I}(y) \leq \max_{I \in P} \KM{I}(y) = \max_{I \in P} \RKM{I}(y)
    \]
    because the left-hand side $\max$-$\min$ reduces to a single $\min$, and vice versa for the right-hand side.
    
    We now assume that \eqref{eq:max_min_inequality} holds for all intervals $J = [r : s]$ with $2 \le |J| = s - r + 1 < m \in \amsmathbb{N}$ and that it \emph{does not} hold for some interval of length $m$, that is, there exist partitions $P, P' \in \mathcal{P}_{r:s}$ with $m = s - r + 1$ such that
    \begin{equation}\label{eq:sidr_well_definedness_violation}
        \max_{I \in P} \RKM{I}(y) < \min_{I \in P'} \RKM{I}(y).
    \end{equation}
    Note that we must have $P \cap P' = \emptyset$.

    By our induction hypothesis, we can assume that $P$ and $P'$ were chosen such that $|P| = |P'| = 2$. To see why, consider merging adjacent partition elements $I_l, I_r \in P$ while $|P| > 2$ (we could also have chosen $P'$) and write $K := I_l \cup I_r$. Because $|K| < m$ we have
    \[
        \RKM{K}(y) \leq \max \{ \RKM{I_l}(y), \RKM{I_r}(y) \} < \min_{I \in P'} \RKM{I}(y).
    \]
    Similarly, merging elements in $P'$ maintains inequality \eqref{eq:sidr_well_definedness_violation}.

    Denote the ``split points'' of $P, P'$ by $k, l$ respectively such that $P = \{ [r:k], [k+1:s] \}$ and $P' = \{ [r:l], [l+1:s] \}$ and assume without loss of generality that $k < l$. Then by our induction hypothesis
    \begin{align*}
        \min \{ \RKM{r:k}(y), \RKM{k+1:l}(y) \} &\leq \RKM{r:l}(y) \leq \max \{ \RKM{r:k}(y), \RKM{k+1:l}(y) \},\\
        \min \{ \RKM{k+1:l}(y), \RKM{l+1:s}(y) \} &\leq \RKM{k+1:s}(y) \leq \max \{ \RKM{k+1:l}(y), \RKM{l+1:s}(y) \},
    \end{align*}
    and so using \eqref{eq:sidr_well_definedness_violation} we should have
    \begin{align*}
        &\RKM{r:k}(y) < \RKM{r:l}(y) \leq \RKM{k+1:l}(y),\tag{$[r:k] \in P$, $[r:l] \in P'$}\\
        &\RKM{k+1:l}(y) \leq \RKM{k+1:s}(y) < \RKM{l+1:s}(y).\tag{$[k+1:s] \in P$, $[l+1:s] \in P'$}
    \end{align*}
    Combining these, we find that
    \[
        \RKM{r:l}(y) \leq \RKM{k+1:l}(y) \leq \RKM{k+1:s}(y),
    \]
    which contradicts our assumption \eqref{eq:sidr_well_definedness_violation} because $[k+1:s] \in P$ and $[r:l] \in P'$, so our induction closes.
\end{proof}

\subsection{Equivalent quantile perspective: Proof of \cref{prop:sidr_characteristics}, property \emph{(iii)}}\label{aa:quantile_perspective}
The quantile formulation of S-IDR relies on the following proposition.

\begin{proposition}\label{prop:quantile_perspective}
The quantile function of the self-consistent Kaplan--Meier estimator satisfies $\tilde{q}_{r:s}(\alpha) = \RKM{r:s}^{-1}(\alpha) := \inf \{ z : \RKM{r:s}(z) \ge \alpha \}$. That is, for all $\alpha \in (0,\lim_{y \to \infty} \RKM{r:s}(y)]$,
\begin{equation} \label{eq:quantile_km}
    \RKM{r:s}(\tilde{q}_{r:s}(\alpha)-) < \alpha \leq \RKM{r:s}(\tilde{q}_{r:s}(\alpha)),
\end{equation}
and $\tilde{q}_{r:s}(\alpha) = \infty$ for $\alpha > \lim_{y \to \infty} \RKM{r:s}(y)$.
\end{proposition}
The quantile formulation of S-IDR is well-defined and has the equivalent representation
\[
    \tilde{q}_{r:s}(\alpha) = \min_{r \leq k < s} \max\{\tilde{q}_{r:k}(\alpha), \tilde{q}_{k+1:s}(\alpha)\};
\]
this holds by proofs analogous to those for \cref{prop:sidr_well_defined} and \cref{lem:sidr_one_step_recursion}, which do not make use of any special properties of $\RKM{r:s}$.

\begin{proof}[Proof of \cref{prop:quantile_perspective}]
The proof is by induction over the interval length $s - r + 1$. Equation \eqref{eq:quantile_km} holds by definition for $r=s$. The induction assumption states that
\[
    \RKM{I}(\tilde{q}_{I}(\alpha)) \geq \alpha > \RKM{I}(\tilde{q}_{I}(\alpha)-)
\]
for $\alpha \in (0,\lim_{y \to \infty} \RKM{I}(y)]$ and $\tilde{q}_{I}(\alpha) = \infty$ for $\alpha > \lim_{y \to \infty} \RKM{I}(y)$, for all index intervals $I \subset \{r, \dots, s\}$. Throughout the proof, we define $y$ as
\begin{equation} \label{eq:quantile_km_y}
    y = \KM{r:s}^{-1}(\alpha) = \hat{q}_{r:s}(\alpha),
\end{equation}
and we repeatedly use the fact that $\KM{r:s}(y) \geq \alpha$ for $\alpha \leq \lim_{t \rightarrow \infty} \KM{r:s}(t)$. The proof requires careful treatment of inequalities. The intuition is that the inequalities in \eqref{eq:quantile_km} can be obtained only if an upper bound in the clamping for $\RKM{r:s}$ is combined with a lower bound in the clamping for $\tilde{q}_{r:s}$, or vice versa, and several case distinctions are necessary to ensure that this happens.

\bigskip\noindent
\textbf{Treatment of $\lim_{t \to \infty} \RKM{r:s}(t) < 1$:} We first treat the case $1 \geq \alpha > \lim_{t \to \infty} \RKM{r:s}(t)$. Define
\[
    {}_K\KM{I}(t) = \begin{cases}
        \KM{I}(t), & \ t \leq K, \\
        1, & \ t \geq K,
    \end{cases}
\]
for any index interval $I \subseteq \{r, \dots, s\}$ and some $K > \max\{T_1, \dots, T_n\}$, where we recall that $T_1, \dots, T_n$ are the observed survival times. That is, if the limit of the Kaplan--Meier estimator is strictly less than $1$, we set it to $1$ at a large $K$. Define ${}_K \RKM{I}$ and ${}_K \tilde{q}_I$ as the clamped Kaplan--Meier estimator and its quantile version, with this modification of $\KM{I}$. We then have ${}_K\RKM{I}(t) = \RKM{I}(t)$ if $t < K$, and ${}_K\tilde{q}_{I}(\alpha) = K$ if and only if $\tilde{q}_{I}(\alpha) = \infty$.
Assume that \eqref{eq:quantile_km} holds for these modified estimators. If $\lim_{t \to \infty}\RKM{r:s}(t) =: \alpha_0 < 1$, then ${}_K \RKM{r:s}(K-) < 1 = {}_K \RKM{r:s}(K) $, and
\[
    {}_K\tilde{q}_{r:s}(\alpha) = K = {}_K\RKM{r:s}^{-1}(\alpha), \quad \alpha > \alpha_0.
\]
Consequently, $\tilde{q}_{r:s}(\alpha) = \infty = \RKM{r:s}^{-1}(\alpha)$. For $\alpha \leq \alpha_0$ we have $\RKM{r:s}^{-1}(\alpha) < K$, which implies that ${}_K\RKM{r:s}^{-1}(\alpha) = \RKM{r:s}^{-1}(\alpha)$ and
\[
    \tilde{q}_{r:s}(\alpha) = {}_K\tilde{q}_{r:s}(\alpha) = {}_K\RKM{r:s}^{-1}(\alpha) = \RKM{r:s}^{-1}(\alpha).
\]
Hence, for the remainder of the proof, we can work with the modified estimators ${}_K \KM{I}$, ${}_K \RKM{I}$, and ${}_K \tilde{q}_I$, which allows us to assume that $\lim_{t \to \infty} \RKM{r:s}(t) = 1$ for all $r \leq s$, but we drop the prescript $K$ to lighten the notation.

\bigskip\noindent
\textbf{Case 1:} Assume that the lower bound in $\RKM{r:s}(y)$ is attained, i.e.,
\[
     \KM{r:s}(y) < \RKM{r:s}(y) = \max_{r \leq k < s} \min \{ \RKM{r:k}(y), \RKM{k+1:s}(y)\},
\]
where we use \cref{lem:sidr_one_step_recursion} for the representation of $\RKM{r:s}(y)$. Since $\KM{r:s}(y) \geq \alpha$, there exists an index $k$ such that
\[
    \alpha < \min \{ \RKM{r:k}(y), \RKM{k+1:s}(y)\}, 
\]
implying that
\[
   \hat{q}_{r:s}(\alpha) = \KM{r:s}^{-1}(\alpha) = y \geq \max \{ \RKM{r:k}^{-1}(\alpha), \RKM{k+1:s}^{-1}(\alpha)\} = \max\{\tilde{q}_{r:k}(\alpha), \tilde{q}_{k+1:s}(\alpha)\},
\]
using \eqref{eq:quantile_km_y} for the definition of $y$, and the induction assumption for the last equality. The definition of $\tilde{q}_{r:s}$ implies
\[
    \tilde{q}_{r:s}(\alpha) = \min_{r \leq k < s} \max\{\tilde{q}_{r:k}(\alpha), \tilde{q}_{k+1:s}(\alpha)\} \leq \hat{q}_{r:s}(\alpha),
\]
i.e., the upper bound in $\tilde{q}_{r:s}(\alpha)$ is attained. Now we have
\begin{align} \label{eq:upperbound_quantile_km}
 \RKM{r:s}(\tilde{q}_{r:s}(\alpha)) & = \RKM{r:s}\left(\min_{r \leq k' < s} \max \{ \tilde{q}_{r:k'}(\alpha),  \tilde{q}_{k'+1:s}(\alpha) \}\right) \\
 & \geq \max_{r \leq k < s} \min \left\{\RKM{r:k}\left(\min_{r \leq k' < s} \max \{ \tilde{q}_{r:k'}(\alpha),  \tilde{q}_{k'+1:s}(\alpha) \}\right), \, \RKM{k+1:s}\left(\min_{r \leq k' < s} \max \{ \tilde{q}_{r:k'}(\alpha),  \tilde{q}_{k'+1:s}(\alpha) \}\right) \right\} \nonumber\\
 & = \max_{r \leq k < s} \min \left\{\RKM{r:k}\left(\max \{ \tilde{q}_{r:k_0'}(\alpha),  \tilde{q}_{k_0'+1:s}(\alpha) \}\right), \, \RKM{k+1:s}\left(\max \{ \tilde{q}_{r:k_0'}(\alpha),  \tilde{q}_{k_0'+1:s}(\alpha) \}\right) \right\} \nonumber\\
 & \geq \min \left\{\RKM{r:k_0'}\left(\max \{ \tilde{q}_{r:k_0'}(\alpha),  \tilde{q}_{k_0'+1:s}(\alpha) \}\right), \, \RKM{k_0'+1:s}\left(\max \{ \tilde{q}_{r:k_0'}(\alpha),  \tilde{q}_{k_0'+1:s}(\alpha) \}\right) \right\} \nonumber\\
 & \geq \min \left\{\RKM{r:k_0'}\left(\tilde{q}_{r:k_0'}(\alpha)\right), \, \RKM{k_0'+1:s}\left(\tilde{q}_{k_0'+1:s}(\alpha)\right)\right\}  \nonumber\\
 & \geq \alpha \nonumber,
\end{align}
denoting by $k_0'$ an index for which the inner minimum is attained, and using the induction assumption in the last inequality. To show the inequality for the left limit, let $\varepsilon > 0$. Assume for a contradiction that
\[
    \RKM{r:s}(\tilde{q}_{r:s}(\alpha) - \varepsilon) \geq \alpha.
\]
Then, because $\hat{q}_{r:s}(\alpha) \geq \tilde{q}_{r:s}(\alpha)$,
\[
     \RKM{r:s}(\tilde{q}_{r:s}(\alpha) - \varepsilon) \geq \alpha > \KM{r:s}(\hat{q}_{r:s}(\alpha) - \varepsilon) \geq \KM{r:s}(\tilde{q}_{r:s}(\alpha) - \varepsilon),
\]
which implies that at $z = \tilde{q}_{r:s}-\varepsilon$ we have $\KM{r:s}(z) < \RKM{r:s}(z)$ and $\RKM{r:s}(z)$ takes the value of the lower bound. So we have
\begin{align} \label{eq:lowerbound_quantile_km}
 & \RKM{r:s}(\tilde{q}_{r:s}(\alpha) - \varepsilon) \\
 & = \max_{r \leq k < s} \min \left\{\RKM{r:k}\left(\tilde{q}_{r:s}(\alpha) - \varepsilon\right), \, \RKM{k+1:s}\left(\tilde{q}_{r:s}(\alpha) - \varepsilon\right) \right\} \nonumber\\
 & = \max_{r \leq k < s} \min \left\{\RKM{r:k}\left(\min_{r \leq k' < s} \max \{ \tilde{q}_{r:k'}(\alpha),  \tilde{q}_{k'+1:s}(\alpha) \}-\varepsilon\right), \, \RKM{k+1:s}\left(\min_{r \leq k' < s} \max \{ \tilde{q}_{r:k'}(\alpha),  \tilde{q}_{k'+1:s}(\alpha) \}-\varepsilon\right) \right\} \nonumber\\
 & = \max_{r \leq k < s} \min \left\{\RKM{r:k}\left(\min_{r \leq k' < s} \tilde{q}_{I(k')}(\alpha)-\varepsilon\right), \, \RKM{k+1:s}\left(\min_{r \leq k' < s} \tilde{q}_{I(k')}(\alpha)-\varepsilon\right) \right\} \nonumber\\
 & \leq \max_{r \leq k < s} \min \left\{\RKM{r:k}\left(\tilde{q}_{I(k)}(\alpha)-\varepsilon\right), \, \RKM{k+1:s}\left(\tilde{q}_{I(k)}(\alpha)-\varepsilon\right) \right\} \nonumber\\
 & \leq \max_{r \leq k < s} \min \left\{\RKM{r:k}\left(\tilde{q}_{I(k)}(\alpha)-\right), \, \RKM{k+1:s}\left(\tilde{q}_{I(k)}(\alpha)-\right) \right\} \nonumber\\
 & < \alpha \nonumber,
\end{align}
where $I(\ell)$ is the index interval $J \in \{ [r:\ell], [\ell+1:s]\}$ for which $\tilde{q}_J(\alpha)$ is highest, and the upper bound is $\alpha$ by the induction assumption, because at least one of the expressions in the minimum is bounded by $\alpha$. This contradicts $\RKM{r:s}(\tilde{q}_{r:s}(\alpha) - \varepsilon) \geq \alpha$.

\bigskip\noindent
\textbf{Case 2:} Assume that one of the following two conditions holds,
\[
    \KM{r:s}(y) > \RKM{r:s}(y) \geq \alpha \quad \text{ or } \quad \alpha \leq \KM{r:s}(y) = \RKM{r:s}(y).
\]
In both cases, we know that
\[
    \min_{r\leq k < s} \max\{\RKM{r:k}(y), \RKM{k+1:s}(y)\} \geq \RKM{r:s}(y) \geq \alpha,
\]
Then, for all $k \in \{r, \dots, s - 1\}$, we have
\[
    \max \{ \RKM{r:k}(y), \RKM{k+1:s}(y)\} \geq \alpha,
\]
implying that
\[
    \max_{r\leq k < s} \min \left\{\RKM{r:k}^{-1}(\alpha), \RKM{k+1:s}^{-1}(\alpha)\right\} \leq y,
\]
which, by induction assumption and \eqref{eq:quantile_km_y}, implies
\[
    \max_{r\leq k < s} \min \left\{\tilde{q}_{r:k}(\alpha), \tilde{q}_{k+1:s}(\alpha)\right\} \leq \hat{q}_{r:s}(\alpha).
\]
Hence, the lower bound in $\tilde{q}_{r:s}(\alpha)$ is not attained or attained with equality, meaning that
\[
    \max_{r\leq k < s} \min \left\{\tilde{q}_{r:k}(\alpha), \tilde{q}_{k+1:s}(\alpha)\right\} \leq \tilde{q}_{r:s}(\alpha) = \min\left\{\hat{q}_{r:s}(\alpha), \, \min_{r\leq k < s} \max \left\{\tilde{q}_{r:k}(\alpha), \tilde{q}_{k+1:s}(\alpha)\right\}\right\}.
\]
If $\hat{q}_{r:s}(\alpha)=\tilde{q}_{r:s}(\alpha)$, then 
\[
    \alpha \leq \RKM{r:s}(y) = \RKM{r:s}(\tilde{q}_{r:s}(\alpha))
\]
holds by assumption. Otherwise, i.e., if $\hat{q}_{r:s}(\alpha)>\tilde{q}_{r:s}(\alpha)$, then the upper bound for $\tilde{q}_{r:s}(\alpha)$ is attained. In this case, exactly the same inequalities as in \eqref{eq:upperbound_quantile_km} show that $\RKM{r:s}(\tilde{q}_{r:s}(\alpha))\geq \alpha$.

For the left limit, if $\hat{q}_{r:s}(\alpha)=\tilde{q}_{r:s}(\alpha)$, then $\KM{r:s}(\tilde{q}_{r:s}(\alpha)-\epsilon) < \alpha$ holds by definition of the lower $\alpha$-quantile. Otherwise, i.e., if $\hat{q}_{r:s}(\alpha)>\tilde{q}_{r:s}(\alpha)$, then the upper bound for $\tilde{q}_{r:s}(\alpha)$ is attained. As shown in Case 1, this implies that $\KM{r:s}(z) < \RKM{r:s}(z)$ for $z = \tilde{q}_{r:s}(\alpha) - \varepsilon$, and one can apply the inequalities \eqref{eq:lowerbound_quantile_km} to obtain $\RKM{r:s}(\tilde{q}_{r:s}(\alpha)-\varepsilon) < \alpha$.

\bigskip\noindent
\textbf{Case 3:} Assume that 
\[
    \KM{r:s}(y) \geq \alpha > \RKM{r:s}(y) = \min_{r \leq k < s} \max \{ \RKM{r:k}(y), \RKM{k+1:s}(y)\}.
\]
Then, there exists $k \in \{r, \dots, s-1\}$ such that
\[
    \max \{ \RKM{r:k}(y), \RKM{k+1:s}(y)\} < \alpha,
\]
implying that
\[
    \max_{r\leq k < s} \min \left\{\RKM{r:k}^{-1}(\alpha), \RKM{k+1:s}^{-1}(\alpha)\right\} > y,
\]
which, by induction assumption and by \eqref{eq:quantile_km_y}, implies
\[
    \max_{r\leq k < s} \min \left\{\tilde{q}_{r:k}(\alpha), \tilde{q}_{k+1:s}(\alpha)\right\} > \hat{q}_{r:s}(\alpha).
\]
Hence, the lower bound in $\tilde{q}_{r:s}(\alpha)$ is attained. For the left limit, let $\varepsilon > 0$. Then,
\begin{align*}
    & \RKM{r:s}(\tilde{q}_{r:s}(\alpha)-\varepsilon) \\
    \ & \leq \min_{r \leq k < s} \max \left\{\RKM{r:k}\left(\tilde{q}_{r:s}(\alpha)-\varepsilon\right), \, \RKM{k+1:s}\left(\tilde{q}_{r:s}(\alpha)-\varepsilon\right) \right\} \nonumber\\
    \ & = \min_{r \leq k < s} \max \left\{\RKM{r:k}\left(\max_{r \leq k' < s} \min \{ \tilde{q}_{r:k'}(\alpha),  \tilde{q}_{k'+1:s}(\alpha) \}-\varepsilon\right), \, \RKM{k+1:s}\left(\max_{r \leq k' < s} \min \{ \tilde{q}_{r:k'}(\alpha),  \tilde{q}_{k'+1:s}(\alpha) \}-\varepsilon\right) \right\} \nonumber\\
    \ & = \min_{r \leq k < s} \max \left\{\RKM{r:k}\left(\min \{ \tilde{q}_{r:k_0'}(\alpha),  \tilde{q}_{k_0'+1:s}(\alpha) \}-\varepsilon\right), \, \RKM{k+1:s}\left(\min \{ \tilde{q}_{r:k_0'}(\alpha),  \tilde{q}_{k_0'+1:s}(\alpha) \}-\varepsilon\right) \right\} \nonumber\\
    \ & \leq \min_{r \leq k < s} \max \left\{\RKM{r:k}\left(\tilde{q}_{r:k_0'}(\alpha)-\varepsilon\right), \, \RKM{k+1:s}\left(\tilde{q}_{k_0'+1:s}(\alpha)-\varepsilon\right) \right\} \nonumber\\
    & \leq \max \left\{\RKM{r:k_0'}\left(\tilde{q}_{r:k_0'}(\alpha)-\varepsilon\right), \, \RKM{k_0'+1:s}\left(\tilde{q}_{k_0'+1:s}(\alpha)-\varepsilon\right) \right\} \\
    & \leq \max \left\{\RKM{r:k_0'}\left(\tilde{q}_{r:k_0'}(\alpha)-\right), \, \RKM{k_0'+1:s}\left(\tilde{q}_{k_0'+1:s}(\alpha)-\right) \right\} \\
    & < \alpha,
\end{align*}
where $k_0'$ is an index attaining the inner maximum over $k'$, and we have used the induction assumption in the last step. To show the inequality $\RKM{r:s}(\tilde{q}_{r:s}(\alpha)) \geq \alpha$, assume for a contradiction that $\RKM{r:s}(\tilde{q}_{r:s}(\alpha)) < \alpha$. Then,
\[
    \KM{r:s}(\tilde{q}_{r:s}(\alpha)) \geq \KM{r:s}(\hat{q}_{r:s}(\alpha)) \geq \alpha > \RKM{r:s}(\tilde{q}_{r:s}(\alpha)),
\]
where the first inequality holds because $\tilde{q}_{r:s}(\alpha) \geq \hat{q}_{r:s}(\alpha)$, and the second by \eqref{eq:quantile_km_y}. This implies that the upper bound in $\RKM{r:s}(\tilde{q}_{r:s}(\alpha))$ is attained, i.e.,
\[
    \RKM{r:s}(\tilde{q}_{r:s}(\alpha)) = \min_{r \leq k < s} \max \{ \RKM{r:k}(\tilde{q}_{r:s}(\alpha)), \RKM{k+1:s}(\tilde{q}_{r:s}(\alpha))\}.
\]
We then have
\begin{align}
& \RKM{r:s}(\tilde{q}_{r:s}(\alpha)) \nonumber \\
\ & = \RKM{r:s}\left(\max_{r \leq k' < s} \min \{ \tilde{q}_{r:k'}(\alpha),  \tilde{q}_{k'+1:s}(\alpha) \} \right) \nonumber\\
\ & = \min_{r \leq k < s} \max \left\{\RKM{r:k}\left(\max_{r \leq k' < s} \min \{ \tilde{q}_{r:k'}(\alpha),  \tilde{q}_{k'+1:s}(\alpha) \} \}\right), \, \RKM{k+1:s}\left(\max_{r \leq k' < s} \min \{ \tilde{q}_{r:k'}(\alpha),  \tilde{q}_{k'+1:s}(\alpha) \} \}\right) \right\} \nonumber\\
\ & = \min_{r \leq k < s} \max \left\{\RKM{r:k}\left(\max_{r \leq k' < s} \tilde{q}_{I(k')}(\alpha)\right), \, \RKM{k+1:s}\left(\max_{r \leq k' < s} \tilde{q}_{I(k')}(\alpha)\right) \right\} \nonumber\\
\ & \geq \min_{r \leq k < s} \max \left\{\RKM{r:k}\left( \tilde{q}_{I(k)}(\alpha)\right), \, \RKM{k+1:s}\left(\tilde{q}_{I(k)}(\alpha)\right) \right\} \nonumber\\
\ & \geq \alpha \nonumber,
\end{align}
letting $I(\ell)$ be the index set attaining the inner minimum. This contradicts $\RKM{r:s}(\tilde{q}_{r:s}(\alpha)) < \alpha$.
\end{proof}

The equality of the quantile S-IDR to the inverse of S-IDR in terms of the CDF follows by exactly the same arguments as the proof of Lemma 2.3 of \citet{moschingMonotoneLeastSquares2020} in the uncensored case, since the latter result only uses the max-min and min-max formulae for CDF and quantile estimators.

\subsection{Consistency: Proof of \cref{th:sidr_modified_consistency}}\label{aa:sidr_modified_consistency}
Recall that the observations are $(X_i, T_i, \Delta_i)$ for $i = 1, \dots, n$, with $m$ unique values of the covariate $\xi_1 < \dots < \xi_m$ and corresponding numbers of observations $N_1, \dots, N_m$, where
\[
    N_j = \#\{i\in \{1\dots, n\}\colon X_i = \xi_j\}.
\]
For a constant $n \geq c_n \geq 1$ and an integer $m \geq 1$, define
\[
    i_1 = i_1(c_n) = \min\{j \in \{1,\dots,m\}\colon N_1 + \dots + N_j \geq c_n\},
\]
and, recursively as long as $i_{j-1} < m$ and the below minimum is defined,
\[
    i_j = i_j(c_n) = \min\{k \in \{i_{j-1} + 1,\dots,m\}\colon N_{i_{j-1} + 1} + \dots + N_k \geq c_n\}.
\]
This creates an index set $\mathcal{I}(c_n) = \{i_1, \dots, i_K\}$ with $i_K \leq m$, such that
\[
    \#\{i \in O_{1:n} \mid X_i \le \xi_{i_1} \} \ge c_n \quad \text{and} \quad \#\{i \in O_{1:n} \mid \xi_{i_{j-1}} < X_i \le \xi_{i_j} \} \ge c_n, \quad j \in 2, \ldots, K = K(c_n),
\]
that is, the number of $i$ for which $X_i$ is between $\xi_{i_{j-1} + 1}$ and $\xi_{i_j}$ is at least $c_n$. To impose a minimal partition size in S-IDR, we provide a definition analogous to \cref{def:sidr} but on the coarser grid $\mathcal{I}(c_n) =\{i_1, \dots, i_K\}$, which guarantees that the Kaplan--Meier estimator in the index interval that realises the estimator is always computed over at least $\lceil c_n \rceil$ observations. To this end, define interval start- and end indices as
\[
    \mathcal{I}_R(c_n) := \{1\} \cup \{k + 1\colon k \in \mathcal{I}(c_n), k < m\}, \quad \mathcal{I}_S(c_n) := \mathcal{I}(c_n) \cup \{ m \}
\]
respectively and denote for $r \in \mathcal{I}_R(c_n), s \in \mathcal{I}_S(c_n)$ by $\mathcal{P}_{r:s,\mathcal{I}(c_n)} $ all partitions of $[r:s]$ into at least two index intervals such that the first and last point of each partition element are in $\mathcal{I}_R(c_n)$ and $\mathcal{I}_S(c_n)$, respectively.

We can now define the self-consistent Kaplan--Meier estimator on this grid. Abbreviate $\mathcal{I} = \mathcal{I}(c_n)$ and for $r \in \mathcal{I}_R(c_n)$ and $s \in \mathcal{I}_S(c_n)$, set $\RKM{r:s, \mathcal{I}}(y) = \KM{r:s}(y)$ whenever $\{ i \in \mathcal{I} \colon r < i < s\} = \emptyset$ (meaning, $[r:s]$ cannot be subdivided in the grid) and
\[
    \RKM{r:s, \mathcal{I}}(y) := \clamp\left(\KM{r:s}(y), \max_{P \in \mathcal{P}_{r:s, \mathcal{I}}} \min_{I \in P} \RKM{I,\mathcal{I}}(y), \min_{P \in \mathcal{P}_{r:s}, \mathcal{I}} \max_{I \in P} \RKM{I, \mathcal{I}}(y)\right)
\]
otherwise. The modified S-IDR estimator is defined as
\begin{equation}
    \hat{F}_{\xi_i, \mathcal{I}}(y) := \min_{r \leq i, r \in \mathcal{I}_R(c_n)} \max_{s \geq i, s \in \mathcal{I}_S(c_n)} \RKM{r:s, \mathcal{I}}(y), \quad y \in \Rp, \, 1 \le i \le m.
\end{equation}
Like for the usual IDR, if $x \not\in \{\xi_1, \dots, \xi_m\}$, we define $\hat{F}_{x, \mathcal{I}}(y)$ as any interpolation that satisfies the monotonicity constraints.

The proof requires the following technical lemmas. We reuse the notation from \cref{a:plain_survival_idr_consistency}.

\begin{lemma}\label{lem:continuity_fbar}
Assume that $F_i, G_i$, $i = 1, \dots, k$, are distribution functions such that 
\begin{align} \label{eq:fg_close}
    \sup_{y \leq \tau}\max_{i,j=1,\dots,k} \max(|F_i(y) - F_j(y)|, |G_i(y) - G_j(y)|) \leq \epsilon.
\end{align}
Let $(T_i, \Delta_i) = (\min(Y_i, C_i), \ind{Y_i \leq C_i})$ for independent $Y_i \sim F_i$, $C_i \sim G_i$. If $\min_{i=1, \ldots, k} \pr(T_i > \tau) \geq \eta > 0$, then $\bar{F}_k$ defined at \eqref{eq:fbar_k} satisfies, for all $y \leq \tau$ and $i \in \{1, \dots, k\}$
\[
    |F_i(y) - \bar{F}_k(y)| \leq \frac{4\exp(\eta^{-1})}{\eta^2} \epsilon.
\]
\end{lemma}
\begin{proof}
For $y \leq \tau$, we have
\[
    |H_i(y) - \bar{H}_k(y)| = \Big\vert\frac{1}{k}\sum_{j=1}^k\left[(1-F_i(y))(1-G_i(y)) - (1-F_j(y))(1-G_j(y))\right]\Big\vert \leq 2\epsilon,
\]
Define
\begin{align*}
    & A_i(t) := \pr(T_i \le t, \Delta_i = 1) = \int_{[0,y]} 1-G_i(t-) \,d F_i(t), \\
    & \bar{A}_k(t) := \frac{1}{k}\sum_{j=1}^k\pr(T_j \le t, \Delta_j = 1) = \frac{1}{k}\sum_{j=1}^k\int_{[0,y]} 1-G_j(t-) \,d F_j(t),
\end{align*}
so that $\Lambda_i(y) = \int_{[0, y]} H_i^{-1}(s) dA_i(s)$, and analogously for $\bar{\Lambda}(y)$. The $A_i$ satisfy
\begin{align*}
    & A_i(t) - \bar{A}_k(t) = \frac{1}{k} \sum_{j=1}^k \biggl( \pr(T_i \le t, \Delta_i = 1) - \pr(T_j \le t, \Delta_j = 1) \biggr) \\
    & = \frac{1}{k} \sum_{j=1}^k \left(\int_{[0,y]} G_j(t-)-G_i(t-) \,d F_i(t) - \int_{[0,y]} 1-G_j(t-) \,d (F_j - F_i)(t)\right).
\end{align*}
For each $j$, the first integral on the right-hand side is bounded in $[-\epsilon, \epsilon]$. By Fubini's Theorem, the second integral can be rewritten as
\begin{align*}
    \int_{[0,y]} 1-G_j(t-) \,d (F_j - F_i)(t) & = \int_{[0,y]} \mathbb{E}_{U \sim G_j}[\ind{U \in [t, \infty)}]\,d (F_j - F_i)(t) \\
    & = \mathbb{E}_{U \sim G_j}\left[\int_{[0,y]} \ind{t \le U}\,d (F_j - F_i)(t) \right] \\
    & = \mathbb{E}_{U \sim G_j}\bigl[F_j(\min(y,U)) - F_i(\min(y,U))\bigr] \in [-\epsilon, \epsilon].
\end{align*}
The two bounds above yield
\[
    |A^1_i(y) - \bar{A}^1(y)| \leq 2\epsilon.
\]
Combining the upper bounds for $|H_i - \bar{H}_k|$ and $|A_i - \bar{A}_k|$ implies that
\[
    \sup_{y \leq \tau} |\bar{\Lambda}_k(y) - \Lambda_i(y)| \leq (2\eta^{-1} + 2\eta^{-2})\epsilon \leq \frac{4\epsilon}{\eta^2}.
\]
The product integral solves the equation
\[
    \bar{F}_k(y) = \int_0^y 1 - \bar{F}_k(s-) d\bar{\Lambda}_k(s),
\]
and analogously for $F_i$. So we have
\[
    \sup_{y \leq \tau} |\bar{F}_k(y) - F_i(y)| \leq \int_0^{\tau} \sup_{t \leq s}|\bar{F}_k(t-) - F_i(t)| d\bar{\Lambda}_k(s) + \sup_{y \leq \tau} |\bar{\Lambda}_k(y) - \Lambda_i(y)|.
\]
Gr\"onwall's inequality \citep[Theorem 9]{gillSurveyProductIntegrationView1990} now implies that
\[
    \sup_{y \leq \tau} |\bar{F}_k(y) - F_i(y)| \leq \exp(\eta^{-1}) \sup_{y \leq \tau} |\bar{\Lambda}_k(y) - \Lambda_i(y)| \leq \frac{4\exp(\eta^{-1})\epsilon}{\eta^2}.
\]
\end{proof}

\begin{lemma} \label{lem:order_statistics_close}
Assume that \cref{cond:support_X} holds, and let $X_{(1)} \leq \dots \leq X_{(n)}$ be the order statistics of $X_1, \dots, X_n$. For $0 < k \leq 1$, the event that
\[
    X_{(r)} - X_{(s)} \leq \delta_n
\]
for all $1 \le s \le r \le n$ such that $kC_2 n \delta_n \geq r - s$ has asymptotic probability one.
\end{lemma}
\begin{proof}
The event defined in this lemma is implied by the event in \cref{cond:dense_covariates}, so it must have a higher probability. Indeed, define the interval $\mathbf{I}_n = [X_{(s)}, X_{(r)}]$. If $\lambda(\mathbf{I}_n) > \delta_n$, then $\mathbf{I}_n$ contains an interval of length greater than $\delta_n$ with number of observations at most $k C_2 n \delta_n < C_2 n \lambda(\mathbf{I}_n)$, which is not possible under \cref{cond:dense_covariates}.
\end{proof}

\begin{proof}[Proof of \cref{th:sidr_modified_consistency}]
As in the proof of \cref{th:sidr_plain_consistency}, let $x \in \tilde{\mathbf{I}}_n$ and define $\tilde{\delta}_n = 4C_3 \rho_n^{1/(1+2\alpha)} = 4\delta_n$. We know that with asymptotic probability one, each of the intervals
\[
    [x-j\delta_n, x - (j-1)\delta_n], \ j = 1,\dots, 4,
\]
contains at least $C_2n\delta_n > 0$ observations. Hence, the interval $[x-\tilde{\delta}_n, x]$ contains at least two different points $\xi_{\ell} < \xi_u$ such that 
\[
    \#\{i \in \{1, \dots, n\}\colon \xi_{\ell} \leq X_i \leq \xi_u \} \geq 2C_2n\delta_n.
\]
Choosing the constant $c_n = kC_2n\delta_n$ with $k \in (0,1]$ guarantees that there are indices $r(x) \leq j(x)$ such that $r(x) \in \{1\}\cup \{k+1\colon k \in \mathcal{I}(c_n, m)\}$ and $j(x) \in \mathcal{I}(c_n, m)$ for which $\xi_{r(x)}, \xi_{j(x)} \in [x-4\delta_n, x]$. By construction, we know that
\[
    \RKM{r(x):s, \mathcal{I}}(y) = \KM{a:b}(y)
\]
for some indices $r(x) \leq a \leq b \leq j(x)$ such that 
\[
    \#\{i \in \{1, \dots, n\} \colon \xi_a \leq X_i \leq \xi_b \} \geq \lceil c_n\rceil \geq kC_2n\delta_n.
\]
Let $s_0 \in \mathcal{I} = \mathcal{I}(c_n, m)$ be an index such that
\[
    \RKM{r(x):s_0,\mathcal{I}}(y) = \max_{s \geq j(x)} \RKM{r(x):s,\mathcal{I}}(y).
\]
With $\mathcal{I} \cap [r(x), \infty] = \{i_{\ell}, \dots, i_{K}\}$ and $i_0 := 0$, the definition of the recursive estimator implies that
\begin{align*}
    \RKM{r(x):s_0,\mathcal{I}}(y) & \leq \max\left( \KM{(i_{j-1}+1):i_j}(y)\colon j = \ell + 1, \dots, K \right).
\end{align*}

With the definition of $M_n$ and the fact that $c_n = k C_2 n \delta_n$, we have
\[
    \KM{(i_{j-1}+1):i_j}(y) \leq \max\left( \bar{F}_{(i_{j-1}+1):i_j}(y)\colon j = \ell + 1, \dots, K \right) + \frac{M_n}{\sqrt{kC_2n\delta_n}}.
\]
\Cref{lem:continuity_fbar,lem:order_statistics_close} and \cref{cond:stochastic_order,cond:dense_covariates,cond:regularity_G} imply that for all $j$,
\[
    \bar{F}_{(i_{j-1}+1):i_j}(y) \leq F_{\xi_{i_{j-1}+1}}(y) + \frac{4\exp(\eta^{-1})}{\eta^2}\cdot C_1\delta_n^{\alpha}.
\]
Combining these observations, we can conclude similarly to the proof of \cref{th:sidr_plain_consistency},
\begin{align*}
    \hat{F}_{x,\mathcal{I}}(y) - F_x(y) & \leq \hat{F}_{\xi_{j(x)},\mathcal{I}}(y) - F_x(y) \\
    & = \min_{r \leq j(x), r \in \{1\} \cup \{k + 1\colon k \in \mathcal{I}\}} \max_{s \geq j(x), s \in \mathcal{I}} \RKM{r:s, \mathcal{I}}(y) - F_x(y) \\
    & \leq \max_{s \geq j(x), s \in \mathcal{I}} \RKM{r(x):s, \mathcal{I}}(y) - F_x(y) \\
    & \leq \frac{M_n}{\sqrt{k C_2 n\delta_n}} + \frac{4\exp(\eta^{-1})}{\eta^2}\cdot C_1\delta_n^{\alpha} + \max_{s \geq r(x), s \in \{1\} \cup \{k+1\colon k \in \mathcal{I}\}} F_{\xi_s}(y) - F_x(y)\\
    & \leq \frac{M_n}{\sqrt{k C_2 n \delta_n}} + \frac{4\exp(\eta^{-1})}{\eta^2}\cdot C_1\delta_n^{\alpha} + F_{\xi_{r(x)}}(y) - F_x(y) \\
    & \leq \frac{M_n}{\sqrt{k C_2 n \delta_n}} + 4C_1\delta_n^\alpha\left(1 + \frac{4\exp(\eta^{-1})}{\eta^2}\right),
\end{align*}
which yields the same convergence rate as in \cref{th:sidr_plain_consistency}, with a different constant.
\end{proof}

The crucial difference from the proof of \cref{th:sidr_plain_consistency} is in the inequality
\[
    \max_{s \geq j(x)} \bar{F}_{r(x)s}(y) \leq F_{x_{r(x)}}(y),
\]
which holds for the plain survival IDR under the hazard rate assumption, but generally not under the usual stochastic order. For the recursive estimator, one can only establish an upper bound in terms of $F_{x_{r(x)}}(y)$, plus an additional error of order $\delta_n^{\alpha}$.

\subsection{One-step recursion: Proof of \cref{lem:sidr_one_step_recursion}}\label{a:sidr_one_step_recursion}
We proceed to prove \cref{lem:sidr_one_step_recursion} by showing that there is always a partition of size 2 that reaches the highest lower bound (resp. the lowest upper bound) in the clamping step \eqref{eq:sidr_clamp}, such that we can restrict the $\max$ (resp. the $\min$) to only those $P \in \mathcal{P}_{r:s}$ with $|P| = 2$. The following argument is given for the lower bound; the argument for the upper bound is similar.

\begin{proof}
    Fix $y \in \Rp$ and let $L \in \mathcal{P}_{r:s}$ be any interval partition for which the lower bound $B_l := \max_{P \in \mathcal{P}_{r:s}} \min_{I \in P} \RKM{I}(y)$ is reached, such that $B_l = \min_{I \in L} \RKM{I}(y)$. We show that if $|L| > 2$, we can modify $L$ by merging partition elements while maintaining $B_l = \min_{I \in L} \RKM{I}(y)$.
    
    Let $I^\ast \in L$ be an interval partition element on which the minimum is reached, i.e., $B_l = \RKM{I^\ast}(y)$. Let $I' \in L$ be an adjacent interval partition element to $I^\ast$ and consider the partition $L' = L \setminus \{I^\ast, I'\} \cup \{ I^\ast \cup I' \}$ in which the two partition elements are merged.

    We claim that
    \[
        \min_{I \in L'} \RKM{I}(y) = \min_{I \in L} \RKM{I}(y).
    \]
    If there is more than one such minimising $I^\ast$, we can clearly merge one with a neighbour without affecting the above minima, so assume that $I^\ast$ is the unique minimiser. Because $L$ and $L'$ are the same up to elements $I^\ast, I'$ and $I^\ast \cup I'$, it is sufficient to show that
    \[
        \RKM{I^\ast \cup I'} = \min_{I = I^\ast, I'} \RKM{I}(y).
    \]
    By the definition of $\RKM{}$, the left-hand side is not smaller than the right-hand side. If the left-hand side were larger than the right-hand side, this would contradict our choice of $L$ as the interval partition maximising the lower bound. So, the claim is proven.

    By downward induction on $|L|$, we conclude that for any minimising partition $L$ with $|L| > 2$, there exists a coarser partition $L'$ with $|L'| = 2$ and $B_l = \min_{I \in L'} \RKM{I}(y)$. As such, we can restrict the outer maximisation in the definition of $B_l$ to only those $P \in \mathcal{P}_{r:s}$ with $|P| = 2$.
\end{proof}

\subsection{Partial orders}\label{aa:partial_orders}

As introduced in \cref{s:extensions}, it is also possible to apply S-IDR directly to data with a multi-dimensional covariate. In this case, not all covariates will have a clear ordering relative to all others, meaning that we only have a partial order on covariate space $\mathcal{X}$. For instance, the \PRK paper predicts the first recurrence of prostate cancer using baseline prostate-specific antigen levels and the Gleason grade, a measure of the aggressiveness of the tumour cells. If a patient has both higher antigen levels and a higher Gleason grade, recurrence is expected to occur sooner, but if one of these metrics is higher and the other is lower than those of another patient, we have no clear ordering. Since \PRK only applies to discrete covariates, they must assign patients to several partially comparable groups based on their original covariates. We briefly discuss how to adapt S-IDR, which can handle continuous covariates, to partially ordered $\mathcal{X}$. The functionality is provided in our Rust and Python implementations accompanying this article.

Denote by $\mathcal{L}$ and $\mathcal{U}$ the lower and upper sets of the unique covariate observations $\xi_1, \dots, \xi_m \in \mathcal{X}$. That is, for $L \in \mathcal{L}$, $\xi \in L$ implies $\xi' \in L$ for all $\xi' \preceq \xi$, and analogously, for all $U \in \mathcal{U}$ and $\xi \in U$, we have that $\xi' \succeq \xi$ implies $\xi' \in U$. We can extend \cref{def:sidr} to
\begin{equation} \label{eq:partial_order_estimator}
    \hat{F}_{\xi_i}(y) = \min_{U \in \mathcal{U}: i \in U} \, \max_{L \in \mathcal{L}:i \in L} \,\,\RKM{L \cap U}(y),
\end{equation}
where the self-consistent Kaplan--Meier estimator is now defined most conveniently by extending from \cref{lem:sidr_one_step_recursion} as
\begin{align*}
    \RKM{S}(y) & := \clamp\left( \KM{S}(y), \max_{(L, U) \in \mathcal{P}_S} \min_{S' = L, U} \RKM{S'}(y), \min_{(L, U) \in \mathcal{P}_S} \max_{S' = L, U} \RKM{S'}(y) \right),
\end{align*}
with the partitions
\[
    \mathcal{P}_S := \{ (S \cap L, S \cap U) \mid L \in \mathcal{L}, U \in \mathcal{U}, L \cup U \supseteq S, L \cap U = \emptyset \}.
\]

Even though the abridged computation over all thresholds $y$, as discussed in \cref{ss:computational_aspects}, also works in the partial order case, the computational costs of the best algorithm we are aware of are exponential in the size of the largest antichain of the partial order graph on $\xi_1, \dots, \xi_m$. Flow-based algorithms for the PAV framework as described by (the 2022 update to) \citet{stoutFastestKnownIsotonic2019} cannot be applied, because they require a separable convex loss function. The quantity $\RKM{\cdot}(y)$ satisfies the CMV property, but this by itself is not enough to avoid NP-hardness of the calculation of \eqref{eq:partial_order_estimator}, and properties of the underlying Kaplan--Meier estimator $\KM{\cdot}(y)$ will need to be exploited to make progress on faster computation. For small numbers of unique covariate values $\xi_1, \dots, \xi_m$, like the groups in the paper of \PRK, or in special cases such as highly correlated covariates, a direct calculation of \eqref{eq:partial_order_estimator} is still feasible.

Our implementation of S-IDR with the recursive Kaplan--Meier estimator relies on an abridged version of PAVA that is analogous to the one proposed by \citet{henziAcceleratingPoolAdjacentViolatorsAlgorithm2022} and the one in the Supplementary Material of \PRK. However, the proof of validity in these references is not applicable to our case, since \citet{henziAcceleratingPoolAdjacentViolatorsAlgorithm2022} only treat the case of squared error loss, i.e., the mean function, and \PRK considers certain special concave loss functions.

\subsection{Consistency in the single index model framework}\label{aa:index_model_consistency}
We provide a proof sketch for consistency of S-IDR under the single index model described in \cref{s:extensions}, using similar arguments to \citet{henziDistributionalSingleIndex2023}. The main model assumption is as follows. 
\begin{condition}\label{cond:dim_model}
	There exists a family of CDFs $(F_u)_{u \in \R}$ and $\theta \colon \mathcal{X} \to \R$ such that
	\[
        \pr(Y \le y \mid \theta(X) = u) = F_{u}(y),
	\]
	with $F_u \leso F_v$ whenever $u \le v$.
\end{condition}
We denote by $(G_u)_{u\in\mathbb{R}}$ the conditional CDFs of the censoring variable given $\theta(X)$, that is,
\[
    G_u(c) = \pr(C \leq c \mid \theta(X) = u).
\]
Neither in \cref{cond:dim_model} nor in the definition of $G_u$ above do we assume that the index is a sufficient dimension reduction, i.e., $F_{\theta(x)}(y)$ is not necessarily equal to $\pr(Y \le y \mid X = x)$, and analogously for $G_{\theta(x)}$. It is only required that the CDFs $F_u$ are stochastically increasing, and, in \cref{cond:lipschitz} below, that $F_u$ and $G_u$ are sufficiently smooth.

The three conditions below are analogous to the setting with a univariate covariate.

\begin{condition} \label{cond:independence_cond_index}
The outcome and censoring variable are independent conditional on $\theta(X)$,
\[
    Y \independent C \mid \theta(X).
\]
\end{condition}

\begin{condition} \label{cond:index_dense}
The support of $\theta(X)$ is a bounded interval $\mathbf{I}$ and there exist constants $C_4, C_5 > 0$ such that for arbitrary intervals $\mathbf{I}_n \subset \mathbf{I}$,
\[
    \frac{|\{i \in \{1, \dots, n\}\colon \theta(X_{i}) \in \mathbf{I}_n\}|}{n\lambda(\mathbf{I}_n)} \geq C_4 \text{ whenever } \lambda(\mathbf{I}_n) \geq \delta_n = C_5 \rho_n^{1/3}
\]
with asymptotic probability one.
\end{condition}

\begin{condition}\label{cond:positive_prob_index}
There exist $\tau \in \Rp$ and $\eta > 0$ such that for all $u \in \mathbf{I}$,
\[
    \pr(T \leq \tau \mid \theta(X) = u) \leq 1 - \eta.
\]
\end{condition}

As in \citet{henziDistributionalSingleIndex2023}, we assume Lipschitz continuity of the conditional distribution functions, and in our setting also of the CDFs of the censoring variable.

\begin{condition} \label{cond:lipschitz}
There exist constants $C_6, C_7 > 0$ such that for all $u, v \in \mathbf{I}$ and $y \leq \tau$
\[
     |F_u(y) - F_v(y)| \leq C_6|u-v|, \qquad |G_u(y) - G_v(y)| \leq C_7|u-v|.
\]
\end{condition}

Finally, we require that the approximation $\hat{\theta}$ of $\theta$, up to a monotone transformation, converges at a sufficient rate.
\begin{condition}\label{cond:index_model_consistency}
	There exists a strictly increasing $g: \R \to \R$ and $C_8 > 0$ such that
	\[
		\lim_{n \to \infty} \pr\left(\sup_{x \in \mathcal{X}} |g(\hat{\theta}(x)) - \theta(x)| \ge C_8 \rho_n^{1/3}\right) = 0.
	\]
\end{condition}
This assumption could be relaxed to let $g$ vary with $\hat{\theta}$ (and with $n$); we assume the stricter \cref{cond:index_model_consistency} for simplicity. Note that $g$ is only applied to $\hat{\theta}$, never to $\theta$; the proof below goes through verbatim with $g(\hat{\theta})$ in place of $\hat{\theta}$ throughout, so $g$ does not interact with the conditions on $\theta$. It does interact with \cref{cond:index_dense}: the consistency of $g(\hat{\theta})$ --- and not of $\hat{\theta}$ itself --- is what propagates the asymptotic denseness from $\theta(X_i)$ to the pseudo-covariates $g(\hat{\theta}(X_i))$ (see \cref{lem:dense_in_index}).

Our proof is for a setting with sample splitting, which is how we apply the method in practice. The total number of observations is $n$ as before, but we split into two parts of size $n_1 = \lfloor \gamma n \rfloor$ and $n_2 = n - n_1$ for some $\gamma \in (0, 1)$.
The samples $((X_i, T_i, \Delta_i))_{i = 1}^{n_1}$ are used to estimate the index model $\hat{\theta}$, and the samples $((\hat{\theta}(X_i), T_i, \Delta_i))_{i = n_1 + 1}^{i=n_1 + n_2}$ for S-IDR. To prove consistency for S-IDR, we again assume that the estimator is constructed with a minimal block size $c_n$ in the index space on values $\{(\hat{\theta}(X_i))_{i = n_1 + 1}^{i=n_1 + n_2}\}$, similar to how blocks are chosen in \cref{aa:sidr_modified_consistency}. We denote the resulting estimator by $\hat{F}_{\hat{\theta}(x),\mathcal{I}(c_n)}(y)$, with any interpolation for the $x$ between the block anchors that respects the monotonicity assumptions.

\begin{theorem}[S-IDR with a single index]\label{th:sidr_single_index_consistency}
Assume that $n_1 = \lfloor \gamma n \rfloor$ and $n_2 = n - n_1$. Under \cref{cond:dim_model,cond:independence_cond_index,cond:index_dense,cond:positive_prob_index,cond:lipschitz,cond:index_model_consistency}, and for $k \in (0,1]$, there are constants $C_9, C_{10} > 0$ such that S-IDR with index estimator $\hat{\theta}$ satisfies
\[
    \mathbb P\left(
        \sup_{x\in \mathcal X_n,\; y\le \tau}
        \left|\hat{F}_{\hat\theta(x),\mathcal{I}(c_n)}(y) - F_{\theta(x)}(y)\right|
        \ge C_9\rho_{n_{\min}}^{1/3}
    \right) \to 0,
\]
as $n_{\min} := \min(n_1, n_2) \rightarrow \infty$, where $c_n = 2 k C_4 \max(C_5, C_8) n_2 \rho_{n_{\min}}^{1/3}$ and
\[
    \mathcal{X}_n := \left\{
        x\in\mathcal X:
        \bigl[\theta(x)\pm C_{10}\rho_{n_{\min}}^{1/3}\bigr] \subseteq \mathbf{I}
    \right\}.
\]
\end{theorem}
\begin{corollary}
    Under the same assumptions as for \cref{th:sidr_single_index_consistency}, if the single index model is well-specified, that is, the index $\theta$ is sufficient in the sense that
    \[
        \pr(Y \le y \mid X) = F_{\theta(X)}(y),
    \]
    then S-IDR is consistent for the true CDFs of $Y \mid X = x$ on $y \in [0,\tau]$ and for $x \in \mathcal{X}_n$.
\end{corollary}
\begin{proof}
The proof follows the same strategy as the proof for the distributional single index model without censoring \citep[p.~503]{henziDistributionalSingleIndex2023}, combined with arguments from the proofs of \cref{th:sidr_plain_consistency} and \cref{th:sidr_modified_consistency}. We only show how to derive an upper bound on the estimation error; the arguments for the lower bound are analogous. For fixed $x$, let $X_j \in \{X_{n_1+1}, \dots, X_n\}$ be such that $g(\hat{\theta}(X_j)) \leq g(\hat{\theta}(x))$. Then,
\[
    \hat{F}_{\hat{\theta}(x),\mathcal{I}(c_n)}(y) - F_{\theta(x)}(y) \leq \hat{F}_{\hat{\theta}(X_j),\mathcal{I}(c_n)}(y) - F_{\theta(x)}(y) = \hat{F}_{g(\hat{\theta}(X_j)),\mathcal{I}(c_n)}(y;g) - F_{\theta(x)}(y),
\]
where $\hat{F}_{g(\hat{\theta}(X_j)),\mathcal{I}(c_n)}(y;g)$ denotes S-IDR applied with pseudo-covariates $g(\hat{\theta}(X_i))$, $i = n_1 + 1, \dots, n$. The equality $\hat{F}_{\hat{\theta}(X_j),\mathcal{I}(c_n)}(y) = \hat{F}_{g(\hat{\theta}(X_j)),\mathcal{I}(c_n)}(y;g)$ holds because $g$ is strictly increasing and the estimator only depends on the ordering of $\hat{\theta}(X_i)$. In the proof below, we show that
\[
    \hat{F}_{g(\hat{\theta}(X_j)),\mathcal{I}(c_n)}(y;g) - F_{\theta(x)}(y)
\]
can be bounded from above by a small error. Since $\hat{F}_{g(\hat{\theta}(X_j)),\mathcal{I}(c_n)}(y;g)$ only involves the transformed index $g(\hat{\theta}(\cdot))$, we can assume without loss of generality that $g$ is the identity function.

Denote by $\xi_1, \dots, \xi_m$ the distinct values of $\hat{\theta}(X_{n_1+1}), \dots, \hat{\theta}(X_n)$, and define
\[
    O_{r:s} = \{i \in \{n_1+1, \dots, n\}\colon \xi_r \leq \hat{\theta}(X_i) \leq \xi_s\}.
\]
We let $\KM{r:s}(y)$ be the Kaplan--Meier estimator over observations $(T_i, \Delta_i)$ with indices $i \in O_{r:s}$.
Define $\bar{F}_{\theta,r:s}(y)$ as the mixture in \eqref{eq:fbar_k}, over the (sub-)survival functions
\[
    H_i^1(t) = \pr(T > t, \Delta = 1 \mid \theta(X) = \theta(X_i)), \quad H_i(t) = \pr(T > t \mid \theta(X) = \theta(X_i)), \quad i \in O_{r:s},
\]
and similarly $\bar{F}_{\hat{\theta},r:s}(y)$ as the mixture with the (sub-)survival functions
\[
    H_i^1(t) = \pr(T > t, \Delta = 1 \mid \theta(X) = \hat{\theta}(X_i)), \quad H_i(t) = \pr(T > t \mid \theta(X) = \hat{\theta}(X_i)), \quad i \in O_{r:s}.
\]
By exactly the same steps as in the proof of \cref{th:sidr_plain_consistency}, with asymptotic probability one,
\[
	M_{n_2} := \max_{1 \leq r \leq s \leq m} \sqrt{|O_{r:s}|} \sup_{y \leq \tau}| \KM{r:s}(y) - \bar{F}_{\theta,r:s}(y)|
\]
is bounded from above by $\log(n_2)^{1/2}\max\{K^{1/2},864\eta^{-5}\}$ for any $K > 2/d_2$, where $d_2$ is the constant from \cref{th:dabrowska_adaptation}. Furthermore, $\hat{\theta}(X_i)$, $i = n_1 + 1, \dots, n$, are dense in the interval $\mathbf{I}$.

\begin{lemma}\label{lem:dense_in_index}
Under \cref{cond:index_dense} and \cref{cond:index_model_consistency}, the event that for arbitrary intervals $\mathbf{I}_n \subseteq \mathbf{I}$,
\[
    \frac{|\{i \in \{n_1+1, \dots, n\}\colon \hat{\theta}(X_i) \in \mathbf{I}_n\}|}{n_2\lambda(\mathbf{I}_n)} \geq \frac{C_4}{2} \text{ whenever } \lambda(\mathbf{I}_n) \geq 4\max(C_5, C_8)\rho_{n_{\min}}^{1/3},
\]
has asymptotic probability one.
\end{lemma}
\begin{proof}
Assume that $\lambda(\mathbf{I}_n) \geq 4\max(C_5, C_8)\rho_{n_{\min}}^{1/3}$. Then $\mathbf{I}_n$ contains the subinterval
\[
    \tilde{\mathbf{I}}_n :=  \left\{u \in \mathbf{I}_n \colon \inf_{v \not\in \mathbf{I}_n}|u-v| \geq C_8\rho_{n_2}^{1/3}\right\},
\]
which has length $\lambda(\tilde{\mathbf{I}}_n) \geq \frac12 \lambda(\mathbf{I}_n)$ and $\lambda(\tilde{\mathbf{I}}_n) \ge C_5 \rho_{n_2}^{1/3}$. On the event of \cref{cond:index_dense}, it therefore satisfies
\[
    |\{i \in \{n_1+1, \dots, n\}\colon \theta(X_i) \in \mathbf{\tilde{I}}_n\}| \geq C_4 n_2\lambda(\mathbf{\tilde{I}}_n) \ge \frac{C_4}{2} n_2 \lambda({\mathbf{I}_n}).
\]
On the event of \cref{cond:index_model_consistency}, we know that $\sup_{x \in \mathcal{X}}|\hat{\theta}_n(x) - \theta(x)| \leq C_8\rho_{n_2}^{1/3}$, so $\hat{\theta}(x) \in \mathbf{I}_n$ if $\theta(x) \in \mathbf{\tilde{I}}_n$. This yields the desired inequality
\begin{align*}
     |\{i \in \{n_1+1, \dots, n\}\colon \hat{\theta}(X_i) \in \mathbf{I}_n\}| \geq \frac{C_4}{2} n_2 \lambda(\mathbf{I}_n).
\end{align*}
\end{proof}

Use $\xi_1, \ldots, \xi_m$ to denote the unique and sorted values of $(\hat{\theta}(X_i))_{i=n_1+1}^{n_1+n_2}$ and let $\mathcal{I}(c_n) \subseteq \{1, \dots, m\}$ be the coarse index set on which the recursive estimator $\RKM{r:s}(y)$ is defined, as in \cref{aa:sidr_modified_consistency}. One can now follow the same steps as in the proof of \cref{th:sidr_modified_consistency} with $C_3$ replaced by $4\max(C_5, C_8)$, $C_2$ replaced by $C_4/2$, $\alpha = 1$, and $c_n = 2 k C_4 \max(C_5, C_8)n_2\rho_{n_{\min}}^{1/3}$, now working on the event of \cref{lem:dense_in_index}. Assume that $x$ is such that $[\hat{\theta}(x) - 4 \cdot 4\max(C_5, C_8)\rho_{n_{\min}}^{1/3}, \hat{\theta}(x)] \subseteq \mathbf{I}$. By the same arguments as for \cref{th:sidr_modified_consistency},
\[
    \KM{(i_{j-1}+1):i_j}(y) \leq \max\left(\bar{F}_{\theta,(i_{j-1}+1):i_j}(y)\colon j = \ell + 1, \dots, K \right) + \frac{M_{n_2}}{\sqrt{c_n}},
\]
where indices $j(x), r(x), i_{\ell}, \dots, i_K \in \mathcal{I}(m, c_n)$ are defined as in the proof of \cref{th:sidr_modified_consistency}. By \cref{cond:lipschitz}, \cref{cond:index_model_consistency} and \cref{lem:order_statistics_close}, \cref{lem:continuity_fbar}, also
\begin{align*}
    |\bar{F}_{\theta,(i_{j-1}+1):i_j}(y) - \bar{F}_{\hat{\theta},(i_{j-1}+1):i_j}(y)| \leq \frac{4\exp(\eta^{-1})}{\eta^2} \max(C_6, C_7)C_8\rho_{n_1}^{1/3}
\end{align*} 
with asymptotic probability one. Furthermore, since the pseudo-covariates $\hat{\theta}(X_i)$ are arranged such that $\hat{\theta}(X_i) \leq \hat{\theta}(X_j)$ for $i \leq j$, we can further bound
\[
    \bar{F}_{\hat{\theta}, (i_{j-1}+1):i_j}(y) \leq F_{\xi_{i_{j-1}+1}}(y) + \frac{4\exp(\eta^{-1})}{\eta^2}\max(C_6, C_7)\cdot 4\max(C_5, C_8)\rho_{n_{\min}}^{1/3}.
\]
Combining the above bounds yields
\begin{align*}
    & \hat{F}_{\xi_{j(x)},\mathcal{I}(c_n)}(y) - F_{\theta(x)}(y) \\
    & = \min_{r \leq j(x), r \in \{1\} \cup \{k + 1\colon k \in \mathcal{I}\}} \max_{s \geq j(x), s \in \mathcal{I}} \RKM{r:s, \mathcal{I}}(y) - F_{\theta(x)}(y) \\
    & \leq \max_{s \geq j(x), s \in \mathcal{I}} \RKM{r(x):s, \mathcal{I}}(y) - F_{\theta(x)}(y) \\
    & \leq \max\left(\bar{F}_{\theta,(i_{j-1}+1):i_j}(y)\colon j = \ell + 1, \dots, K \right) + \frac{M_{n_2}}{\sqrt{c_n}} - F_{\theta(x)}(y)\\
    & \leq \max\left(\bar{F}_{\hat{\theta},(i_{j-1}+1):i_j}(y)\colon j = \ell + 1, \dots, K \right) + \frac{4\exp(\eta^{-1})}{\eta^2} \max(C_6, C_7)C_8\rho_{n_1}^{1/3} + \frac{M_{n_2}}{\sqrt{c_n}} - F_{\theta(x)}(y)\\
    & \leq  \max\left(F_{\xi_{i_{j-1}+1}}(y)\colon j = \ell + 1, \dots, K \right) + \frac{4\exp(\eta^{-1})}{\eta^2}\max(C_6, C_7)\cdot 4\max(C_5, C_8)\rho_{n_{\min}}^{1/3} \\
    & \qquad +\frac{4\exp(\eta^{-1})}{\eta^2} \max(C_6, C_7)C_8\rho_{n_1}^{1/3} + \frac{M_{n_2}}{\sqrt{c_n}} - F_{\theta(x)}(y)\\
    & \leq \frac{4\exp(\eta^{-1})}{\eta^2}\max(C_6, C_7)\cdot 4\max(C_5, C_8)\rho_{n_{\min}}^{1/3} \\
    & \qquad +\frac{4\exp(\eta^{-1})}{\eta^2} \max(C_6, C_7)C_8\rho_{n_1}^{1/3} + \frac{M_{n_2}}{\sqrt{c_n}} + F_{\xi_{r(x)}}(y) - F_{\theta(x)}(y).
\end{align*}
Now $\xi_{r(x)} = \hat{\theta}(X_i)$ for some $i$ such that $|\hat{\theta}(X_i) - \theta(x)| \leq 16\max(C_5, C_8)\rho_{n_{\min}}^{1/3}$. Since $M_{n_2}$ is of order $\log(n_2)^{1/2}$, $c_n = 2kC_4\max(C_5, C_8)n_r\rho_{n_{\min}}^{1/3}$, and $n_1/n_2 \rightarrow \gamma / (1-\gamma)$, the overall error is of order $\rho_{n_{\min}}^{1/3}$. For the domain $\mathcal{X}_n$, notice that the arguments above assume
\[
    [\hat{\theta}(x) - 4 \cdot 4\max(C_5, C_8)\rho_{n_{\min}}^{1/3}, \hat{\theta}(x)] \subseteq \mathbf{I}_n,
\]
and under the event in \cref{cond:index_model_consistency}, a sufficient condition for this is that 
\[
    [\theta(x) \pm (C_8\rho_{n_1}^{1/3} + 4 \cdot 4\max(C_5, C_8)\rho_{n_{\min}}^{1/3}) ] \subseteq \mathbf{I}.
\]
\end{proof}

\section{Simulation details}\label{a:sidr_simulations}
This section elaborates on the performance comparison of \cref{ss:simulations}, detailing the problems on which \cref{fig:simulation} is based and the choice of metric.

\paragraph{Problem instances}
Each of the problems is described in detail in \cref{tab:simulation_problems} and visualised in \cref{fig:problem_grid}, with an example estimate from each method (S-IDR, \EBM, \PRK) shown in \cref{fig:problem_method_grid}. For each problem we specify
\begin{itemize}
    \item a covariate distribution, which is uniform on an interval $\mathbf{I}\subset\mathbb{R}$,
    \item a conditional distribution of the event time $Y\mid X$,
    \item a conditional latent censoring distribution $C^\star\mid X$, independent of $Y \mid X$.
\end{itemize}
A random variable $B \sim \mathrm{Ber}(\pi)$, independent of everything else, determines whether the observation is at risk of being censored:
\[
C \;=\; B\cdot C^\star + (1-B)\cdot(+\infty).
\]
The censoring risk $\pi$ is also specified in \cref{tab:simulation_problems}.

P1 is specified by \eqref{eq:illustration}. P2 is similar to the distribution of \cref{fig:post_op_survival} and contains hazard rate order violations, a high level of censoring and point mass censoring at thresholds, and is fully censored after some threshold. P3 is a collection of uniform distributions in a blocked form, approaching a uniform distribution on the left. P4 is like P3 but discrete, yet approaches the same uniform distribution on the left. P5 is a close analogue to \cref{fig:example_cmv_violation_population}, but with two repetitions of the pattern as specified in \cref{a:cmv_violation} instead of one. P6 repeats the pattern of \cref{a:cmv_violation} four times at changing frequencies, and approaches a point mass on the left.

\begin{table}[h]
    \centering
    \small
    \renewcommand{\arraystretch}{1.35}
    \begin{tabular}{@{}p{0.07\linewidth} p{0.04\linewidth} p{0.04\linewidth} p{0.03\linewidth} p{0.34\linewidth} p{0.34\linewidth}@{}}
        \toprule
        Problem & $\mathbf{I}$ & Dir. & $\pi$ & $Y\mid X=x$ & $C^\star\mid X=x$ \\
        \midrule
        P1 & $[0,10]$ & $\uparrow$ & $0.5$
        & $Y\sim\mathrm{Gamma}\bigl(k=\sqrt{x},\ \theta=\clamp(x,1,6)\bigr)$.
        & $C^\star\stackrel{d}{=}Y$ given $X$. \\
        \addlinespace
        P2 & $[0,1]$ & $\uparrow$ & $1.0$
        & Time-warped $\mathrm{Exp}(1)$: with $t_0=x/2$, $t_e=3/4$, the map
          $y\mapsto y_{\mathrm{eff}}=(y-t_0)/(t_e-t_0) \cdot \mathbb{1}\{y\le t_e\}+(1+y-t_e)\mathbb{1}\{y>t_e\}$
          pushes $\mathrm{Exp}(1)$ forward to $Y$.
        & Mixture (independent of $x$): with prob.\ $\tfrac13$, uniform over the four atoms $\{\tfrac14,\tfrac12,\tfrac34,1\}$; with prob.\ $\tfrac23$, piecewise-uniform on $[0,1]$ placing mass $\tfrac13$ on $[0,t_e]$ and $\tfrac23$ on $(t_e,1]$. \\
        \addlinespace
        P3 & $[0,1]$ & $\downarrow$ & $0.5$
        & Uniform on a union of $r(x)$ equally-spaced sub-intervals of length $(x+1)/(2r)$ that tile $[1-x,2]$ with gaps of equal length, each carrying mass $1/r$.
        & Uniform on $[1-x,2]$. \\
        \addlinespace
        P4 & $[0,1]$ & $\downarrow$ & $0.5$
        & Discrete uniform on the $r(x)$ atoms $y_k=(k/r)(x+1)+(1-x)$, $k=0,\ldots,r-1$.
        & Uniform on $[1-x,2]$ (as in P3). \\
        \addlinespace
        P5 & $[0,1]$ & $\uparrow$ & $1.0$
        & Discrete, two regimes. For $x<\tfrac12$: atoms $\{0,\tfrac12,1\}$ with weights $(0.495,0.495,0.01)$. For $x\ge\tfrac12$: atoms $\{\tfrac14,\tfrac34,1\}$ with weights $(0.495,0.495,0.01)$.
        & Discrete, two regimes. For $x<\tfrac12$: atoms $\{\tfrac38,1\}$ with weights $(0.99,0.01)$. For $x\ge\tfrac12$: atoms $\{\tfrac18,\tfrac58,1\}$ with weights $(0.9,0.099,0.001)$. \\
        \addlinespace
        P6 & $[0,1]$ & $\downarrow$ & $1.0$
        & Continuous; uniform with mass $1/B$ on each of the $B=3$ blocks $[\,1-2/(g+2k),\;1-2/(g+2k+1)\,]$, $k=0,1,2$. As $x\downarrow 0$, $g\to\infty$ and the support concentrates near $1$.
        & Discrete uniform on the $B+1=4$ atoms $\{\,1-2/(g+2k)\,\}_{k=0}^{B}$, each of weight $1/(B+1)$. \\
        \bottomrule
    \end{tabular}
    \caption{Summary of the six synthetic benchmarks. ``Dir.'' indicates whether $x \mapsto Y \mid X$ is isotonically increasing ($\uparrow$) or decreasing ($\downarrow$). The censoring probability $\pi$ is the Bernoulli probability that the latent censoring time $C^\star$ is applied; with probability $1-\pi$ the observation is forced to be uncensored ($C=+\infty$). The resolution $r(x)=2^{\lfloor-\log_2 x^2\rfloor}$ is used for P3, P4, and the group index $g(x)=\lfloor 2/x\rfloor$ for P6.}
    \label{tab:simulation_problems}
\end{table}

\begin{figure}
    \centering
    \includegraphics[
        width=0.98\linewidth,
        alt={A grid of 6 problems visualised in six rows, with three columns containing a scatter plot showing a random sample of the problem instance, a line plot showing the quantiles of the target distribution F, and a line plot showing the nuisance distribution.}
    ]{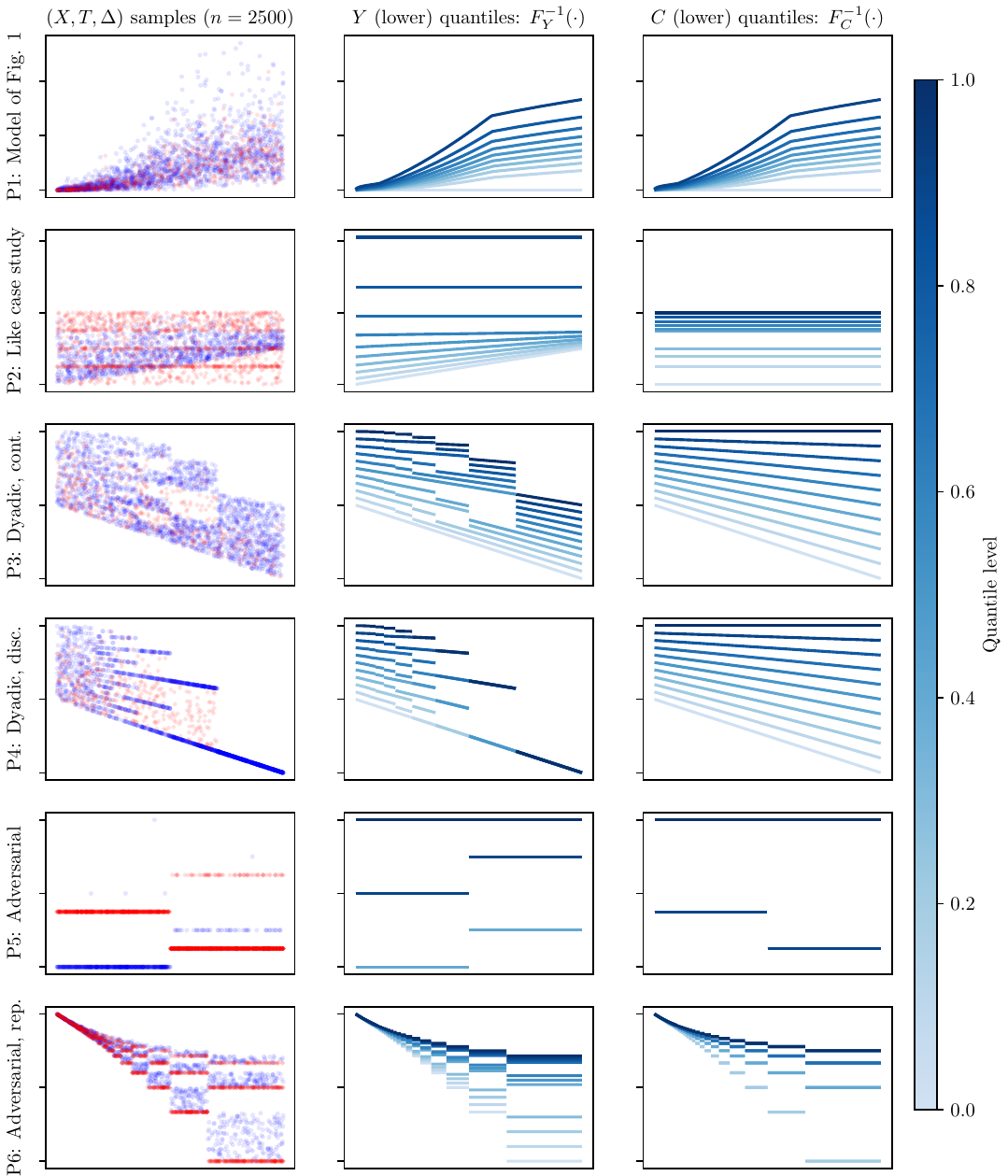}
    \caption{Test problems used in \cref{fig:simulation}, the distributions of interest $F_Y$ stochastically ordered. The left column shows an example instance, with uncensored observations marked blue and censored observations marked red. The two right-most columns show the (lower) quantile levels of the data and censoring distributions.}
    \label{fig:problem_grid}
\end{figure}

\begin{landscape}
  \begin{figure}[p]
    \centering
    \includegraphics[
        width=\linewidth,
        alt={A landscape-oreinted grid of 6 problems visualised in six rows, with mostly similar images in each column. The first column is the ground truth, then follow three column pairs for S-IDR; the EBM method, and the PRK method, once with bandwidth at the n to the minus one third rate and n to the minus one fifth rate.}
    ]{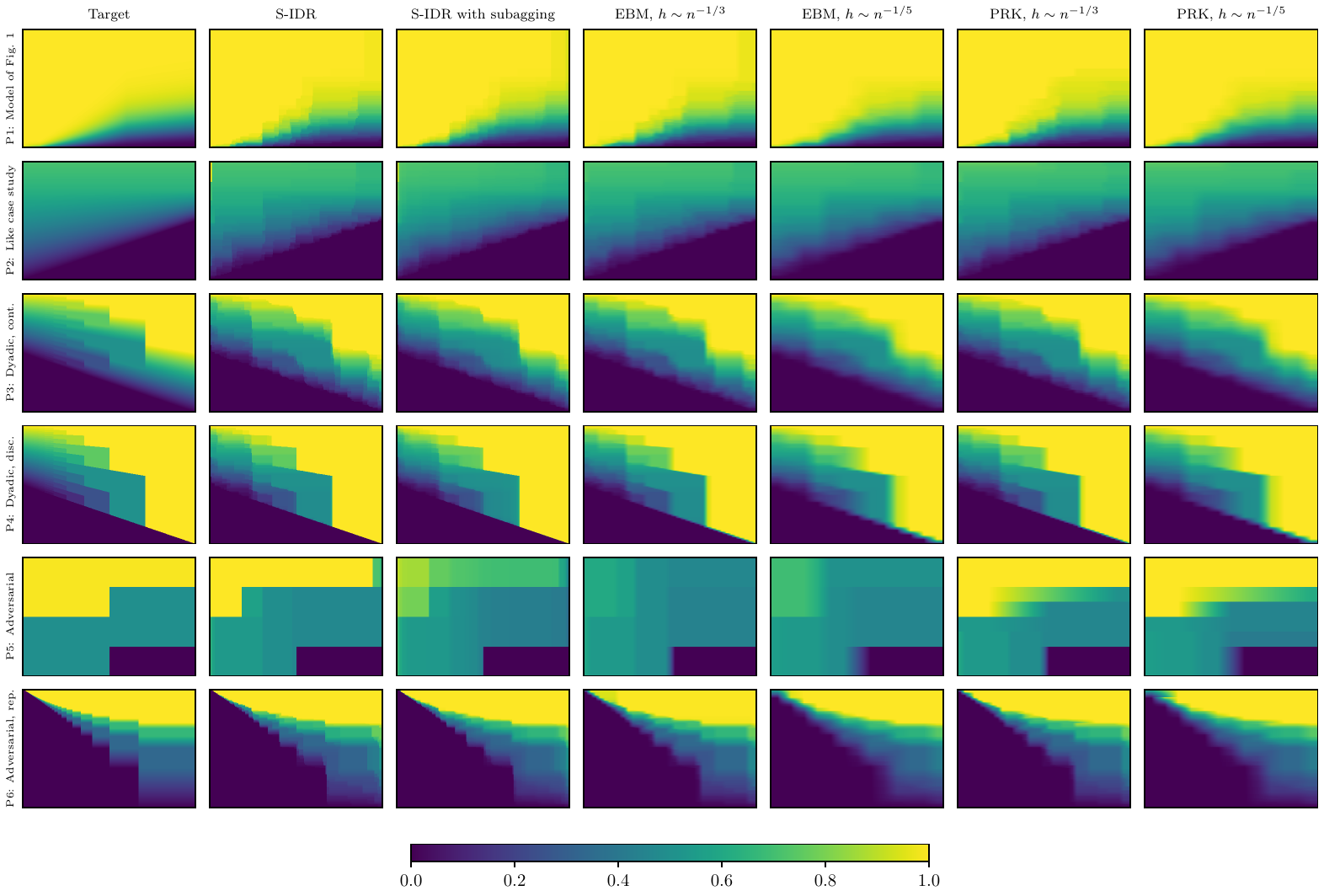}
    \caption{Estimated (sub-)CDFs of an example instance with $n = \num{2000}$ solved using each of the methods benchmarked in \cref{fig:simulation}.}
    \label{fig:problem_method_grid}
  \end{figure}
\end{landscape}

\paragraph{Metrics}
Because the estimates are merely sub-CDFs in general, distances between distributions like earth mover/Wasserstein metrics are not suitable candidates. The discrete nature of some test problems requires care, too. For example, the Kolmogorov--Smirnov test statistic, the maximum distance between (sub-) CDFs, will not converge to $0$ unless atoms in the estimated distribution align precisely. Similarly, statistics like the Cramér-von Mises statistic
\[
    \int (\hat{F}(x) - F(x))^2 dF(x),
\]
are not compatible with the interpolation strategies used by the methods being compared, in the sense that such metrics may not converge to $0$ if $Y$ has atoms. Instead, the chosen metric is an $L_1$ distance between CDFs restricted to some pragmatically chosen range:
\begin{equation}\label{eq:metric}
    \int_\mathcal{Y} |\hat{F}_x(y) - F_x(y)| \ind{y \in S_F} dy,
\end{equation}
where we choose $S_F$ as the convex hull of $\bigcup_{x \in \mathbf{I}} \operatorname{supp}(F_x)$ if it is bounded. When it is not, we take $S_F' := S_F \cap (-\infty, q]$ with $q = q_H = \sup_{x \in \mathbf{I}} H_x^{-1}(0.999)$. This quantity is then averaged over the covariate range.
The metrics are approximated by simulating \num{100} independent samples from each problem instance.

\paragraph{Bucketing choice for discrete methods}
While S-IDR decides its bucketing dynamically, the \EBM and \PRK methods require a bucketing choice. The ideal choice is problem specific, but to understand the sensitivity to this choice, a comparison of bucket counts is displayed in \cref{fig:bucketing_sensitivity}. On the adversarial problem P5 (see \cref{fig:problem_grid}), with a fixed sample size $n = \num{4000}$, the \EBM and \PRK methods are applied with a varying bucket count. The bucket boundaries are chosen at equally spaced quantiles of the covariate samples. For this problem, there appears to exist no bucketing choice for which the other methods perform as well as S-IDR.

\begin{figure}[h]
    \centering
    \includegraphics[
        width=0.5\textwidth,
        alt={Line plot with the horizontal axis the number of bucket used and the vertical axis the metric. S-IDR is constant and lower than the EBM (minimal at 2-5 buckets) and PRK (minimal at 5-1000 buckets).}
    ]{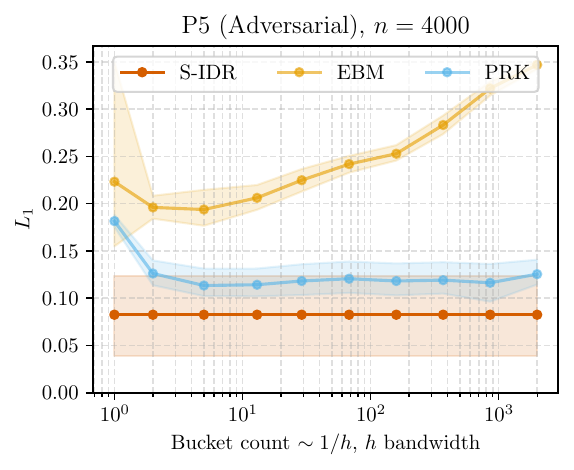}
    \caption{The metric \eqref{eq:metric} averaged over \num{100} random instances of the adversarial problem P5, lower is better. No bucketing granularity/bandwidth allows the \EBM and \PRK methods to reach the S-IDR performance in this problem.}
    \label{fig:bucketing_sensitivity}
\end{figure}

\section{Single index model benchmark details}\label{a:benchmark}
This section describes in detail the benchmark process that drives the results of \cref{ss:index_model_benchmark}. Our interest lies in the effect of using S-IDR as an index model in combination with a base model, \emph{relative} to the base model itself. There is a cost to sample splitting and fitting a nonparametric method, and we aim to understand under which circumstances this trade-off is favourable. All software packages referred to in this section are Python \citep{rossumPythonReferenceManual1995} packages.

A selection of mostly well-known data sets, summarised in \cref{tab:single_index_benchmark_data} and available in the `survdata` package, hosted at \url{https://github.com/vandenheuvel/survival-datasets}, was prepared for the benchmark in a standardised processing pipeline. The source for each data set is provided in the package.

Our goal is to compare performance with and without S-IDR, not to report absolute predictive performance. The pipeline is applied to the full data set before the train/test split for simplicity. This is information leakage, so the absolute test-set numbers are optimistic and not reported. For the relative comparison we report, however, both branches (base model alone, and base model composed with S-IDR) train on the identical sample and are evaluated on the identical test set; the same leakage is present in both branches and cancels in the differential. The following pipeline steps are applied:
\begin{enumerate}
    \item Feature columns with over 25\% missing data are dropped; the remaining missing values are mean-imputed.
    \item Date columns are converted to integer day counts and anchored at the origin, and categorical columns are one-hot encoded.
    \item To improve model fit and address numerical issues, observed times are shifted away from the origin by $1.001$, and if the right tail ratio (see \cref{tab:single_index_benchmark_data}) is above $8.0$, we $\log$-transform.
    \item Individual feature selection consists of a Cox PH model between each column and the target, keeping only the columns with $p$-value below $0.1$ (the implementation from the \texttt{lifelines} package \citep{davidson-pilonLifelinesSurvivalAnalysis2024} is used).
    \item Joint feature selection repeatedly drops columns that have a correlation to another column of at least 0.95, starting with the column that has the most pairwise correlations above that threshold (tie breaking by the mean correlation strength).
\end{enumerate}
\begin{table}
    \centering
    \begin{tabular}{lrrrrrrrr}
        \toprule
        Name & $n$ & $p$ & $p_\text{cat}$ & $p_\text{cat}^\text{enc}$ & $\hat{\pr}(\Delta = 0)$ & $\#\{T\}$ & $\#\{T | \Delta = 1\}$ & Tail ratio \\
        \midrule
        \texttt{aids} & \num{1151} & 11 & 8 & 24 & 0.92 & \num{264} & \num{76} & 0.73 \\
        \texttt{flchain} & \num{7874} & 9 & 5 & 39 & 0.72 & \num{2977} & \num{1738} & 0.39 \\
        \texttt{insurance (d)} & \num{11611} & 5 & 4 & 19 & 0.07 & \num{11} & \num{10} & 2.00 \\
        \texttt{insurance (h)} & \num{11611} & 5 & 4 & 19 & 0.07 & \num{10286} & \num{9714} & 187.68\rlap{*} \\
        \texttt{gbsg2} & \num{686} & 8 & 3 & 7 & 0.56 & \num{574} & \num{270} & 1.24 \\
        \texttt{lossalaefull} & \num{1500} & 2 & 0 & 0 & 0.02 & \num{542} & \num{541} & 14.94\rlap{*} \\
        \texttt{metabric} & \num{1904} & 9 & 0 & 0 & 0.42 & \num{1686} & \num{1011} & 1.44 \\
        \texttt{nhanes} & \num{14264} & 79 & 0 & 0 & 0.67 & \num{441} & \num{426} & 0.54 \\
        \texttt{seer} & \num{4024} & 14 & 0 & 0 & 0.85 & \num{107} & \num{100} & 1.00 \\
        \texttt{support} & \num{8873} & 14 & 0 & 0 & 0.32 & \num{1714} & \num{1027} & 2.38 \\
        \texttt{veterans} & \num{137} & 6 & 3 & 8 & 0.07 & \num{101} & \num{97} & 6.43 \\
        \texttt{whas500} & \num{500} & 14 & 8 & 16 & 0.57 & \num{395} & \num{162} & 1.45 \\
        \bottomrule
    \end{tabular}
    
    \caption{
        Key properties of the \num{12} data sets used in the benchmark: The number of observations before train/test split $n$, the initial dimensionality of the covariate $p$, number of categorical columns $p_\text{cat}$ and into how many one-hot columns they are encoded $p_\text{cat}^\text{enc}$, the share of the observations that is censored $\hat{\pr}(\Delta = 0)$, the number of unique times $\#\{T\}$ and how many of those had an uncensored observation $\#\{T | \Delta = 1\}$, and the right tail ratio (difference between the empirical $99^\text{th}$ percentile and median, divided by the interquartile range), with an * indicating that a $\log$-transform was performed. The \texttt{insurance (d)/(h)} collections are the claim duration in years and claim height of \texttt{freclaimset3fire9207} (subsampled from $n = \num{58056}$ for computational reasons).
    }
    \label{tab:single_index_benchmark_data}
\end{table}
Three base models are used, once standalone and once with S-IDR through sample splitting:
\begin{itemize}
    \item Accelerated Failure Time (AFT) model \citep{buckleyLinearRegressionCensored1979}, implemented in the \texttt{lifelines} package.
    \item Cox PH model \citep{coxRegressionModelsLifeTables1972}, implemented in the \texttt{scikit-survival} package \citep{polsterlScikitsurvivalLibraryTimetoevent2020}.
    \item Random Survival Forests \citep{ishwaranRandomSurvivalForests2008} (with \num{400} trees, or \num{250} trees when sample splitting with S-IDR, with default setting), also implemented in the \texttt{scikit-survival} package.
\end{itemize}
Then, we split each data set 50 times into an 80\% train and 20\% test set, each time fitting each base model and the base model as a single index model combined with S-IDR. The combined fit is done through sample splitting into two subsamples of equal size, repeated 50 times.

For the c-index \citep{harrellEvaluatingYieldMedical1982} and 1-calibration \citep{hosmerGoodnessofFitTestsLogistic1985, andresNovelLearningAlgorithm2018}, an evaluation time $t^\ast$ must be chosen, which we set to be the median of a Kaplan--Meier estimate over all samples. The other metrics --- the integrated Brier score (IBS) \citep{grafAssessmentComparisonPrognostic1999} and D-calibration \citep{haiderEffectiveWaysBuild2020} --- do not require such a choice.

\section{Reproducibility}
The materials used to produce the figures in this paper is available at \url{https://gitlab.math.ethz.ch/bramva/s-idr-experiments}, as well as via \doi{10.5281/zenodo.21161048}. They are R and Python notebooks, with their environments pinned through \texttt{renv} and \texttt{uv} respectively, the test data sets from \cref{ss:index_model_benchmark} are available through a contained python package. The code from \cref{ss:case_study} contains no data inputs, and also outputs are not provided, in accordance with the relevant data agreement.

\end{document}